\documentclass[12pt]{amsart}
\usepackage[russian,english]{babel}
\usepackage{amsfonts,amsmath,amssymb,epsfig,rotate}
\usepackage{color}
\newtheorem{theo}{Theorem}
\newtheorem{prop}{Proposition}
\newtheorem{lemm}{Lemma}
\newtheorem*{keylem}{Key Lemma}
\newtheorem{coro}{Corollary}
\theoremstyle{definition}
\newtheorem{defi}{Definition}
\newtheorem{rem}{Remark}
\newtheorem{agree}{Convention}
\def\hm{}
\def\tb{\qopname\relax o{tb}}
\def\lk{\qopname\relax o{lk}}
\def\sl{\qopname\relax o{sl}}
\def\exp{\qopname\relax o{c}}
\def\page{\mathcal P}
\def\nesw{\kern-1pt\rotatebox{45}{$\leftrightarrow$}\kern-2pt}
\def\nwse{\kern-2pt{}^{\rotatebox{315}{$\leftrightarrow$}}\kern-2pt}
\def\bm(#1){[#1]_{\mathrm{BM}}}
\def\trans(#1){[#1]_{\mathrm{t}}}
\def\BM{\mathrm{BM}}

\def\Iright{$\overrightarrow{\mathrm I}$}
\def\Ileft{$\overleftarrow{\mathrm I}$}
\def\IIright{$\overrightarrow{\mathrm{II}}$}
\def\IIleft{$\overleftarrow{\mathrm{II}}$}
\def\overline#1{#1^\curvearrowright}

\author{Ivan Dynnikov and Maxim Prasolov}
\thanks{The work is supported in part by Russian Foundation for Basic Research (grant no.~10-01-91056-\textcyrillic{\CYRN\CYRC\CYRN\CYRI\_\cyra})
and the Russian Government (grant no.~2010-220-01-077)}

\title{Bypasses for rectangular diagrams. Proof of Jones' conjecture and related questions}

\begin{document}

\maketitle
\markright{\sc Bypasses for rectangular diagrams. Proof of Jones' conjecture and related questions}

\begin{abstract}
In the present paper a criteria for a rectangular diagram to admit a simplification is given
in terms of Legendrian knots. It is shown that there are two types of simplifications
which are mutually independent in a sense.
It is shown that a minimal rectangular diagram maximizes the
Thurston--Bennequin number for the corresponding Legendrian links.
Jones' conjecture about the invariance of the algebraic number of intersections of a minimal
braid representing a fixed link type is proved. A new proof of
the monotonic simplification theorem for the unknot is given.
\end{abstract}

\tableofcontents

\section*{Introduction}
As shown in~\cite{Dyn} any rectangular diagram of the unknot admits a monotonic simplification
to the trivial diagram. The starting point for the present work was the question when a rectangular
diagram of an arbitrary link admits a simplification, i.e. decreasing the complexity at least by one.

By complexity of a rectangular diagram we mean one half of the number of its edges.
There is an analogue of Reidemeister moves, which we call elementary moves, that
allow to transform into each other any two rectangular diagrams representing equivalent links.
Some of them change the complexity of the diagram. Complexity increasing
elementary moves are called stabilizations, their inverses are called destabilizations.
By monotonic simplification we mean a sequence of elementary moves that does
not include stabilizations.

Legendrian knots came to interest in the beginning of 1980s due to
the development of contact topology~\cite{Ben}.
Each class of topologically equivalent links contains infinitely many Legendrian classes
of Legendrian links. There is an integral invariant of Legendrian isotopy
called the Thurston--Bennequin number,
which can be made arbitrarily small for any topological link type.
One defines operations called stabilizations (respectively, destabilizations) also for Legendrian links.
Those operations preserve the topological type of the link
and drop (respectively, increase) the Thurston--Bennequin number by one.

In 2003 W.Menasco drew first author's attention to a connection between rectangular
diagrams and Legendrian links. This connection can be described as follows (see~\cite{MM,NgTh}).

With any rectangular diagram $R$ one associates naturally
a Legendrian link $L_R$ whose Legendrian type does not change under
complexity preserving elementary moves of the diagram~$R$.
Stabilizations and destabilizations split into two types, those of the first
type preserve the Legendrian type of $L_R$, and those of the second one
preserve the Legendrian type of $L_{\overline R}$, where $\overline R$ stays for
the diagram obtained from~$R$ by clockwise rotation by $\pi/2$ (for the topological type of the link this corresponds
to taking the mirror image).
The change of the link $L_R$ (respectively, $L_{\overline R})$ under second (respectively, first) type
(de)stabilization is a Legendrian (de)stabilization.

It turns out that the question about simplifiability of a rectangular diagram
can be answered in a very natural way in terms
of the corresponding Legendrian links. Namely, any Legendrian
destabilization $L_R\mapsto L'$ can be realized by a monotonic simplification of
the diagram~$R$ while keeping unchanged the Legendrian type of the link~$L_{\overline R}$, 
and the existence of a Legendrian destabilization for~$L_R$ or $L_{\overline R}$ is a criteria for
simplifiability of~$R$.

We introduce below a notion of a bypass, which plays a key r\^ole in the proof of this and other statements.
We define it in the language of rectangular diagrams as well as in that of Legendrian links.
This notion is somehow related, though does not coincide, with an object called bypass in papers~\cite{Etn,Hon}.

The idea behind bypass is to represent any Legendrian destabilization as a replacement of an arc
contained in the link by another arc, which is the one that we call bypass.
After translating to the combinatorial language, such a replacement is no longer a destabilization, it
does not even drops the complexity in general, but a rectangular diagram equipped with
a bypass can be transformed by elementary moves so as to preserve the complexity of the diagram and
shorten the bypass. This is the main technical result of this paper.

In order to prove it we use the technique that was developed by J.Birman and W.Menasco
in a series of work where they studied links represented by closed braids~\cite{BM1,BM2}. 
Some important elements of this technique appeared earlier in paper~\cite{Ben} by D.Bennequin.
P.Cromwell in~\cite{Cro} noticed that Birman--Menasco's method can be extended
to arc presentations of links, which, from combinatorial point of view, are just another instance of rectangular diagrams.

The proof of the main result of~\cite{Dyn} about monotonic simplification for the unknot
was obtained by further development of this method. That proof can be simplified
by using the connection to Legendrian links that was mentioned above.
We show below how to reduce the monotonic simplification theorem
to the theorem of Eliashberg and Fraser~\cite{EF,EF2} on classification of Legendrian
unknots or to Erlandsson's Theorem~\cite{erl} on negativity of
the Thurston--Bennequin number of the unknot.

The above mentioned simplifiability criteria for rectangular diagrams
has a number of corollaries that have been conjectured earlier.
E.g., we prove below (in a generalized form) Jones' conjecture that states that two braids
of minimal braid index whose closures are equivalent oriented links
have the same algebraic number or crossings.

\begin{agree}
Throughout this paper except Section~\ref{braidsection},
by a link we mean a link in $\mathbb R^3$ (i.e. compact one-dimensional submanifold)
equipped with a selective orientation and coloring of connected components.
This means that every component of the link may or may not be oriented
and given a color, which is an positive integer. In particular, a link may be non-oriented, and colors may be
all ones, or, on the contrary, a link may be oriented, and all components numbered by different integers.

When speaking about an isotopy between links we assume that the orientations and the colors of components are
preserved under the isotopy. This additional structure, selective orientation and coloring,
is extended in an obvious way to all link diagrams that we consider below (rectangular diagrams
and front projections), and we don't mention this explicitly for the purpose of clarity.

The exception from these rules in Section~\ref{braidsection} concerns only orientations
as all links there are assumed to be oriented.

The coloring plays no role in the paper, we just point out
that our statements remain true for colored links.
\end{agree}

\subsection*{Acknowledgements.}
We are deeply indebted to Sergei Melikhov who pointed
out a number of inaccuracies in the original version of this work.

\section{Rectangular diagrams}
\begin{defi}
\emph{A rectangular diagram of a link} is 
a finite union of closed broken lines in the plane that consist of
horizontal and vertical straight line segments (called \emph{edges} of the diagram)
no two of which are collinear. Such a diagram is interpreted as a planar link
diagram in which vertical edges at all crossings are overpasses.

The endpoints of edges are called \emph{vertices} of the diagram.
The number of vertical edges of a rectangular diagram~$R$ is called \emph{the complexity of $R$}
and denoted by~$c(R)$.
\end{defi}

An example of a rectangular diagram is shown in Fig.~\ref{fig8knot}.

\begin{figure}[ht]
\center{\includegraphics[width=0.3\textwidth]{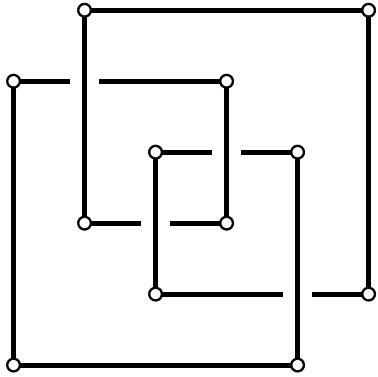}}
\caption{A rectangular diagram of the figure eight knot}\label{fig8knot}
\end{figure}

As shown in~\cite{Cro,Dyn} two rectangular diagrams represent equivalent links if and only if
they are related by a finite sequence of the following \emph{elementary moves}:
\def\labelenumi{(\roman{enumi})}
\begin{enumerate}
\item
\emph{a cyclic permutation} of vertical or horizontal edges consists in moving
one of the extreme (upper, lower, left, or right) edges onto the opposite side
with simultaneous adjustment of the two adjacent edges, see Fig.~\ref{cycle};
\begin{figure}[ht]
\center{\includegraphics[scale=1.3]{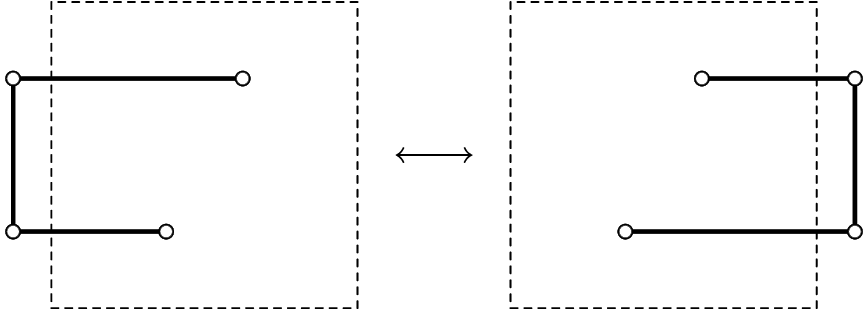}}
\caption{Cyclic permutation}\label{cycle}
\end{figure}
\item
\emph{a commutation} of vertical (respectively, horizontal) edges consists in exchange of
the horizontal (respectively, vertical) positions of two neighboring vertical
(respectively, horizontal) edges provided that the pairs of endpoints of
those edges projected to the vertical (respectively, horizontal) axis are disjoint and do not interleave. The edges
are regarded neighboring if there are no vertices of the diagram between
the straight lines containing them, see Fig.~\ref{castle};
\begin{figure}
\center{\includegraphics[scale=1.2,angle=-90]{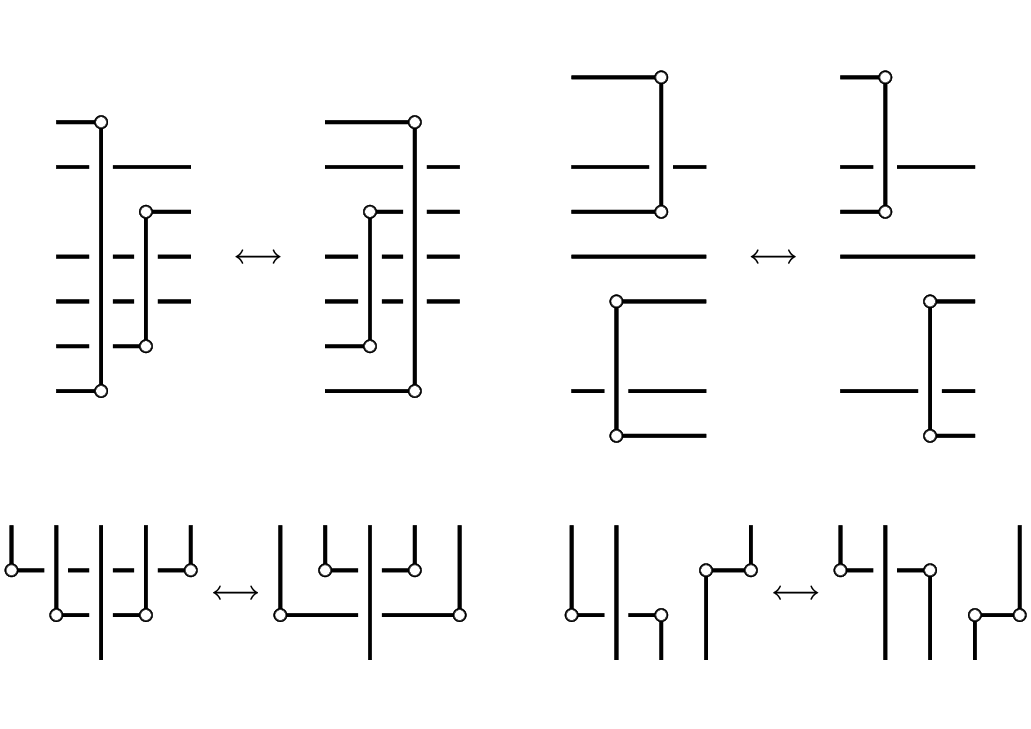}}
\caption{Commutations}\label{castle}
\end{figure}
\item
\emph{a stabilization} consists in replacement of a vertex
by three new ones that together with the deleted one
form the vertex set of a small square,
addition two short edges that are sides of the square,
and appropriate extension or shortening the edges that approached
the deleted vertex;
the inverse operation is called \emph{a destabilization} (see Fig.~\ref{stab}).
\begin{figure}[ht]
\center{\includegraphics{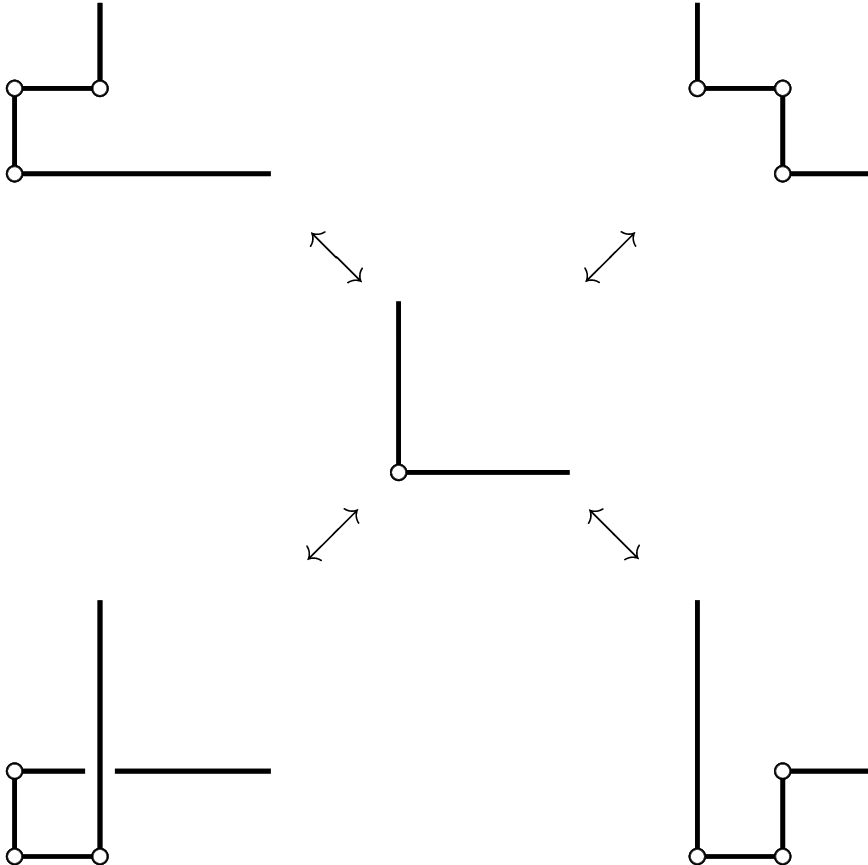}}
\caption{Stabilizations and destabilizations(``long'' edges can be directed
in other ways, which gives 12 more similar pictures).}\label{stab}
\end{figure}
\end{enumerate}

We will distinguish two types of stabilizations and destabilizations.

\begin{defi}
If the two edges emerging from the stabilization point from their common end downward and leftward or upward and rightward,
then the stabilization is of \emph{type I},
\begin{figure}[ht]
\center{\includegraphics{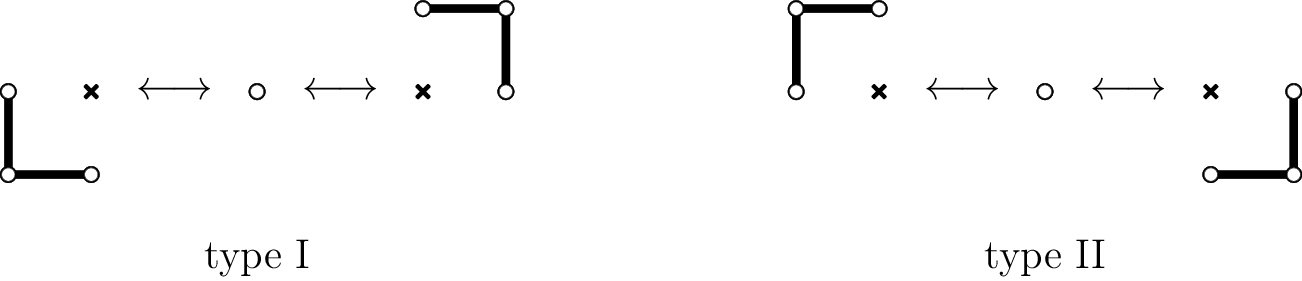}
}
\caption{Types of stabilizations and destabilizations}\label{stabilizationtypes}
\end{figure}
and otherwise of \emph{type II}, see Fig.~\ref{stabilizationtypes}.
In this definition, the direction of the edges attached to the
vertex at which the stabilization occurs does not matter.

\emph{The type of a destabilization} is defined as that of the inverse stabilization.

By \emph{elementary simplification} of a rectangular diagram we mean any sequence
of elementary moves in which the last move is a destabilization and all
the preceding ones are cyclic permutations and commutations. \emph{The type
of an elementary simplification} is defined by that of the final destabilization.
\end{defi}

The core result of the present paper is formulated in combinatorial terms as follows.

\begin{theo}\label{T2types}
Let a rectangular diagram $R$ admit $k$ successive elementary type~I simplifications
$R\hm\mapsto R_1'\mapsto R_2'\mapsto\ldots\mapsto R_k'$,
as well as $\ell$ successive elementary type~II simplifications $R\mapsto R_1''\mapsto\ldots\mapsto R_\ell''$.

Then the diagram $R_k'$ admits $\ell$ successive elementary type~II simplifications,
and the resulting diagram can be connected with $R_\ell''$ by a sequence
of cyclic permutations, commutations, and type~I stabilizations/destabilizations.

Similarly, the diagram $R_\ell''$ admits $k$ successive elementary type~I simplifications,
and the resulting diagram can be connected with $R_k'$
by a sequence of cyclic permutations, commutations, and type~II stabilizations/destabilizations.
\end{theo}

This statement will be proved in Sections \ref{bypasssection} and \ref{proofofkeylemma}.

\section{Legendrian links}
\subsection{Front projections}
\begin{defi}
A link $L$ in $\mathbb R^3$ is called \emph{Legendrian} if it has a form
of a smooth curve everywhere tangent to
the plane distribution defined the 1-form
$$\omega=x\,dy+dz$$
and called \emph{the standard contact structure.}
Legendrian links are considered equivalent if they are isotopic
within the class of Legendrian links.
\end{defi}

The most convenient way to specify Legendrian links is to draw their
projections to the $yz$-plane, which are called \emph{front projections}.
A front projections forms a piecewise smooth curve whose singularities have the form of a cusp
and whose tangents are never vertical (see Fig.~\ref{front}).
\begin{figure}[ht]
\center{\includegraphics[scale=0.3,angle=270]{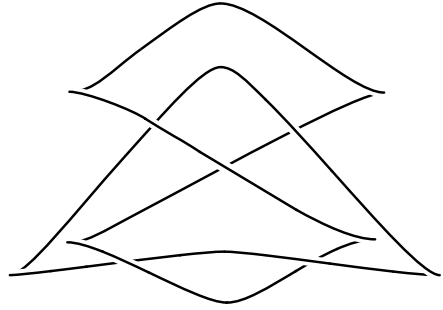}}
\caption{Front projections of a Legendrian link}\label{front}
\end{figure}
For reader's convenience we display at self-intersections
and cusps of front projections which branch is overcrossing (i.e. have
larger $x$-coordinate), and which is undercrossing.
A generic front projection uniquely defines the corresponding Legendrian curve in $\mathbb R^3$
since the $x$-coordinate of any point of the curve can be recovered from the relation $x=-dz/dy$.

Fig.~\ref{frontmoves} demonstrates schematically (up to central symmetry)
which front projections moves
\begin{figure}[ht]
\center{\includegraphics{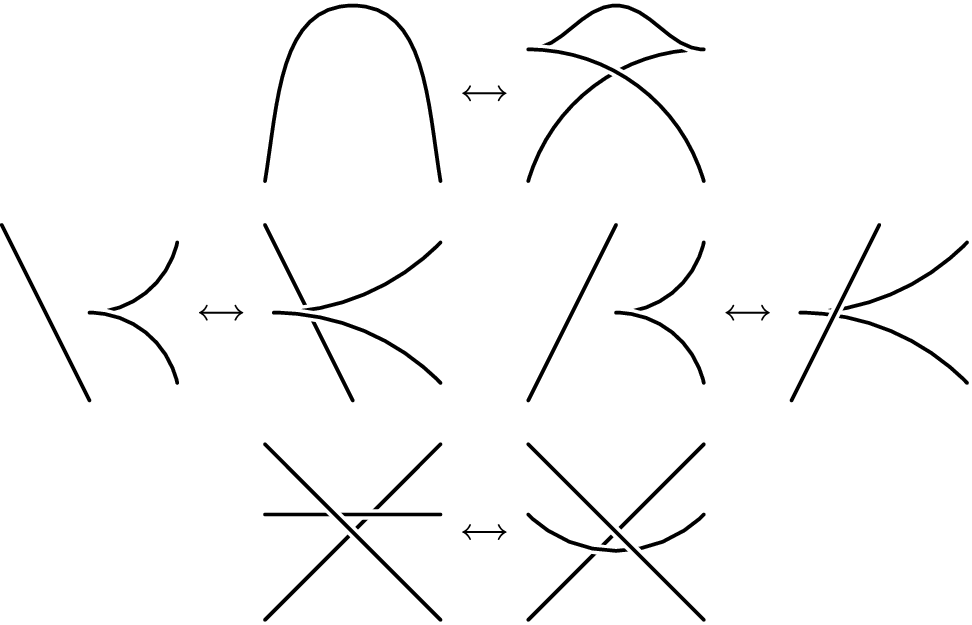}}
\caption{Front projection moves}\label{frontmoves}
\end{figure}
generate the equivalence of Legendrian links, see~\cite{Swi}.

Some front projection moves that preserve the topological
type of the link are forbidden for Legendrian links because
they change their Legendrian type. Examples of such
moves are shown in Fig.~\ref{forbidmoves}.
\begin{figure}[ht]
\center{\includegraphics{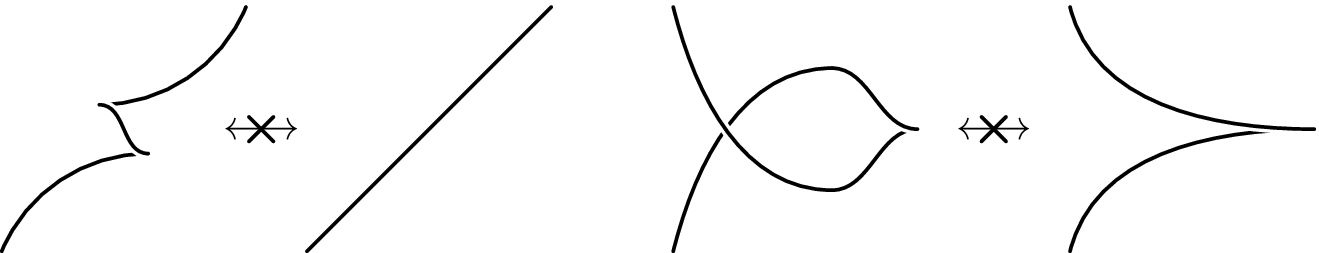}}
\caption{Forbidden front projection moves}\label{forbidmoves}
\end{figure}

A piecewise smooth link consisting of finitely many arcs that are everywhere tangent to
the standard contact structure can be smoothen at the break points so that
the combinatoric structure of the projection to the $yz$-plane (including information
about cusps) does not change.
Therefore, such piecewise smooth links also uniquely define some Legendrian type,
and they themselves can be spoken of as Legendrian links.

\subsection{Presenting by rectangular diagrams}
The connection between Legendrian links and rectangular diagrams is described as follows.

Let $R$ be a rectangular diagram. Rotate it by $\pi/4$ counterclockwise,
smooth out the corners pointing up and down and turn into cusps
corners pointing to the left and to the right.
An example is shown in Fig.~\ref{r->l}.
\begin{figure}[ht]
\center{\includegraphics{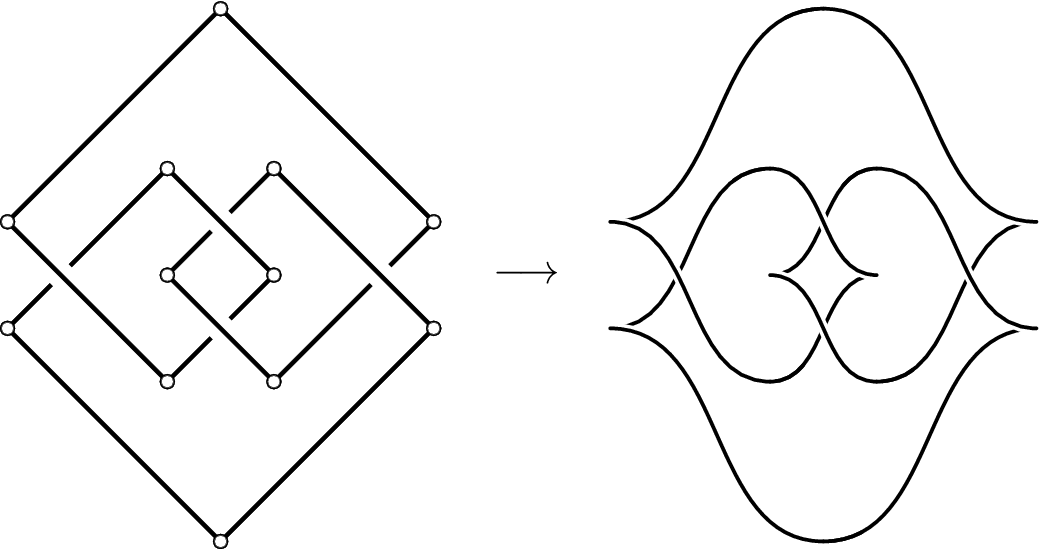}}
\caption{Front projection of a Legendrian link obtained from
a rectangular diagram}\label{r->l}
\end{figure}

The Legendrian link defined by the obtained front projection will be denoted by $L_R$.

\begin{theo}[\cite{OST}, Section~4]\label{L_R}
Any Legendrian link is equivalent to the link $L_R$ for some $R$.

Legendrian links $L_R$ and $L_{R'}$ are Legendrian equivalent if and only if
the rectangular diagrams~$R$ and $R'$ can be connected by a sequence of
elementary moves that does not include type~II stabilizations and destabilizations.
\end{theo}

By \emph{Legendrian type} of a rectangular diagram we mean that of
the corresponding Legendrian link~$L_R$. It follows from Theorem~\ref{L_R} 
that the Legendrian type of a rectangular diagram is its equivalence class with
respect to cyclic permutations, commutations, and type~I stabilizations and destabilizations.

\subsection{Thurston--Bennequin number}
Notice that the vector $\mathbf e_z=(0,0,1)$ is nowhere tangent to the contact structure,
therefore, a small enough shift of any Legendrian link~$L$ along this vector produces
a link $L^+$ that is disjoint from $L$. A link obtained by a small shift
in the opposite direction will be denoted by~$L^-$.

\emph{The Thurston--Bennequin number} of a Legendrian link $L$, denoted by $\tb(L)$, 
is the linking number $\lk(L,L^+)=\lk(L,L^-)$. This definition assumes that the link $L$
is oriented. Reversing the orientation of all components does not affect $\tb(L)$.
So, in particular, if $L$ is a knot its Thurston--Bennequin number does not depend
on the orientation.

The Thurston--Bennequin number is a Legendrian isotopy invariant,
and it can be computed from the following simple formula
\begin{equation}\label{bennequinformula}
\tb(L)=w(F)-\frac12c(F),
\end{equation}
where $w(F)$ is the writhe number and $c(F)$ is the number of cusps of
a front projection $F$ representing the link~$L$.

The rectangular diagram obtained from a rectangular diagram~$R$ by a clockwise $\pi/2$-rotation
(with simultaneous flipping all crossings) will be denoted by $\overline R$.

Note that if $R_1\mapsto R_2$~--- is a type~I stabilization,
then $\overline R_1\mapsto\overline R_2$~--- is a type~II stabilization and vice versa.

For the front projections~$F$ and $\overline F$ corresponding
to rectangular diagrams $R$ and $\overline R$ we have $w(F)\hm=-w(\overline F)$, 
whereas the total number of their cusps is equal to the number of vertices of~$R$.
This implies the following relation:
\begin{equation}\label{complexityandtb}
-c(R)=\tb(L_R)+\tb(L_{\overline R}).
\end{equation}

From~\eqref{bennequinformula} one see that the forbidden moves shown in Fig.~\ref{forbidmoves} change
the Thurston--Bennequin number by one. These moves of front projections of Legendrian links
will also be called \emph{stabilizations} and \emph{destabilizations}, where the former drop and the latter
increase the Thurston--Bennequin number.

Stabilizations, in contrast with destabilizations, can be applied to any front projection.
If the Legendrian class of a Legendrian link~$L$ contains a representative
whose front projection admits a destabilization, then we say that
so does the Legendrian link~$L$.

Clearly, type~II stabilizations and destabilizations of a rectangular diagram $R$ result
in stabilizations and destabilizations, respectively, of the corresponding Legendrian link~$L_R$.

The following result is due to T.Erlandsson~\cite{erl} (see also Theorem~2 in~\cite{Ben}).
\begin{theo}\label{erltheo}
The Thuston--Bennequin number of a Legendrian unknot
is always negative.
\end{theo}

Y.Eliashberg and M.Fraser proved a stronger result in~\cite{EF,EF2} a particular
case of which can be formulated as follows.

\begin{theo}\label{EFtheo}
Let $K$ be an oriented Legendrian knot in $\mathbb R^3$ having topological type of the unknot.
Then it is Legendrian equivalent to a knot from the following list

\centerline{\includegraphics{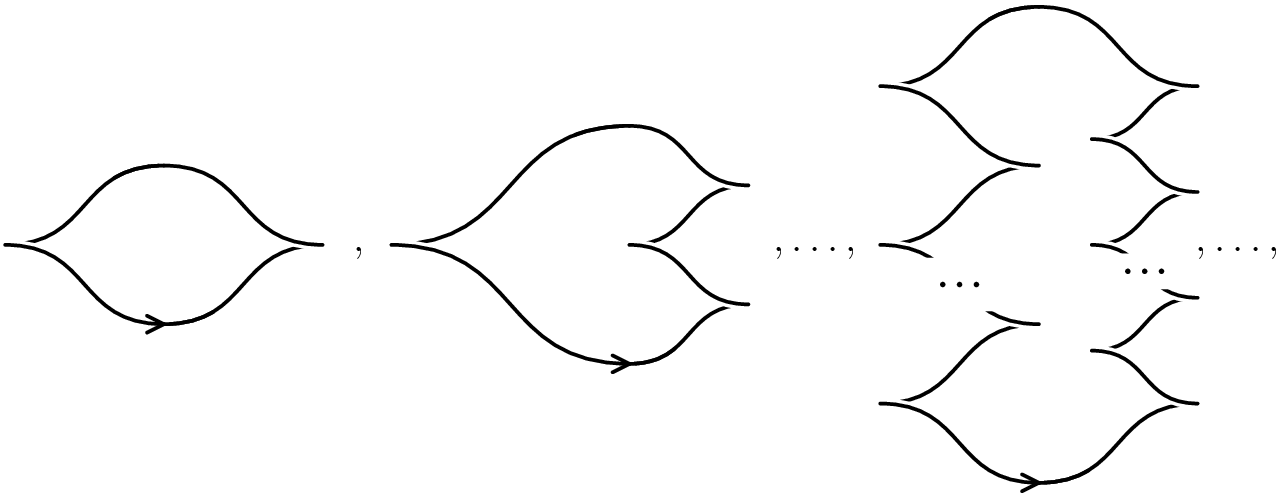}}

\noindent
in which the number of left and right ``zigzags'' run independently
the set of non-negative integers.
\end{theo}

\section{Bypasses}\label{bypasssection}
For settling the main results of this paper we will need only a combinatorial model of bypasses.
However, for making the definition look less artificial we start by describing the
geometrical object underlying the concept.

\subsection{Legendrian description}
\begin{defi}
Let $L$ be a Legendrian link. An ordered pair $(\alpha,\beta)$ in which $\alpha\subset\mathbb R^3$ is a smooth simple Legendrian
(i.e. everywhere tangent to the standard contact structure) arc
with endpoints at $L$, and $\beta\subset L$ is an arc with the same endpoints
is called \emph{a~bypass for} $L$ if
there exists a two-dimensional disc $D\subset\mathbb R^3$ satisfying the following:
\begin{enumerate}\setcounter{enumi}{-1}
\def\labelenumi{(A\theenumi)}
\item
$D$ is the image of the semidisc $\{(x,y)\in\mathbb R^2\,;\,x^2+y^2\leqslant1,\ x\hm\leqslant0\}$
under a smooth embedding into $\mathbb R^3$;
\item
the disc boundary $\partial D$ coincides with $\alpha\cup\beta$;
\item
the intersection $D\cap L$ coincides with $\beta$;
\item
disc $D$ is tangent to the standard contact structure along $\alpha$.
\end{enumerate}
\end{defi}

Let us discuss in more detail Condition~(A3) as it will be crucial for us to replace it by
another condition that can be reformulated in combinatorial terms. Assume that we have a pair~$(\alpha,\beta)$
and a disc~$D$ that satisfy~(A0--A2).

Since $\alpha$ and $\beta$ are Legendrian, the disc $D$ is tangent to the contact structure at
the endpoints $\partial\alpha=\partial\beta$, but not necessarily along the whole arc~$\alpha$.
So, when one proceeds from one endpoint of $\alpha$ to the other
the tangent plane to $D$ makes a half-integer algebraic number of turns around
the tangent vector to~$\alpha$ with
respect to the standard contact structure. Denote this number by~$r$. If we have~$r=0$,
then $D$ can be deformed near~$\alpha$ so as to satisfy Condition~(A3), hence $(\alpha,\beta)$ is a bypass.

Denote the link $(L\setminus\beta)\cup\alpha$ by $L_{\beta\rightarrow\alpha}$.
One can see that the pairs of linking numbers
$$\{\lk(L_{\beta\rightarrow\alpha},(\alpha\cup\beta)^+),\,
\lk(L_{\beta\rightarrow\alpha},(\alpha\cup\beta)^-)\}\text{ and }
\{\lk(L_{\beta\rightarrow\alpha}^+,\alpha\cup\beta),\,\lk(L_{\beta\rightarrow\alpha}^-,\alpha\cup\beta)\}$$
for an appropriate choice of orientations coincide with $\{[r],[r+1/2]\}$, where $[x]$ denotes the integral part of~$x$. If these numbers
are all zero, then the knot $\alpha\cup\beta$ shifted in either way is actually
unlinked with $L_{\beta\rightarrow\alpha}$, i.e.\ separated from~$L_{\beta\rightarrow\alpha}$ by
an embedded two-sphere.
So, Condition~(A3)
in the definition of bypass can be replaced by either of the following three:
\begin{enumerate}
\item[(A$3')\phantom{''}$] $L_{\beta\rightarrow\alpha}$ is unlinked with $(\alpha\cup\beta)^+\cup(\alpha\cup\beta)^-$;
\item[(A$3'')\phantom{'}$]
$L_{\beta\rightarrow\alpha}^+\cup L_{\beta\rightarrow\alpha}^-$ is unlinked with $\alpha\cup\beta$;
\item[(A$3'''$)] we have
$\lk(L_{\beta\rightarrow\alpha},(\alpha\cup\beta)^+)=\lk(L_{\beta\rightarrow\alpha},(\alpha\cup\beta)^-)=0$.
\end{enumerate}

Note also that Condition~(A0) is then unnecessary.

With Condition~(A3) replaced by~$(\text{A3}')$, $(\text{A3}'')$ or $(\text{A3}''')$
it is convenient to use front projections to check if an arc is a bypass. E.g., it can be seen
from Fig.~\ref{fronatlbypass} that the dashed line on the left is not a bypass, and on the right is.

\begin{figure}[ht]
\center{\begin{picture}(232,263)
\put(0,232){\includegraphics[scale=1,angle=-90]{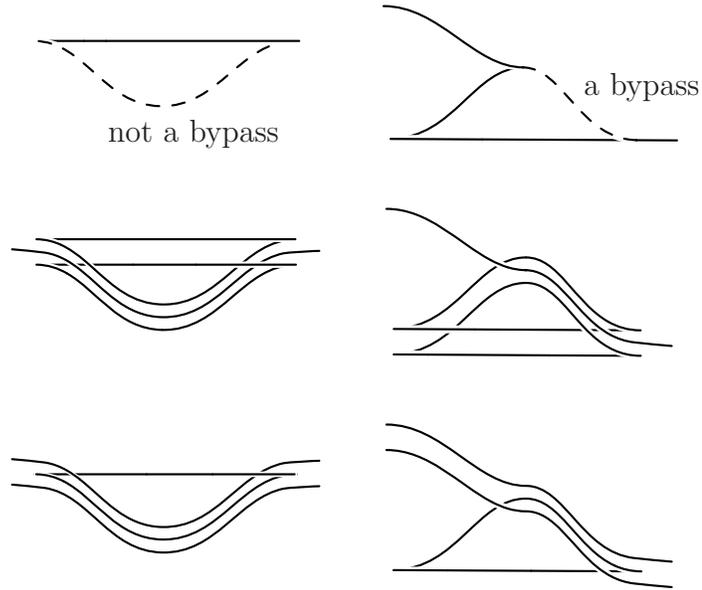}}
\put(40,172){not a bypass}
\put(220,190){a bypass}
\end{picture}}
\caption{Detecting a bypass on a front projection}\label{fronatlbypass}
\end{figure}

From Condition~(A$3'''$) and Theorem~\ref{erltheo} it follows
that the link $L_{\beta\rightarrow\alpha}$ is topologically equivalent to $L$ but has larger
Thurston--Bennequin number, namely, we have
$$\tb(L_{\beta\rightarrow\alpha})=\tb(L)-\tb(\alpha\cup\beta).$$
As we will see below (in combinatorial terms) the link $L_{\beta\rightarrow\alpha}$ can be obtained from $L$ by a sequence of
$(-\tb(\alpha\cup\beta))$ destabilizations.

\subsection{Description in terms of rectangular diagrams}\label{rectdiagrdescr}
Now we translate what was just said to the combinatorial language.
For ease of exposition we introduce the following terminology convention.

\begin{agree}
Since a rectangular diagram can be recovered from its set of vertices,
in what follows we do not distinguish this set from the diagram itself, and
use the same notation for both. Thus a finite subset of the plane
will be considered to be a rectangular diagram if any vertical or horizontal straight line
contains either no or exactly two points from the subset.
\end{agree}

\begin{defi}
By \emph{the Thurston--Bennequin number $\tb(R)$ of a rectangular diagram~$R$ of a knot} we call
the Thurston--Bennequin number of the corresponding Legendrian knot:
$$\tb(R)=\tb(L_R).$$
\end{defi}

For a rectangular diagram $R$ we denote by $R^\nearrow$ a diagram obtained from~$R$
by a small positive shift along the vector $(1,1)$, where smallness of the shift
will be clear from the context. Similarly, $R^\swarrow$ will denote the result
of a small shift in the opposite direction.

If $R$ is a rectangular diagram of a knot, then $\tb(R)$ is by definition
equal to the linking number $\lk(R,R^\nearrow)$ provided that
the respective connected components of the link presented by $R\cup R^\nearrow$ are oriented coherently (instead of $\lk(R,R^\nearrow)$
one can equally well take $\lk(R,R^\swarrow)$ or $\lk(R^\swarrow,R^\nearrow)$).

\begin{defi}
By \emph{a rectangular path} we mean a subset of the plane that can be turned
into a rectangular diagram of a knot by adding exactly one or exactly two points
located at one vertical or horizontal straight line. Additionally,
a single point is also regarded as a rectangular path.

Horizontal and vertical straight line segments connecting two points of a rectangular path
are called \emph{edges}, and the points themselves \emph{vertices} of the path.
Horizontal and vertical straight lines that contain exactly one vertex
of a rectangular path will be called \emph{the ends}, and the corresponding vertices \emph{the endpoints}. 
All other vertices of a rectangular path are called \emph{internal}.
\end{defi}

Examples of rectangular paths are shown in Fig.~\ref{pathexamples}, where edges and the ends
of the paths are also depicted for clarity.
\begin{figure}[ht]
\center{\includegraphics[scale=0.3,angle=-90]{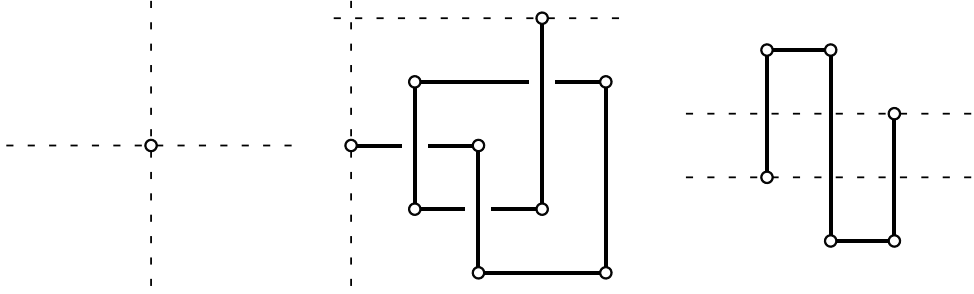}}
\caption{Examples of rectangular paths. Solid lines show edges, and
dashed lines the ends of the paths}\label{pathexamples}
\end{figure}

\begin{agree}\label{convent3}
Whenever we discuss a union of a few rectangular diagrams of links and rectangular paths
we assume that they are in general position, which means that no two vertices
lie on the same horizontal or vertical straight line unless they must do so by
construction.
\end{agree}

E.g., having set this convention, the following is true: any two different rectangular paths
with common ends form a rectangular diagram of a knot.

The fact that an edge of a rectangular diagram or a rectangular
path is contained in an end of a rectangular path will be expressed
in the opposite (but linguistically more habitual) way saying
that the end \emph{lies} on the edge.
This does not mean in general that the endpoint of the path is contained in the edge,
it can lie on its extension, see Fig.~\ref{endonedge}.
\begin{figure}[ht]
\centerline{\includegraphics{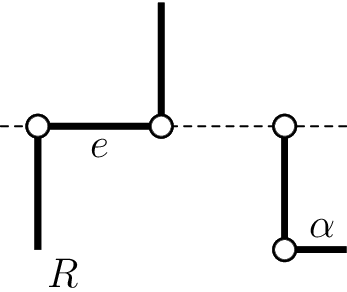}}
\caption{An end of the path $\alpha$ lie on the edge $e$ of $R$}\label{endonedge}
\end{figure}
Our terminology is also justified by
the following construction.

For a finite subset $X\subset\mathbb R^2$, we define $\widetilde X$ to be
the following union of straight line segments in $\mathbb R^3$:
$$\widetilde X=\bigcup_{(i,j)\in X}[(2i,0,1),(0,2j,-1)].$$
$\widetilde X$ can also be characterized as the union of all straight line segments that have
endpoints at the straight lines $\ell_1=\mathbb R\times\{0\}\times\{1\}$
and $\ell_2=\{0\}\times\mathbb R\times\{-1\}$ and pierce the plane
$\mathbb R^2\times\{0\}$ at points from~$X$, see~Fig.~\ref{tildeX}.

\begin{figure}[ht]
\centerline{\begin{picture}(200,168)\put(0,0){%
\includegraphics[width=200pt]{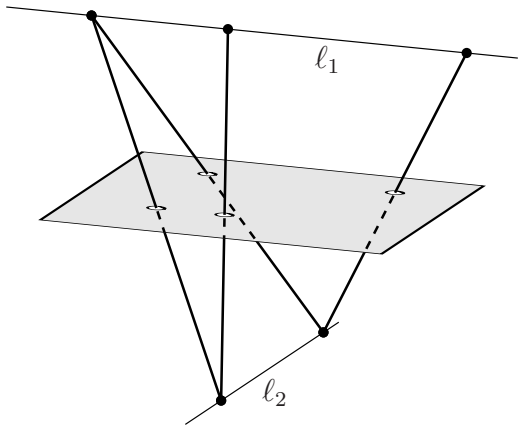}}
\put(100,14){$\ell_2$}
\put(120,140){$\ell_1$}
\end{picture}}
\caption{A rectangular path $X$ and its 3d-realization $\widetilde X$ by a broken line}\label{tildeX}
\end{figure}

One can see that if $R$ is a rectangular diagram, then $\widetilde R$ is a link in $\mathbb R^3$
isotopic to the one defined by~$R$. Indeed, the intersection of $\widetilde R$ with the upper
half-space~$\mathbb R^2\times[0,+\infty)$ consists of all PL-arcs of the form
$[(i,j_1,0),(2i,0,1)]\cup[(2i,0,1),(i,j_2,0)]$ with
$[(i,j_1),(i,j_2)]$ a vertical edge of~$R$. This family of arcs is isotopic in
$\mathbb R^2\times[0,+\infty)$ to the union of vertical edges of~$R$, with the arc endpoints
being fixed during the isotopy.

Similarly, the intersection of~$\widetilde R$ with the lower half-space~$\mathbb R^2\times(-\infty,0]$
consists of PL-arcs isotopic in this half-space to the union of horizontal edges of~$R$.

If $\alpha$ is a rectangular path whose
ends lie at edges of a rectangular diagram~$R$, then $\widetilde\alpha$ is a broken line whose
endpoints lie at~$\widetilde R$. Moreover, with Convention~\ref{convent3} set
the interior of the PL-arc $\widetilde\alpha$ will be disjoint from~$\widetilde R$.
If two rectangular paths~$\alpha$ and~$\beta$ have common ends, then
the broken lines~$\widetilde\alpha$ and~$\widetilde\beta$ have common endpoints.

We are ready to define our key object.

\begin{defi}
By \emph{a bypass for a rectangular diagram} $R$ we call an ordered pair $(\alpha,\beta)$ of
rectangular paths having common ends such that $\beta$ is a subset of~$R$, and there
exists an embedded two-dimensional disc~$D\subset\mathbb R^3$ satisfying the following:
\def\labelenumi{(B\theenumi)}
\begin{enumerate}
\item
the disc boundary $\partial D$ coincides with $\widetilde\alpha\cup\widetilde\beta$;
\item
the intersection $D\cap\widetilde R$ coincides with~$\widetilde\beta$;
\item
in the link defined by the rectangular diagram $(R\setminus\beta)\cup\alpha\cup(\alpha\cup\beta)^\nearrow\cup(\alpha\cup\beta)^\swarrow$,
the components presented by $(R\setminus\beta)\cup\alpha$ are unlinked with the two others.
\end{enumerate}
A bypass $(\alpha,\beta)$ is called \emph{elementary} if we have $\tb(\alpha\cup\beta)=-1$.
The value $-\tb(\alpha\cup\beta)$ will be called \emph{the weight} of the bypass $\alpha$,
and the path~$\beta$ \emph{the bypassed path}.
\end{defi}

In most cases the path $\beta$ is uniquely determined by $\alpha$. E.g., this is true
if the component of~$\widetilde R$ that contains the endpoints of $\widetilde\alpha$
is knotted or linked with the rest of the link.
Anyway, if $\alpha$ is given, then there are at most two options
for $\beta$. So we will often refer to $\alpha$ as a bypass,
keeping, however, in mind that the choice of $\beta$ is made if there are two options.

One can see that any bypass for a rectangular diagram $R$ can be turned into a bypass
for the Legendrian link $L_R$ by using similar procedure to the one that turns
a rectangular diagram into a front projection. Namely, we add the edges of the path
(if an endpoint of the path is collinear with an edge of~$R$ but does not belong to it
we also connect it with the closest endpoint of the edge)
thus obtaining a projection of a knotted graph isotopic to~$\widetilde R\cup\widetilde\alpha$
(at self-intersections vertical edges are always overpasses).
Then we rotate the whole picture
counterclockwise, smooth out some corners and turn into cusps the others
in accordance with the general rules, see Fig.~\ref{bypassgrid2front}.

\begin{figure}[ht]
\center{\includegraphics[scale=1.5]{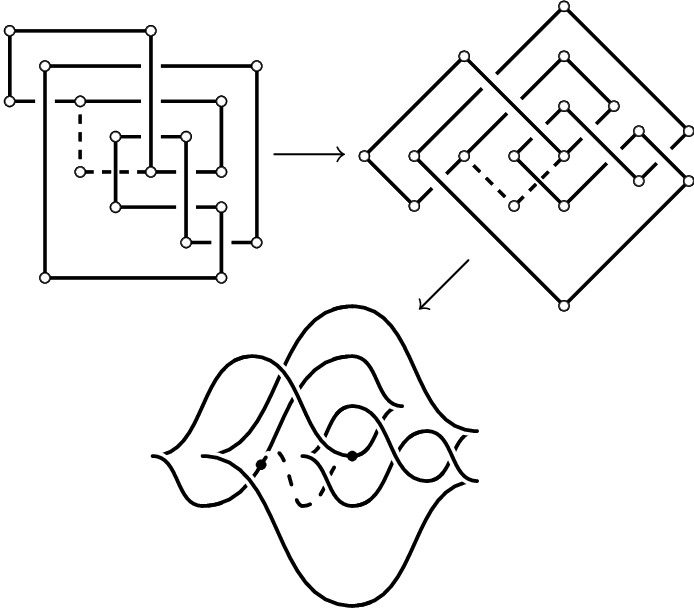}}
\caption{A bypass for a rectangular diagram and the corresponding bypass
on the front projection}\label{bypassgrid2front}
\end{figure}

One can also show that any bypass for a Legendrian link $L_R$ considered up to a Legendrian isotopy
can be obtained in this way. We will not use these facts in the proofs and leave
them as an exercise to the reader.

\subsection{$\Theta$-diagrams}
Here we draw our attention to the object formed by a rectangular diagram equipped by a bypass.

\begin{defi}
By \emph{a rectangular $\Theta$-diagram} we call a union $\alpha\cup\beta\cup\gamma\cup\delta$
of three rectangular paths $\alpha,\beta,\gamma$ having common pairs of ends and a rectangular diagram~$\delta$ of a link.
The paths are distinguished, so, formally, a rectangular $\Theta$-diagram is
a four-tuple $(\alpha,\beta,\gamma,\delta)$ in which $\alpha,\beta,\gamma$ are rectangular paths having common
pairs of ends,
and $\delta$ is a rectangular diagram.
In accordance with Convention~\ref{convent3} we assume by default that all edges of these objects are pairwise non-collinear
and disjoint from the ends of the paths $\alpha$, $\beta$, and $\gamma$.

Two rectangular $\Theta$-diagrams are called \emph{Legendrian equivalent}
if one can be obtained from the other by a finite sequence of \emph{permitted moves},
which include cyclic permutations, commutations, type~I stabilizations and destabilizations, and end shifts
defined below. The
corresponding equivalence class of a rectangular $\Theta$-diagram will be called its
\emph{Legendrian type}.
\end{defi}

\emph{Cyclic permutations}, \emph{commutations} and (\emph{de})\emph{stabilizations} for rectangular
$\Theta$-diagrams are defined similarly to those for rectangular diagrams of links. We focus only on the differences.

There may be not only an edge but also an end of the paths that constitute a rectangular $\Theta$-diagram,
at the extreme left, right, upper, or lower position. Its translation to the opposite side is
also regarded as a cyclic permutation (see Fig.~\ref{endcycle}).
\begin{figure}
\center{\includegraphics{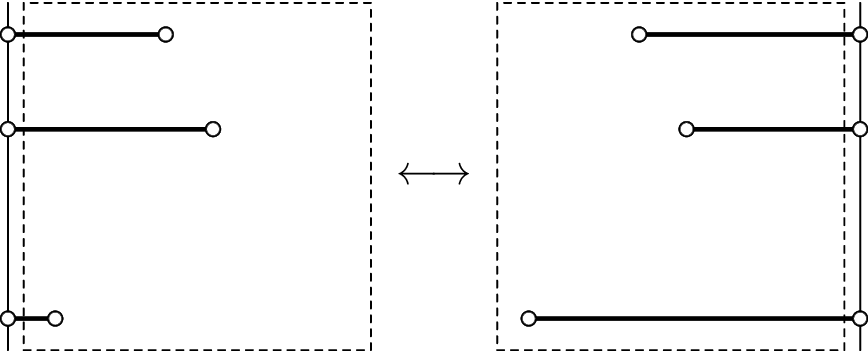}}
\caption{A cyclic permutation involving path ends}\label{endcycle}
\end{figure}

Commutations for edges of a rectangular $\Theta$-diagram are defined in the same way as in
the case of rectangular diagrams of links: one can exchange two neighboring parallel edges, each
may be an edge of $\alpha$, $\beta$, $\gamma$, or $\delta$, provided that their pairs of endpoints do not interleave.
A common end of $\alpha$, $\beta$, and $\gamma$ can also be exchanged with neighboring edges and with
the other end provided that no pair of endpoints at the end to be exchanged interleave
with the endpoints of the edge or a pair of endpoints at the other end, respectively, see Fig.~\ref{endcastling}.
\begin{figure}
\center{\includegraphics{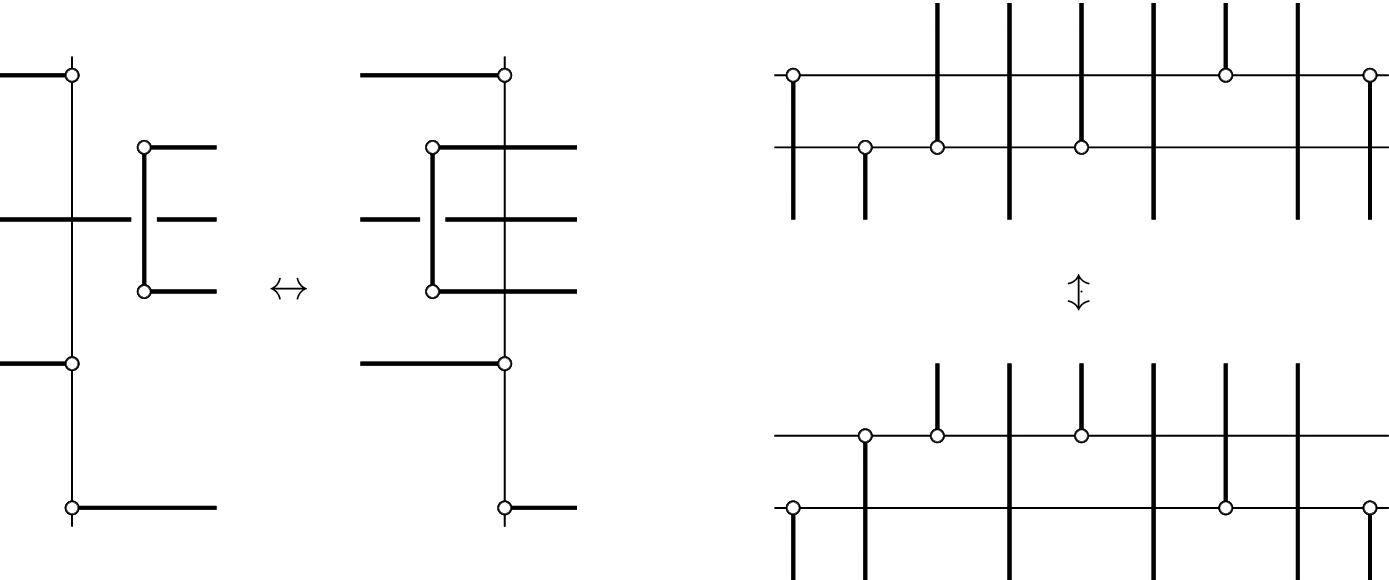}}
\caption{Commutations involving path ends}\label{endcastling}
\end{figure}

Type~I stabilizations and destabilizations are allowed at any vertex of a rectangular $\Theta$-diagram,
including the endpoints of the paths $\alpha$, $\beta$, and $\gamma$.

Finally, we need one more type of moves, which we call \emph{an end shift} and define as follows.
Let $P_1,P_2,P_3$ be three endpoints of the rectangular paths $\alpha$, $\beta$, $\gamma$ (not necessarily in this order)
lying at the same horizontal straight line and listed from left to right.
Denote by $P_i'$ and $P_i''$ points obtained from $P_i$ by a shift by the vector $(0,\varepsilon)$ and $(0,-\varepsilon)$,
respectively, where  $\varepsilon>0$ is a small real number (it should be smaller than the vertical distance
between any two vertices of the diagram that do no lie at the same horizontal straight line).
Any of the following replacements in the paths $\alpha$, $\beta$, and $\gamma$ will be called an end shift:
$$\begin{aligned}
P_1&\mapsto\{P_1',P_2'\},\ &P_2&\mapsto\varnothing,\ &P_3&\mapsto\{P_2,P_3\};\\
P_1&\mapsto\{P_1'',P_3''\},\ &P_2&\mapsto\{P_2,P_3\},\ &P_3&\mapsto\varnothing;\\
P_1&\mapsto\{P_1,P_3\},\ &P_2&\mapsto\{P_2',P_3'\},\ &P_3&\mapsto\varnothing;\\
P_1&\mapsto\varnothing,\ &P_2&\mapsto\{P_2'',P_1''\},\ &P_3&\mapsto\{P_1,P_3\};\\
P_1&\mapsto\varnothing,\ &P_2&\mapsto\{P_1,P_2\},\ &P_3&\mapsto\{P_3',P_1'\},\\
P_1&\mapsto\{P_1,P_2\},\ &P_2&\mapsto\varnothing,\ &P_3&\mapsto\{P_3'',P_2''\},\\
\end{aligned}$$
provided that it results in a $\Theta$-diagram. Namely, the first and the sixth replacement
is allowed if the vertex $P_2$ is not the only vertex of the corresponding rectangular path,
the second and the third if so is $P_3$, and the fourth and the fifth if so is $P_1$.

Note that, from the combinatorial point of view, there are actually three pairs of coinciding moves
among these six replacements.

A similar operation is defined for three endpoints of the rectangular paths included
in a rectangular $\Theta$-diagram if they lie at the same vertical straight line, see Fig.~\ref{endmove}.
\begin{figure}[ht]
\center{\includegraphics{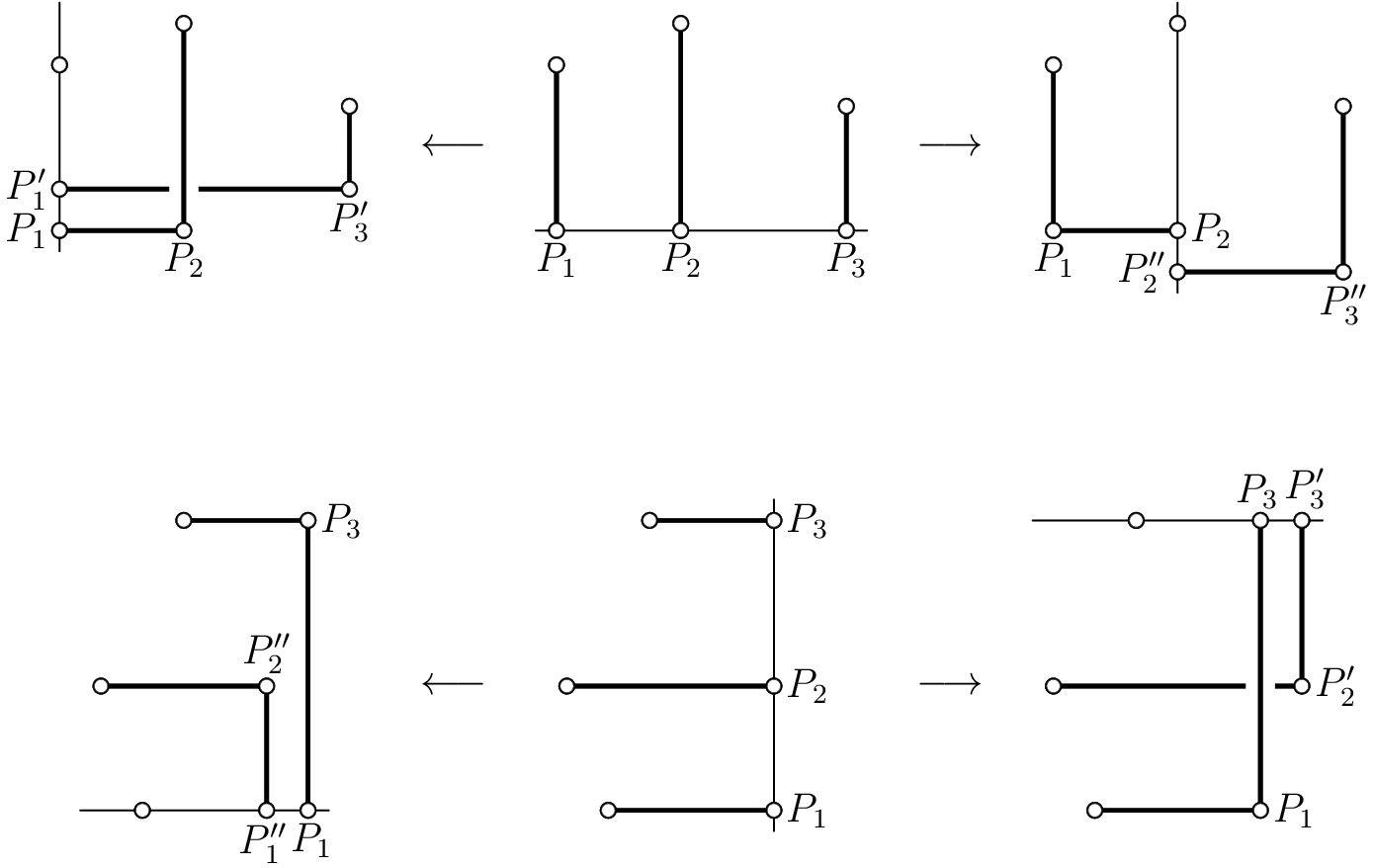}}
\caption{End shift}\label{endmove}
\end{figure}
Namely, one should take the mirror image of the previous construction in a diagonal direction (any of the two).

Thus, in all cases, an end shift consists in deleting one of the vertices $P_1$, $P_2$, $P_3$ and
adding instead two vertices close to the remaining ones. If the deleted vertex
is an endpoint of the rectangular path $\alpha$ of a rectangular $\Theta$-diagram
$\alpha\cup\beta\cup\gamma\cup\delta$, then the rectangular diagram $\beta\cup\gamma\cup\delta$ does not change under
the end shift, and the path $\alpha$ gets longer by one edge one of whose ends shifts along $\beta\cup\gamma$ 
to an adjacent edge. Also one of the paths $\beta$ or $\gamma$ transfers one of its vertices to the other.

We intentionally do not define an inverse operation to an end shift because one can come back
by shifting the end backward and then performing a few commutations and a type~I 
destabilization (see Fig.~\ref{endunmove}).
\begin{figure}
\center{\includegraphics{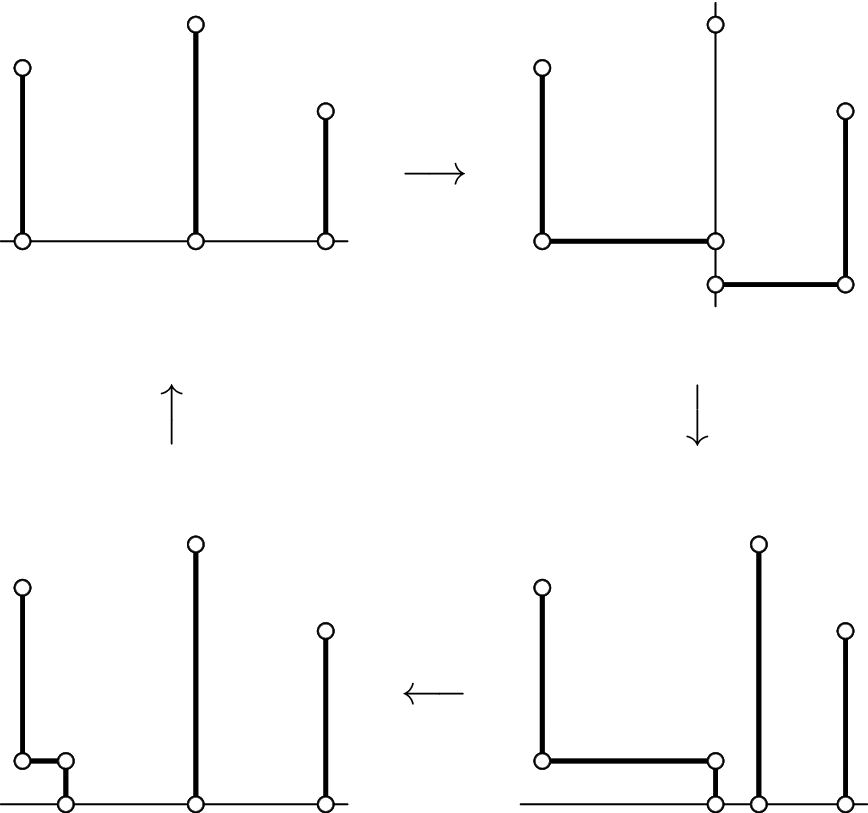}}
\caption{Inverting an end shift}\label{endunmove}
\end{figure}

\begin{prop}\label{abcd}
Let $\alpha\cup\beta\cup\gamma\cup\delta$ and $\alpha'\cup\beta'\cup\gamma'\cup\delta'$
be Legendrian equivalent rectangular $\Theta$-diagrams such that $(\alpha,\beta)$ is a bypass
for $\beta\cup\gamma\cup\delta$ of weight~$b$. Then: \emph{(i)}~$(\alpha',\beta')$ is a bypass
for $\beta'\cup\gamma'\cup\delta'$ of weight~$b$;
\emph{(ii)}~the rectangular diagrams $\alpha\cup\gamma\cup\delta$ and $\alpha'\cup\gamma'\cup\delta'$
are Legendrian equivalent.
\end{prop}

\begin{proof}
It is enough to prove the claim in the case when $\alpha\cup\beta\cup\gamma\cup\delta\mapsto
\alpha'\cup\beta'\cup\gamma'\cup\delta'$ is a single permitted move. This reduces to checking that
the isotopy classes of~$\widetilde\alpha\cup\widetilde\beta\cup\widetilde\gamma\cup\widetilde\delta$ and~$\widetilde D$, where
$D=(\alpha\cup\beta)^\nearrow\cup(\alpha\cup\beta)^\swarrow\cup\alpha\cup\gamma\cup\delta$,
and the Legendrian type of~$\alpha\cup\gamma\cup\delta$
are preserved, which is straightforward. The cases when the
considered operation is a commutation, permitted (de)stabilization, or
a cyclic permutation are trivial and left to the reader.

The least obvious here is to see that
the isotopy class of the link~$\widetilde D$ does not change under an end shift. This is illustrated in Fig.~\ref{associatedlinks}
for one of the possible relative positions of
the end vertices of $\alpha$, $\beta$, and $\gamma$.
Since our statement is true for cyclic permutations of horizontal
edges only the cyclic order of the end vertices matters, which, in turn, can be
reversed by a central symmetry.

The case when the end being shifted is horizontal is obtained
by reflecting in a diagonal direction.

Note that the desired Legendrian equivalences of rectangular diagrams
are established by a few commutations and stabilizations that can be done
by moving only vertices from a small neighborhood of the common
end of $\alpha$, $\beta$, $\gamma$, thus, the position
of other vertices connected by edges to the ones involved in the
discussed transformations plays no role. It also does not matter
whether all three paths $\alpha$, $\beta$, and $\gamma$ have more than one vertex.
Thus, all the cases are reduced to the one shown in Fig.~\ref{associatedlinks}.
\begin{figure}
\centerline{\begin{picture}(470,290)
\put(0,180){\includegraphics[width=100pt]{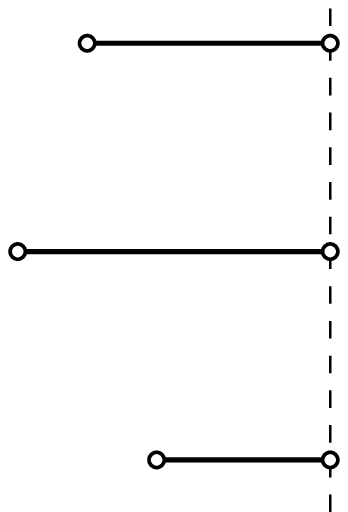}}
\put(110,180){\includegraphics[width=100pt]{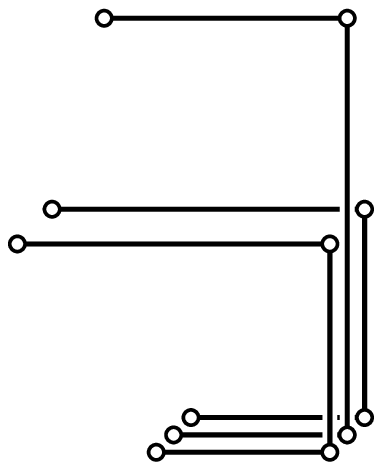}}
\put(260,180){\includegraphics[width=100pt]{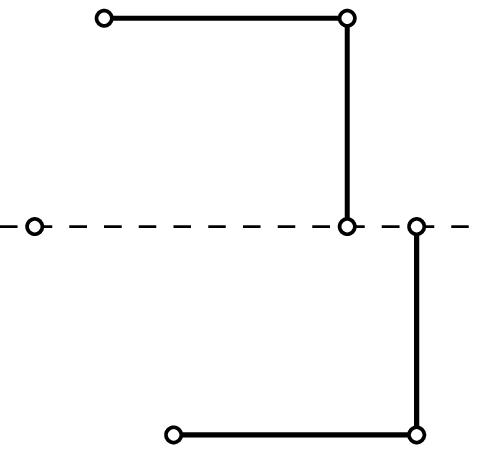}}
\put(370,180){\includegraphics[width=100pt]{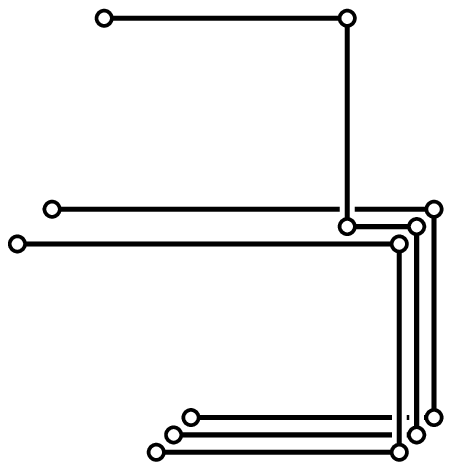}}
\put(0,20){\includegraphics[width=100pt]{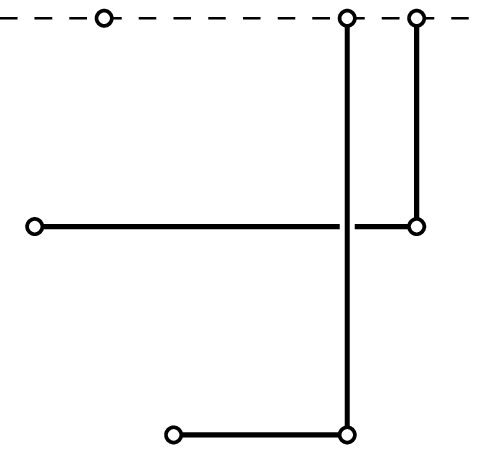}}
\put(110,20){\includegraphics[width=100pt]{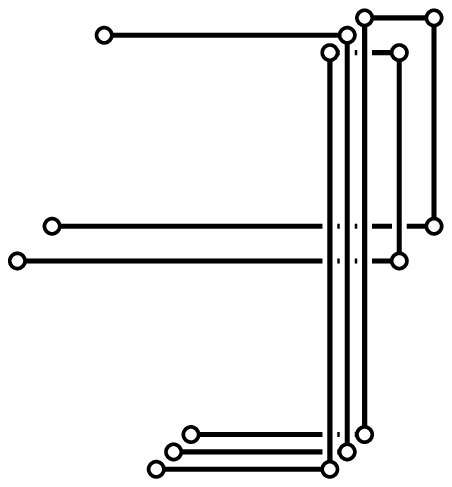}}
\put(260,20){\includegraphics[width=100pt]{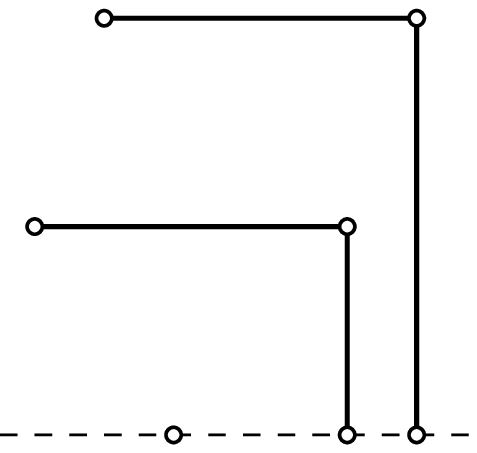}}
\put(370,20){\includegraphics[width=100pt]{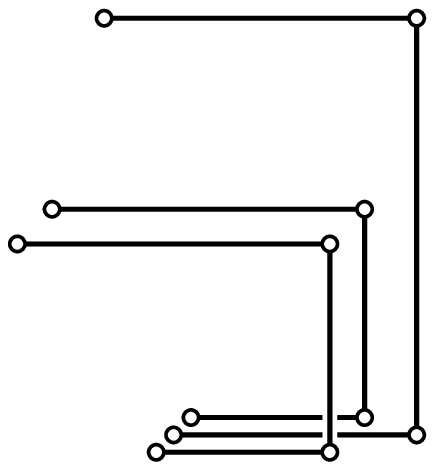}}
\put(50,195){$\alpha$}
\put(50, 240){$\beta$}
\put(50,284){$\gamma$}
\put(150,170){$D$}
\put(310,195){$\alpha'$}
\put(260, 240){$\beta'$}
\put(310,284){$\gamma'$}
\put(410,170){$D'$}
\put(50,35){$\alpha'$}
\put(50, 80){$\beta'$}
\put(15,124){$\gamma'$}
\put(150,10){$D'$}
\put(288,35){$\alpha'$}
\put(310, 80){$\beta'$}
\put(310,124){$\gamma'$}
\put(410,10){$D'$}
\end{picture}}
\caption{In all cases the diagram~$D'=(\alpha'\cup\beta')^\nearrow\cup(\alpha'\cup\beta')^\swarrow\cup\alpha'\cup\gamma'\cup\delta'$
is equivalent to $D=(\alpha\cup\beta)^\nearrow\cup(\alpha\cup\beta)^\swarrow\cup\alpha\cup\gamma\cup\delta$,
and $\alpha\cup\gamma\cup\delta$ is Legendrian equivalent to $\alpha'\cup\gamma'\cup\delta'$}\label{associatedlinks}
\end{figure}
\end{proof}

Surely there is a natural correspondence, which we don't formally use, between Legendrian types of
rectangular $\Theta$-diagrams and Legendrian types of one-dimensional complexes in  $\mathbb R^3$
having form of a few circles and one $\Theta$-component, which is a graph with two vertices
connected by three edges, such that all curves in the complex are Legendrian.
For completeness of the exposition we briefly mention moves of front projections of such complexes
that involve a 3-valent vertex and, together with the moves in Fig.~\ref{frontmoves},
generate the equivalence of Legendrian graphs that we are interested in.
\begin{figure}[ht]
\center{\includegraphics[scale=1.7]{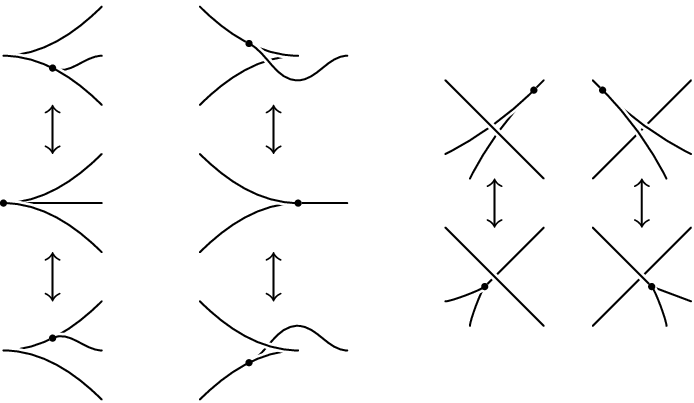}}
\caption{Moves of a Legendrian graph involving a 3-valent vertex}\label{3-graphs}
\end{figure}
This moves are shown in Fig.~\ref{3-graphs} (one should also add the moves symmetric to the shown ones).

\begin{prop}\label{bypasspreserved}
Let $R$ and $R'$ be rectangular diagrams of the same Legendrian type, and let
$\alpha$ be a rectangular path with ends at edges of the diagram $R$. Then
there exists a rectangular path~$\alpha'$ with ends at edges of $R'$ such
that the rectangular $\Theta$-diagrams $R\cup\alpha$ and $R'\cup\alpha'$ are Legendrian equivalent.
\end{prop}

\begin{proof}
It suffices to prove the statement in the case when $R'$ is obtained from $R$ by a single elementary move
preserving the Legendrian type. For a stabilization it is obvious since we can simply take
$\alpha'=\alpha$ (we assume that the new edges emerged from the stabilization are short enough).

Let $R\mapsto R'$ be a cyclic permutation. For definiteness we assume that the extreme left edge is translated
to the right. The other cases are similar.

If there are edges of $\alpha$ located on the left of $R$, then we first move them to the right
by using cyclic permutations, and then perform the desired cyclic permutation on $R$ with no obstacle
regardless whether the translated edge contains an end of $\alpha$ or not.

Let $R\mapsto R'$ be a commutation. Due to the symmetry of our constructions
it suffices to consider a commutation of horizontal edges. Moreover, since we have already
proved the statement for cyclic permutations, we may assume that the horizontal projections
of the exchanged edges do not overlap, and the lower of them
has its left endpoint at the very left of the $\Theta$-diagram~$R\cup\alpha$, i.e. there are no vertices
of $R$ or $\alpha$ on the left of this vertex.

There may be only two obstructions originating from $\alpha$ to perform the desired commutation on $R$:
(i) the ends of $\alpha$ can lie on the exchanged edges or edges adjacent to them; (ii) there may
be edges of $\alpha$ in the strip between the straight lines containing the exchanged edges.

In the first case one shifts the ends of $\alpha$ to other horizontal edges of $R$, which
can be done by applying a few end shift moves. We now draw our attention to the second type of obstacle.

Denote the left of the commuted edges of the diagram~$R$ by $e$ and the right one by~$e'$.
Denote by $P$ the right vertex of the edge $e$. Without the presence of~$\alpha$ the desired commutation could have
been made by shifting the edge~$e$ upward. The only obstruction to that may be edges of~$\alpha$
that are located over the vertex~$P$ but below the level of the edge $e'$, see Fig.~\ref{bypasscastling}~(a).

\begin{figure}[ht]
\center{\includegraphics{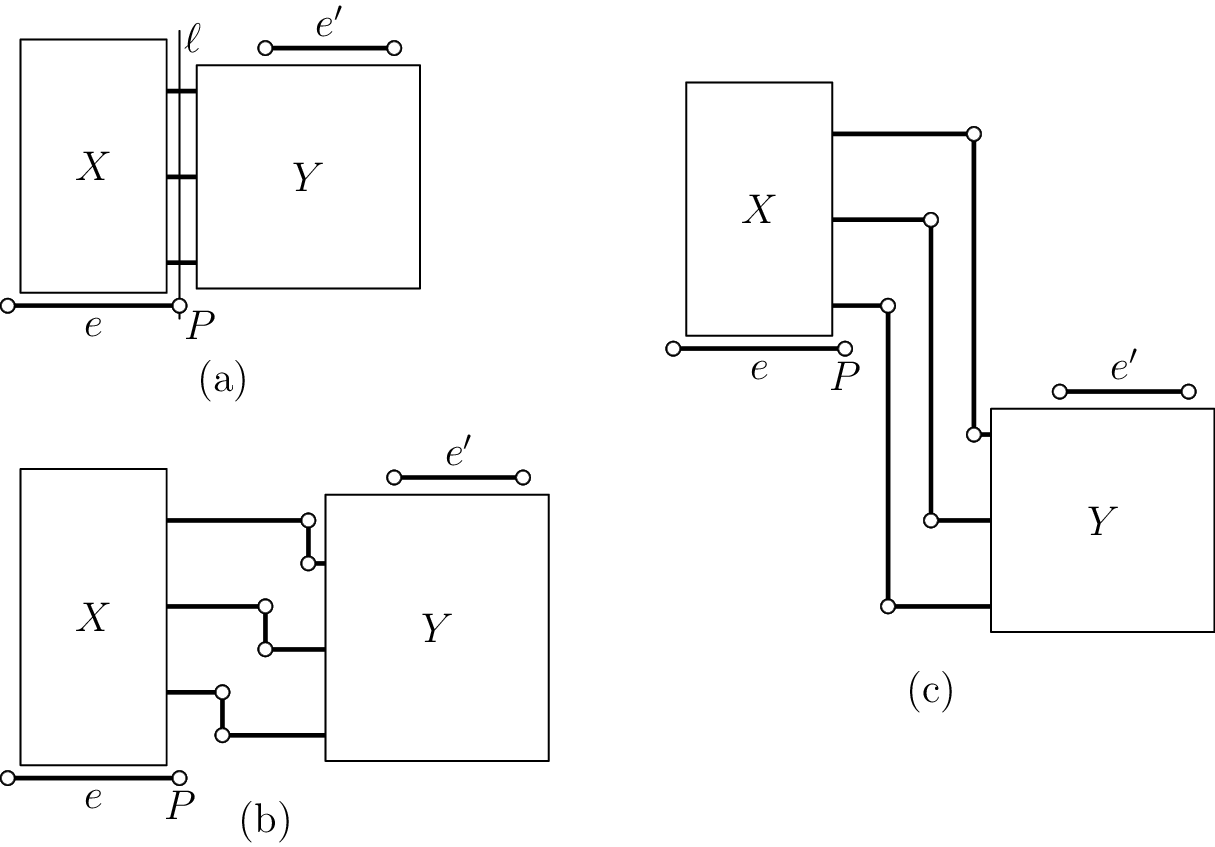}}
\caption{Commutation in the presence of a rectangular path. There may be unshown vertical edges of the path~$\alpha$
coming out from the boxes $X$ and $Y$.}\label{bypasscastling}
\end{figure}

Let $\ell$ be the vertical straight line passing through $P$. ``Break'' every obstructing edge of~$\alpha$ near
its intersection with the line~$\ell$. Namely, we apply an appropriate permitted stabilization near an endpoint of each
obstructing edge
and then shift the new short vertical edge toward the line~$\ell$ by using commutations.
We arrange these new short vertical edges on the right of $\ell$ so that the lower an edge
the closer to~$\ell$ it is positioned, see Fig.~\ref{bypasscastling}~(b).

Now we can freely shift all edges located over~$e$ but below the level of~$e'$ upward so
as to remove the obstacle to the desired commutation on~$R$, and implement it, see Fig.~\ref{bypasscastling}~(c).

We are left to consider the case when $R\mapsto R'$ is a permitted destabilization.
There are two kinds of such destabilizations, but since they are symmetric to each other it suffices
to consider one of them, when the edges being reduced point from their common vertex
upward and rightward. By using cyclic permutations and shifting the ends of~$\alpha$
we can make it so that the common vertex of the edges being reduced is the leftmost and lowermost
vertex of the diagram $R\cup\alpha$ and the endpoints of~$\alpha$ do not lie at the edges being reduced
and ones adjacent to them.
\begin{figure}[ht]
\center{\includegraphics{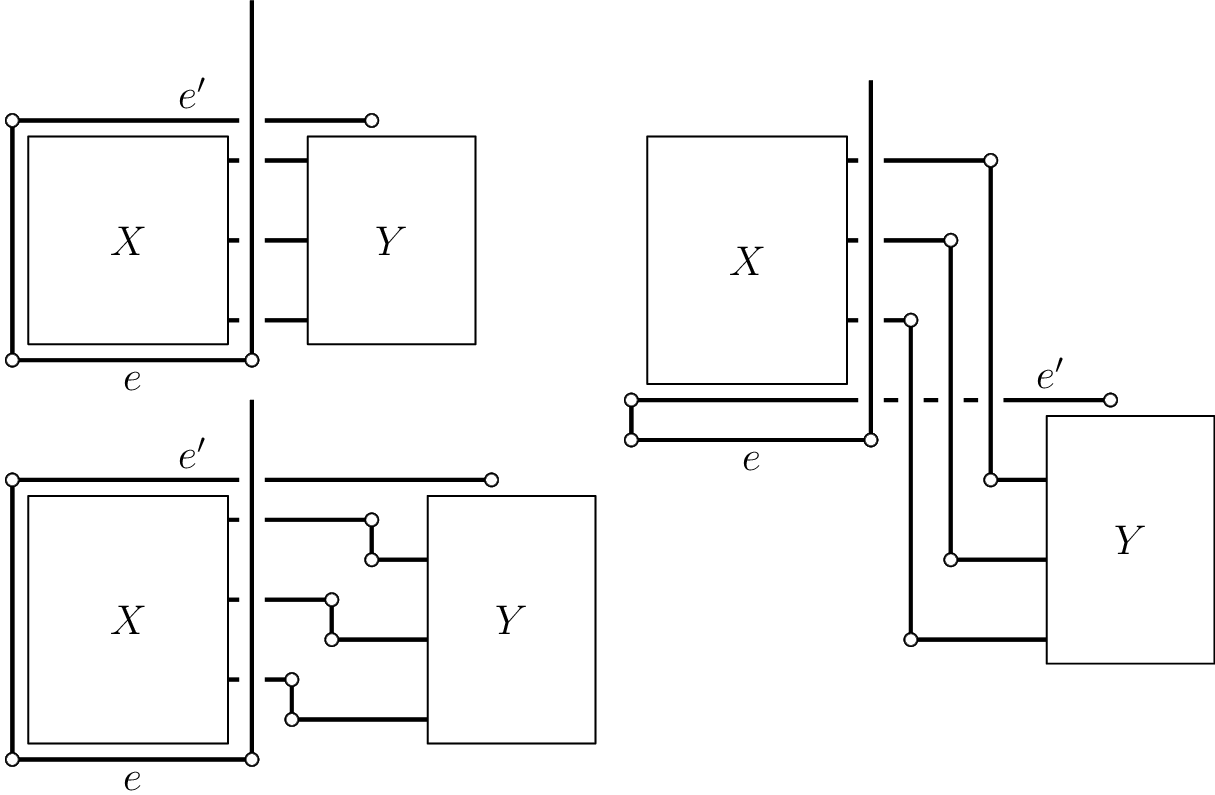}}
\caption{Destabilization in the presence of a rectangular path.
There may be unshown vertical edges of~$\alpha$ coming
out from the boxes~$X$ and $Y$}\label{bypassdestab}
\end{figure}

The horizontal edge that is going to be reduced will be denoted by~$e$, and the horizontal edge
of~$R$ positioned immediately above it by~$e'$.
An obstruction to move the edge~$e$ upward so as to make it neighboring to $e'$ in
the diagram $R\cup\alpha$ can come from edges positioned over the right endpoint of~$e$
but below the level of~$e'$. We proceed exactly the same as in the case of commutation,
see Fig.~\ref{bypassdestab}.

After that we shift the edge~$e$ upward close to~$e'$, then freely shift the vertical edge we want to reduce
to the right, and perform the desired destabilization.
\end{proof}

\subsection{The Key Lemma and its corollaries}
By a \emph{connected component} of a rectangular diagram~$R$ of a link we mean
any subset $K\subset R$ that forms a rectangular diagram of a knot.

\begin{keylem}
Let $R$ be a rectangular diagram of a link, and $(\alpha,\beta)$ be a bypass whose
weight is smaller than the number of vertical edges of the component of $R$ that
contains the ends of~$\alpha$.

Then there exists a rectangular diagram $R'$ that is Legendrian equivalent to $(R\setminus\beta)\cup\alpha$
and can be obtained from $R$ by $b$ successive type~II elementary simplifications, where $b$ is the weight of the
bypass~$(\alpha,\beta)$.
\end{keylem}

The proof of this statement will occupy the whole next section. Here we deduce basic corollaries from it.

\begin{theo}\label{singlesimplification}
Let rectangular diagrams $R_1$ and $R_2$ be Legendrian equivalent, and
$R_1$ admit  a type~II elementary simplification $R_1\mapsto R_1'$. Then
the diagram $R_2$ admits a type~II elementary simplification $R_2\mapsto R_2'$ such that
the diagram $R_2'$ is Legendrian equivalent to~$R_1'$.
\end{theo}

\begin{proof}
Since commutations and cyclic permutations preserve the Legendrian type of a
rectangular diagram it suffices to consider the case when $R_1\hm\mapsto R_1'$ is a type~II
destabilization. In this case, as one can see in Fig.~\ref{elementarybypass}, there exists an elementary
bypass~$\alpha_1$ for $R_1$
\begin{figure}[ht]
\centerline{\begin{picture}(386,82)
\put(0,0){\includegraphics{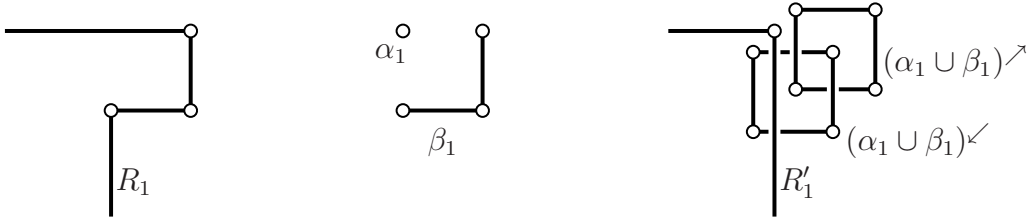}}
\put(42,10){$R_1$}
\put(160,25){$\beta_1$}
\put(140,60){$\alpha_1$}
\put(293,10){$R_1'$}
\put(318,25){$(\alpha_1\cup\beta_1)^\swarrow$}
\put(332,55){$(\alpha_1\cup\beta_1)^\nearrow$}
\end{picture}}
\caption{An elementary bypass consisting of one vertex}\label{elementarybypass}
\end{figure}
that consists of just one vertex, and $R_1'$ is obtained from $R_1$ by replacing
the corresponding bypassed path~$\beta_1$ (consisting of three vertices) by $\alpha_1$.

If follows from Proposition~\ref{bypasspreserved} that there exists a rectangular path~$\alpha_2$
such that the rectangular $\Theta$-diagrams
$R_1\cup \alpha_1$ and $R_2\cup\alpha_2$ are Legendrian equivalent.
According to Proposition~\ref{abcd}
this means, in particular, that $\alpha_2$ is an elementary bypass for~$R_2$.
Let $\beta_2$ be the corresponding bypassed path.
Then the diagram $(R_2\setminus\beta_2)\cup\alpha_2$ is Legendrian equivalent
to~$(R_1\setminus\beta_1)\cup\alpha_1=R_1'$.

Since the bypass $\alpha_2$ is elementary, it satisfies the hypothesis of the Key Lemma whose
application completes the proof.
\end{proof}

\begin{coro}\label{coroone}
Let $R$ be a rectangular diagram of a link such that the corresponding Legendrian link $L_R$ (respectively, $L_{\overline R}$)
admits a destabilization $L_R\mapsto L'$ (respectively,
$L_{\overline R}\mapsto L'$).
Then the diagram $R$ admits a type~II (respectively, type~I) elementary simplification $R\mapsto R'$
such that links $L'$ and $L_{R'}$ (respectively, $L'$ and $L_{\overline{R'}}$) are Legendrian equivalent.
\end{coro}

\begin{proof}
This immediately follows from Theorems~\ref{L_R} and \ref{singlesimplification}.
\end{proof}

We also get the main result of~\cite{Dyn}:

\begin{coro}[Monotonic simplification theorem for the unknot]\label{monosimpl}
Any nontrivial rectangular diagram of the unknot admits successive elementary
simplifications arriving at the trivial diagram.
\end{coro}

\begin{proof}
Nontriviality of the diagram $R$ means that is complexity exceeds~$2$.
Together with~\eqref{complexityandtb} it implies that at least one
of the Legendrian knots $L_R$ and $L_{\overline R}$ has Thurston--Bennequin number smaller than~$(-1)$.
It now follows from Eliashberg--Fraser's theorem on classification of Legendrian unknots (Theorem~\ref{EFtheo} above)
that this Legendrian knot
admits a destabilization. Therefore, $R$ admits an elementary simplification. An obvious induction
on the complexity of the diagram $R$ completes the proof.
\end{proof}

In the end of Section~\ref{proofofkeylemma} we show one more way to proof
the monotonic simplification theorem that does not need the use
of Eliashberg--Fraser's theorem.

The statement announced in the Introduction is one more corollary of Theorem~\ref{singlesimplification}.

\begin{proof}[Proof of Theorem~\ref{T2types}]
The diagrams $R$ and $R_k'$ are Legendrian equivalent by hypothesis.
By applying Theorem~\ref{singlesimplification} successively  $\ell$ times we prove the existence
of a sequence of elementary simplifications
$R_k'\mapsto R_1'''\mapsto R_2'''\mapsto\ldots\mapsto R_\ell'''$
in which the diagram $R_i'''$ is Legendrian equivalent to
$R_i''$ for all $i=1,\ldots,\ell$, which is a reformulation of the conclusion of the theorem.

The second claim is symmetric to the first one.
\end{proof}

Now we establish more general connections between bypasses and simplifications.

\begin{theo}\label{bypass=simplifications}
Let $R$ be a rectangular diagram of a link, $K\subset R$ be one of its connected components, 
and $\mathcal L$ be some Legendrian type. The following conditions are equivalent:
\def\labelenumi{(C\theenumi)}
\begin{enumerate}
\item
The diagram $R$ admits  $b>0$ successive type~II elementary simplifications on a component
$K$ that arrive at a diagram having Legendrian type~$\mathcal L$;
\item
There exists a bypass $\alpha$ of weight $b$ for the diagram $R$
with ends at edges of the component~$K$, and the replacement
of the bypassed path by the bypass results in a diagram of Legendrian type~$\mathcal L$.
\end{enumerate}
\end{theo}

\begin{proof}
$(\text C2)\Rightarrow(\text C1)$
Apply a type~I stabilization to the diagram $R$ at the component~$K$ sufficiently many times
so as to make the number of vertical edges of~$K$ larger than~$b$. Denote
the obtained rectangular diagram by~$\check R$.

According to the Key Lemma there exists a sequence of elementary
simplifications
$\check R\hm\mapsto\check R_1\mapsto\check R_2\mapsto\ldots\mapsto\check R_b$
such that the last diagram in it is Legendrian equivalent to $(\check R\setminus\beta)\cup
\alpha$, which, in turn, is Legendrian equivalent to $(R\setminus\beta)\cup\alpha$,
where~$\beta$ is the bypassed path.

By applying Theorem~\ref{singlesimplification} successively  $b$ times we prove that there exists
a sequence of elementary simplifications
$R\mapsto R_1\mapsto R_2\mapsto\ldots\mapsto R_b$ in which
every diagram $R_i$ is Legendrian equivalent to the corresponding diagram~$\check R_i$.
\smallskip

\noindent $(\text C1)\Rightarrow(\text C2)$
We make use of the following two lemmas, which will be also used in the sequel.

\begin{lemm}\label{stabilizations2start}
Let $R_0\mapsto R_1\mapsto\ldots\mapsto R_m$ be an arbitrary sequence of elementary moves
that contains $k$ stabilizations and $\ell$ destabilizations.
The there exists another sequence of elementary moves $R_0\mapsto R_1'\mapsto\ldots\hm\mapsto R_n'=R_m$,
in which the first $k$ moves are stabilizations, the last~$\ell$ moves are destabilization
and all the others do not include stabilizations and destabilizations.
Moreover, the number of stabilizations and destabilizations of each type in this sequence is
the same as that in the original one.
\end{lemm}

\begin{proof}
This statement is proved by induction on the tuple $(k,\ell,p,s)$, where
$k$ is the number of stabilizations, $\ell$ is the number of destabilizations,
$(p+1)$ is the number of the first stabilization (if there is no such we set $p=0$),
and $(m-s)$ is the number of the last destabilization
(if there is no such we set $s=0$). Tuples $(k,\ell,p,s)$ are ordered lexicographically.

The induction base $k=\ell=0$ is obvious. For induction step, it suffices to prove our statement in the case $m=2$.
Indeed, if the first move is a stabilization, then by removing it we obtain a sequence with a smaller $k$
and the same $\ell$, the validity of the statement for which implies that for the original sequence.
Similarly, if the last move is a destabilization, then by removing it we obtain a sequence with the same $k$
but smaller $\ell$. If none of these takes place but $k>0$, then
we apply the statement to the subsequence consisting of the $p$th and the $(p+1)$st moves
and obtain a sequence with the same $k,\ell$ but smaller $p$. Finally, if $k=0$ and $\ell>0$,
we apply similarly the statement to the subsequence consisting of the $(m-s)$th and $(m-s+1)$st moves.

So it remains to prove the statement for sequences consisting of two operations in the following two cases:
(1) the second operation is a stabilization and the first one is not a stabilization;
(2) the first operation is a destabilization and the second one is not a destabilization.
These cases are obtained from each other by reversing the sequence, so it suffices
to consider just one of them. Fig.~\ref{movestabtostart} shows how to move a stabilization to
the beginning of the sequence if it is preceded by a cyclic permutation, commutation or destabilization
in situations where it is not just a question of changing the order of moves.
The other cases are similar and are left to the reader.
\begin{figure}[ht]
\center{\includegraphics[scale=1.5]{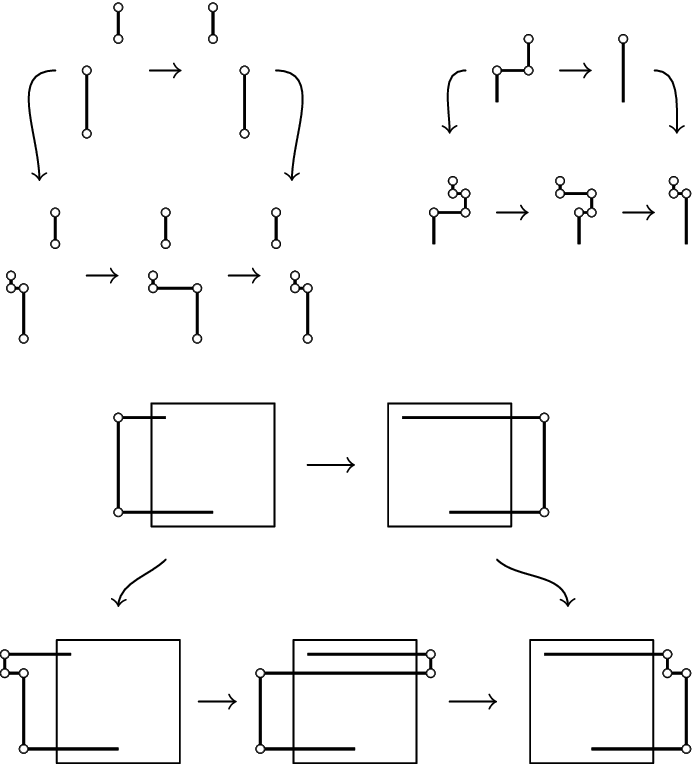}}
\caption{Moving a stabilization to the beginning of the sequence}\label{movestabtostart}
\end{figure}
\end{proof}

\begin{lemm}\label{stabilizationatanyvertex}
Let $R\mapsto R'$ be a stabilization, and $P$ be any vertex of the same component of $R$
on which this stabilization occur. Then there is a stabilization $R\mapsto R''$ of the same type as $R\mapsto R'$
in a small neighborhood of $P$ such that the diagrams $R'$ and $R''$ can be connected
by a sequence of commutations and cyclic permutations.
\end{lemm}

\begin{proof}
It is enough to consider the case when the vertex $P$ is connected by an edge to the one
in whose neighborhood occurs the stabilization $R\mapsto R'$. Let $e$ be this edge. As a result of
the stabilization $R\mapsto R'$ two short edges appear. The one that is orthogonal to $e$ can be
shifted toward $P$ by using commutations and at most one cyclic permutation.
For such obtained diagram $R''$ the move $R\mapsto R''$ will be a stabilization, see Fig.~\ref{movestabplace}.
\begin{figure}[ht]
\center{\includegraphics{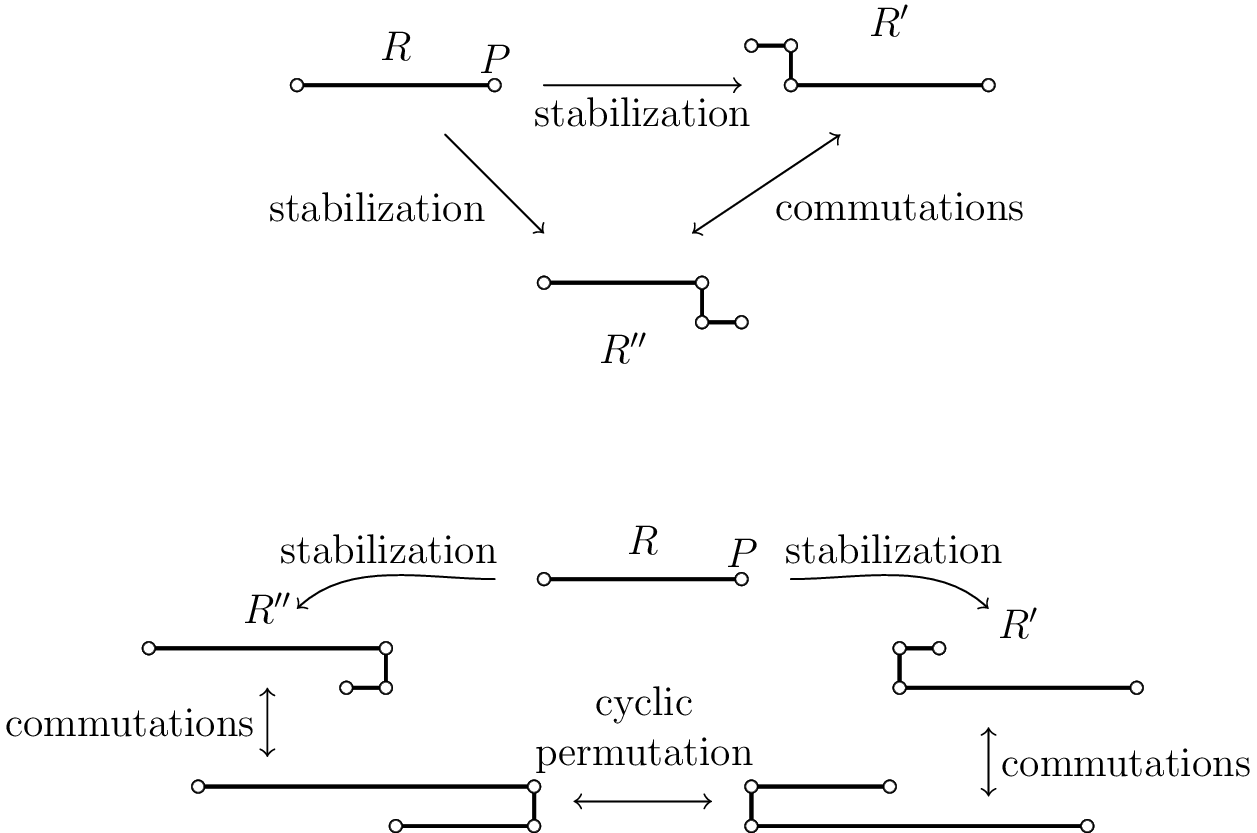}}
\caption{Moving the stabilization scene to a neighborhood of another vertex}\label{movestabplace}
\end{figure}
\end{proof}

Now we resume the proof of Theorem~\ref{bypass=simplifications}.
Let $R=R_0\mapsto R_1\mapsto\ldots\mapsto R_b$ be a sequence of elementary simplifications that satisfies Condition~(C1).
It follows from Lemma~\ref{stabilizations2start} that $R_b$ can be obtained from $R$ by a sequence of
elementary moves in which the last $b$ moves are destabilizations and all the previous ones are
cyclic permutations and commutations. Moreover, according to Lemma~\ref{stabilizationatanyvertex} 
we may assume that all $b$ destabilizations occur in a small neighborhood of the same vertex $P$ of the diagram $R_b$.

Therefore, we can find a sequence of complexity preserving elementary moves
from the diagram $R$ to some diagram $R'$ such that
$R_b$ can be obtained from $R'$ by a sequence of $b$ type~II destabilizations
that occur in a small neighborhood of a vertex $P$ of the diagram $R_b$.
One can see that the vertex $P$ is then a bypass of weight $b$, and
transition from $R'$ to $R_b$ is a replacement of the bypassed path by $P$,
see Fig.~\ref{weight3bypass}.
\begin{figure}[ht]
\center{\includegraphics{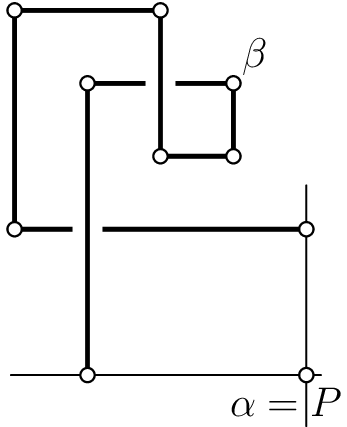}}
\caption{A bypass $\alpha$ of weight $b=3$}\label{weight3bypass}
\end{figure}
It now follows from Propositions~\ref{bypasspreserved} and~\ref{abcd} that there exists a bypass of weight~$b$
for $R$, and placing it instead of the bypassed path results in a diagram that is
Legendrian equivalent to~$R_b$.
\end{proof}

\begin{rem}
It follows from Theorem~\ref{bypass=simplifications} that the restriction on the weight of the bypass
in the hypothesis of the Key Lemma is actually unnecessary,
the weight of a bypass is always smaller than the number
of vertical edges of the component to whose edges the ends of
the bypass are attached.
\end{rem}

\begin{rem}
Theorem~\ref{bypass=simplifications} can be strengthen by allowing simplifications
on different components and by introducing an appropriate notion of independent bypasses.
Namely, bypasses are independent if the corresponding discs mentioned
in the bypass definition are disjoint.
If there is a collection of independent bypasses, then each of them
independently provides for as many successive elementary simplifications on the component
it is attached to as is its weight.
The proof of this statement encounters no new difficulties compared
to the one of Theorem~\ref{bypass=simplifications}, but in order to keep
exposition clearer we prefer to restrict ourselves to the case of a single bypass.
\end{rem}

\begin{theo}\label{legendriancombine}
Let $\mathcal L_1$ and $\mathcal L_2$ be two Legendrian types of Legendrian links whose
topological types are mirror images of each other.
Then there exists a rectangular diagram $R$ such that
the Legendrian links $L_R$ and $L_{\overline R}$ have Legendrian types
$\mathcal L_1$ and $\mathcal L_2$, respectively.
\end{theo}

\begin{proof}
Let $R_1$ and $R_2$ be rectangular diagrams whose corresponding Legendrian links $L_{R_1}$ and
$L_{\overline{R_2}}$ have types $\mathcal L_1$ and $\mathcal L_2$, respectively.
Then the topological link types defined by them are the same,
and the diagrams can be obtained from each other by a sequence of
elementary moves.

By Lemma~\ref{stabilizations2start} we may assume that the sequence starts
from stabilizations, then only commutations and cyclic permutations occur for a while,
and destabilizations occur at the end of the sequence.
Lemma~\ref{stabilizationatanyvertex} implies that, moreover,
we can make it so that stabilizations of different types occur far from each other, i.e. in small neighborhoods
of different vertices. In this case they commute, and we may assume without loss of generality
that first all type~I stabilizations occur, and then follow all type~II ones.
Since type~I stabilizations do not affect the Legendrian type of $L_{R_1}$
we may simply assume that they are not present in the sequence.

Similarly, we may assume without loss of generality that our sequence of elementary moves does not include
type~II destabilizations.

Starting from the diagram $R_1$ and having made all (type~II) stabilizations we get
a diagram~$R'$ that can be turned into $R_1$ by type~II elementary simplifications,
and into $R_2$ by type~I elementary simplifications.
It now follows from Theorem~\ref{T2types} that a rectangular diagram~$R$ exists
that can be obtained from $R_2$ by type~II elementary simplifications and is
Legendrian equivalent to~$R_1$. The diagram~$R$ is a sought for one.
\end{proof}

Theorem~\ref{legendriancombine} and relation~\eqref{complexityandtb} imply
a positive answer to Question~1 by J.Green from~\cite{NG}:

\begin{coro}
If $R$ is the simplest rectangular diagram representing a given link type, then $L_R$ has the largest
Thurston--Bennequin number among all Legendrian links of the same topological type.
\end{coro}

\section{Proof of the Key Lemma}\label{proofofkeylemma}
The plan of the proof is as follows.
\begin{itemize}
\item
With the rectangular diagram $R$ and the rectangular paths $\alpha$ and $\beta$
we associate geometrical objects
$\widehat R$, $\widehat\alpha$, $\widehat\beta$ in $\mathbb R^3$ that are called arc presentations.
\item
We span the trivial knot $\widehat{\alpha\cup\beta}$ by a disc $D$ that obey certain restrictions.
By using the ``open book'' foliation we introduce a special combinatorial structure on the disc.
\item
Then we apply induction. For the induction step we modify $\widehat R\cup D$
in a certain way so that the disc~$D$ gets simpler.
As a result, type~II destabilizations may occur on the path~$\beta$, type~I stabilizations on~$\alpha$,
and commutations as well as cyclic permutations may occur everywhere in the
$\Theta$-diagram $R\cup\alpha$.
\item
Rearranging saddles is an operation that shortens $\alpha$ or $\beta$
in some cases, and in others leads to a situation in which smoothing
out a wrinkle inside the disc~$D$ is possible.
\item
Smoothing out a wrinkle simplifies the disc~$D$. If the wrinkle is
at the boundary, then $\alpha$ or $\beta$ shortens.
\item
If further simplification of~$D$ is impossible, then $D$ has
a standard form, and the transition from~$R$ to~$(R\setminus\beta)\cup\alpha$
is a type~II elementary simplification. This is the induction base.
\end{itemize}

\subsection{Arc presentations}\label{booksection}

We fix the standard cylindrical coordinate system $(\rho,\theta,z)$ in $\mathbb R^3$
whose axis will be denoted by $\ell$ and called \emph{the binding line}.
The half-planes of the form $\{\theta=\mathrm{const}\}$ will be called \emph{pages}.
A page $\{\theta=\theta_0\}$ will be denoted by $\page_{\theta_0}$.

\begin{defi}
By \emph{an arc presentation} of a topological link type $\mathcal L$ we call
a link $L\in\mathcal L$ composed of a finite number of smooth simple arcs with
endpoints at~$\ell$, each arc lying in a separate page and having its interior disjoint from the binding line.
Intersection points $L\cap\ell$ are called \emph{vertices} of the arc-presentation~$L$.
\end{defi}

As noticed in~\cite{Dyn} arc presentations viewed combinatorially
are the same thing as rectangular diagrams. More precisely, to any
arc presentation~$L$ there corresponds a rectangular diagram~$R$
whose vertices are those points $(\theta_0,z_0)\in\mathbb R^2$, $\theta_0\hm\in[0,2\pi)$
for which the link~$L$ has an arc in the page~$\page_{\theta_0}$ with an endpoint at $z=z_0$, see Fig.~\ref{arclink}.

\begin{figure}[ht]
\center{\includegraphics{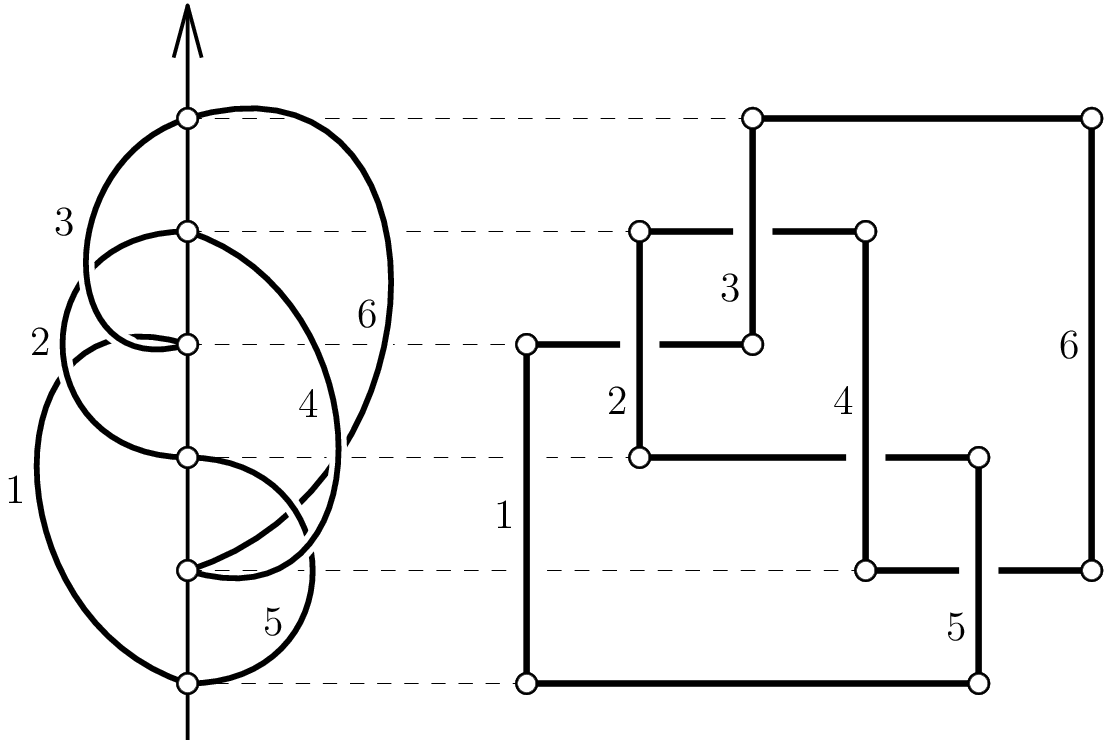}}
\caption{An arc presentation and the corresponding rectangular diagram}\label{arclink}
\end{figure}
The inverse construction, an arc presentation from a rectangular diagram, is obvious
with just one remark: the rectangular diagram should be contained in the strip~$[0,2\pi)\hm\times\mathbb R$,
which impose no restriction on the combinatorics of~$R$, and is assumed by default in the sequel.
The arc presentation corresponding to a rectangular diagram~$R$ will be denoted by~$\widehat R$.

For a rectangular path~$\gamma$, the associated \emph{book-like path}~$\widehat\gamma$
is defined similarly provided that the ends of~$\gamma$ are horizontal.
Namely, to any vertical edge~$[(\theta_0,z_1),(\theta_0,z_2)]$ of $\gamma$
there corresponds an arc of~$\widehat\gamma$ in the page $\page_{\theta_0}$
with endpoints at~$z=z_1$ and $z=z_2$, see Fig.~\ref{arcpath}.
\begin{figure}[ht]
\center{\includegraphics{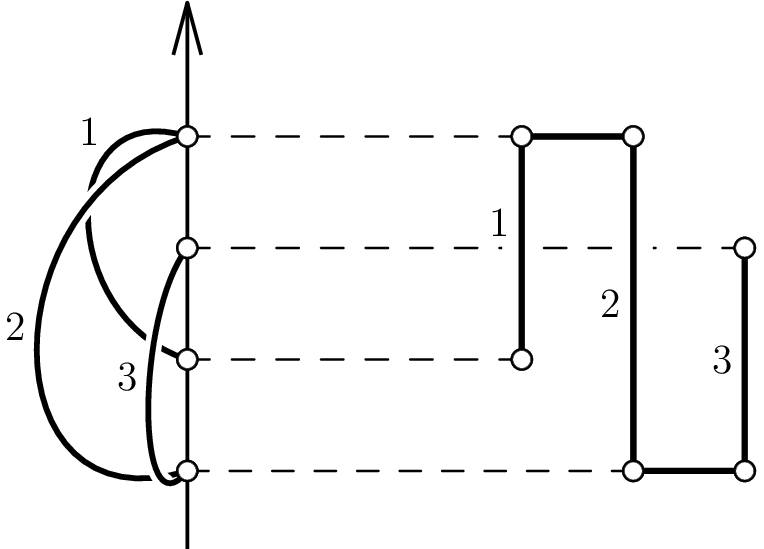}}
\caption{Book-like path and the corresponding rectangular path}\label{arcpath}
\end{figure}
Only such rectangular paths are considered in the sequel. By \emph{the length}
of such a path~$\gamma$ we call the number of vertical edges in~$\gamma$,
which coincides with the number of pages occupied by~$\widehat\gamma$.

Note that the geometric realization~$\widehat X$ of a rectangular diagram or
path~$X$ has an inessential difference from
the one denoted by~$\widetilde X$ in Subsection~\ref{rectdiagrdescr}.
To obtain~$\widehat X$ from $\widetilde X$ one only needs to designate~$\ell_2$
to be the binding line, smoothen the arcs, and rotate them around the binding line keeping
preserved their cyclic order.

By the hypothesis of the Key Lemma there is a bypass~$\alpha$ of weight~$b$
for the rectangular diagram~$R$, and, moreover, the component of~$R$
containing~$\partial\alpha$ has complexity larger than~$b$. Nothing is assumed
on the position of the ends of the bypass, but by using the end shift move
we can place the ends to any edges of the same components.
So, without loss of generality we may assume that the following condition is satisfied.

\begin{agree}
Throughout the rest of the section it is assumed
that the ends of the bypass~$\alpha$ lie on horizontal edges of the diagram~$R$,
and the length of~$\beta$ equals exactly~$b$.
\end{agree}

\subsection{Suitable disc~$D$}\label{properdisc} Now we proceed with constructing an embedded cooriented two-dimensional
disc~$D\hookrightarrow\mathbb R^3$ satisfying the following conditions:
\begin{enumerate}
\def\labelenumi{(D\theenumi)}
\item\label{d1} the boundary~$\partial D$ coincides with~$\widehat{\alpha\cup\beta}$, and the interior of~$D$
is disjoint from~$\widehat R \cup\widehat\alpha$;
\item\label{d2}
 $D$ is the image of a~$(a+b)$-gon under a regular smooth map that takes
 the vertices of the polygon to the vertices of the knot~$\widehat{\alpha\cup\beta}$, where $a$
 is the length of~$\alpha$;
\item\label{d3}
the interior of~$D$ intersects the binding line~$\ell$ transversely finitely many times;
\item\label{d4}
$D$ is orthogonal to~$\ell$ at points from~$\partial D\cap\ell$ and, moreover,
for any of the common endpoints of~$\widehat\alpha$ and $\widehat\beta$ the arc
of $\widehat{R\setminus\beta}$ coming from this endpoint lies outside
of the $\theta$-interval occupied by~$D\cap U$, where $U$ is a small neighborhood of
the endpoint (see Fig.~\ref{conditiond4});
\begin{figure}[ht]
\center{\includegraphics{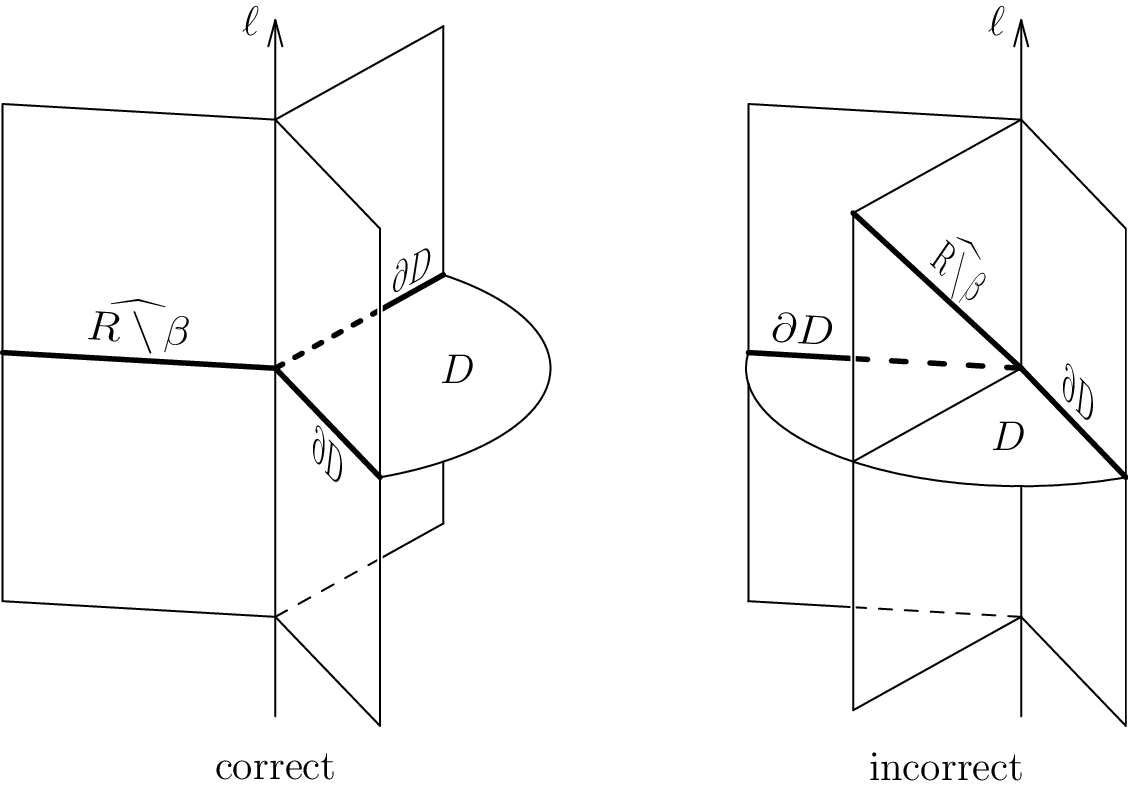}}
\caption{Relative position of a suitable disc and the link at the endpoints of~$\widehat\alpha$ and $\widehat\beta$}\label{conditiond4}
\end{figure}
\item\label{d5}
the foliation with singularities $\mathcal F$ defined on $D\setminus\ell$ by the restriction of the form~$d\theta$
has only simple saddle singularities inside $D$ and have no regular closed fibers;
\item\label{d6}
at points from~$D\cap\ell$ the coorientation of~$D$ coincides with the orientation of~$\ell$;
\item\label{d7}
there is exactly one positive (respectively, negative) half-saddle (see definition below) of the foliation~$\mathcal F$
at every arc of the form $\alpha\cap\page_t$ (respectively, $\beta\cap\page_t$);
\item\label{d8}
all saddles and half-saddles of the foliation~$\mathcal F$ lie in different pages.
\end{enumerate}

Most of these conditions are quite common for the technique we are going to explore,
see~\cite{BM1,BM2,Cro,Dyn}, and we discuss them only briefly.
The new observation here is the consistency of Condition~(D\ref{d7}) with the others,
which holds when constructing the disc~$D$ as well as during its subsequent simplification.

A disc~$D$ satisfying~(D1--D8) will be called \emph{suitable}. Now we proceed with its construction.

Recall that, by the bypass definition, in the link defined by the diagram~$(R\hm\setminus\beta)\cup\alpha\cup(\alpha\cup\beta)^\nearrow\cup(\alpha\cup\beta)^\swarrow$,
the components represented by
$(\alpha\cup\beta)^\nearrow\cup(\alpha\hm\cup\beta)^\swarrow$ 
are unknotted and unlinked with the rest of the link,
and their linking number is equal to~$-b$.
Here~$X^\nearrow$ stays for the result of shifting
a set~$X$ by the vector~$(\varepsilon,\varepsilon)$, and
$X^\swarrow$ for that of shifting by the vector~$(-\varepsilon,-\varepsilon)$ with small enough~$\varepsilon>0$.

We also introduce notation $X^\nwarrow$ and $X^\searrow$ for the results of shifting of $X$ by $(-\varepsilon,\varepsilon)$
and $(\varepsilon,-\varepsilon)$, respectively, notation $X\nesw$ for $X^\nearrow\cup X^\swarrow$,
and $X\nwse$ for $X^\searrow\cup X^\nwarrow$.

We claim that the two components presented by $\alpha\nesw\cup\beta\nwse$
in the link
$(R\setminus\beta)\cup\alpha\cup\alpha\nesw\cup\beta\nwse$
are unlinked with the rest of the link and with each other.

Indeed $\alpha\nesw\cup\beta\nwse$ is obtained from $\alpha\nesw\cup\beta\nesw$
by $b$ forbidden commutations of vertical edges originating from
edges of~$\beta$, each increases the linking number by one, see Fig.~\ref{shiftlinking}.
\begin{figure}[hr]
\center{\includegraphics{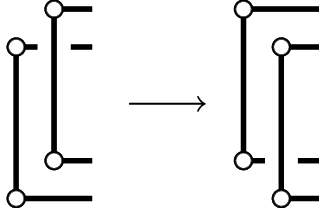}}
\caption{A ``forbidden commutation'' increases the linking number by one}\label{shiftlinking}
\end{figure}

These changes do not exchange edges of the altered components with edges
of the diagram $(R\setminus\beta)\cup\alpha$, hence, each
of the modified components remain unlinked with $(R\setminus\beta)\cup\alpha$.

Now we proceed from rectangular diagrams to arc presentations.
The path $\widehat{\alpha^\nearrow}$ is obtained from~$\widehat\alpha$ by a small rotation around
the binding line~$\ell$ in positive direction (the one in which $\theta$ grows) and a small upward shift, see Fig.~\ref{arcshift}.
\begin{figure}[ht]
\center{\includegraphics{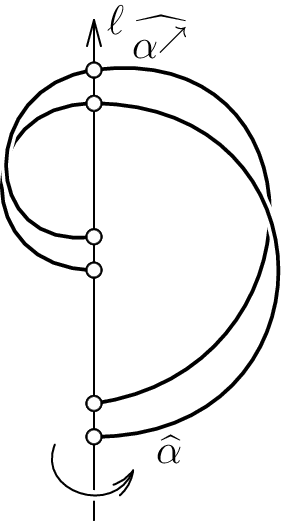}}
\caption{Paths $\widehat\alpha$ and $\widehat{\alpha^\nearrow}$}\label{arcshift}
\end{figure}

Similarly, $\nwarrow$, $\swarrow$, and $\searrow$ will denote a negative rotation combined with
an upward shift, a negative rotation combined with a downward shift, and a positive rotation combined
with a negative shift, respectively.

Denote by $S$ a narrow band that spans the link
$\widehat{\alpha\nesw}\cup\widehat{\beta\nwse}$
whose core line coincides with the trivial knot~$\widehat{\alpha\cup\beta}$.
The band~$S$ is composed of strips each of which binds to the binding line at the ends and is twisted a half turn
as shown in Fig.~\ref{strirps}.
\begin{figure}[ht]
\center{\includegraphics{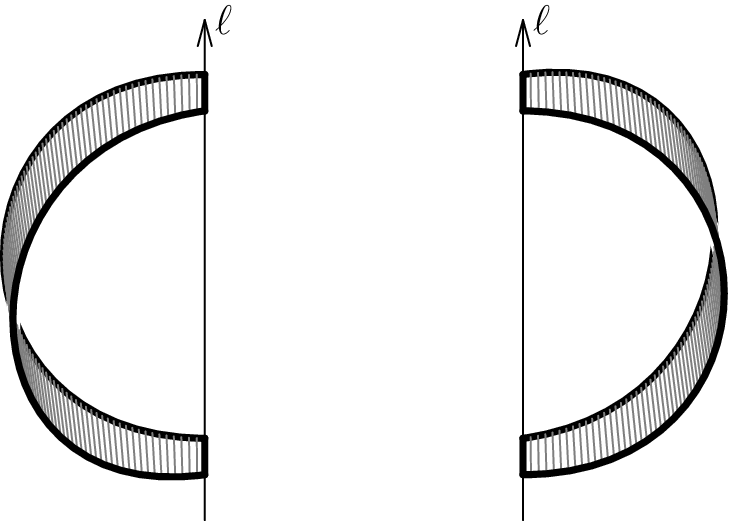}}
\caption{Strips forming the band $S$}\label{strirps}
\end{figure}

The boundary of~$S$ forms a trivial link of two components each of which---and we see in a moment
that their union, too---is unlinked with~$(\widehat{R\setminus\beta})\cup\widehat\alpha$,
and the core $\widehat{\alpha\cup\beta}$ of~$S$ is spanned by an embedded disc whose
interior is disjoint from~$\widehat R$.
In other words, topologically we have the picture shown in Fig.~\ref{bandS}.
\begin{figure}[ht]
\center{\includegraphics{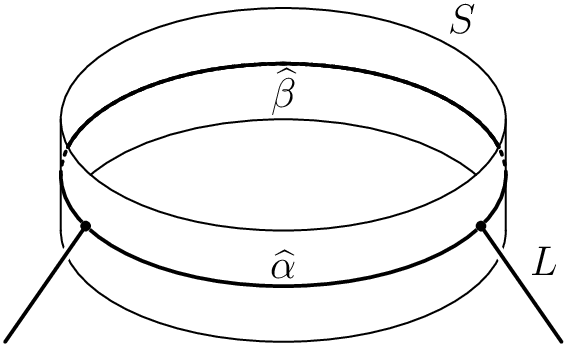}}
\caption{The band $S$ viewed topologically}\label{bandS}
\end{figure}

It means that we can
find a spanning disc for the unknot~$\widehat{\alpha\cup\beta}$ such that
the interior of~$D$ be still disjoint from~$\widehat R$, and $D$ be orthogonal to~$S$ at the boundary.
Additionally, we can ensure that the disc $D$ and the curve~$\widehat{R\setminus\beta}$
approach each of the endpoints of~$\widehat\alpha$ from opposite sides of~$S$, and also
that the general position requirements~(D\ref{d2}) and (D\ref{d3}) are fulfilled.
Conditions~(D\ref{d1}) and~(D\ref{d4}) are satisfied by construction.

Now we consider the foliation $\mathcal F$ defined by the closed $1$-form~$d\theta$ on~$D$.
This foliation is not defined at intersection points of~$D$ with~$\ell$
and has singularities at the tangency points of $D$ with pages~$\page_t$.
By general position argument we may assume that all singularities are of Morse type.

Since $D$ is orthogonal to~$S$ along $\partial D$, there are singularities of~$\mathcal F$
at~$\partial D$.  Namely, since each of the strips making up the band~$S$ and resting on~$\ell$
is twisted a half turn, there is a \emph{half-saddle} at each corresponding
segment of the boundary, see Fig.~\ref{halfsaddle}.
\begin{figure}[ht]
\center{\includegraphics{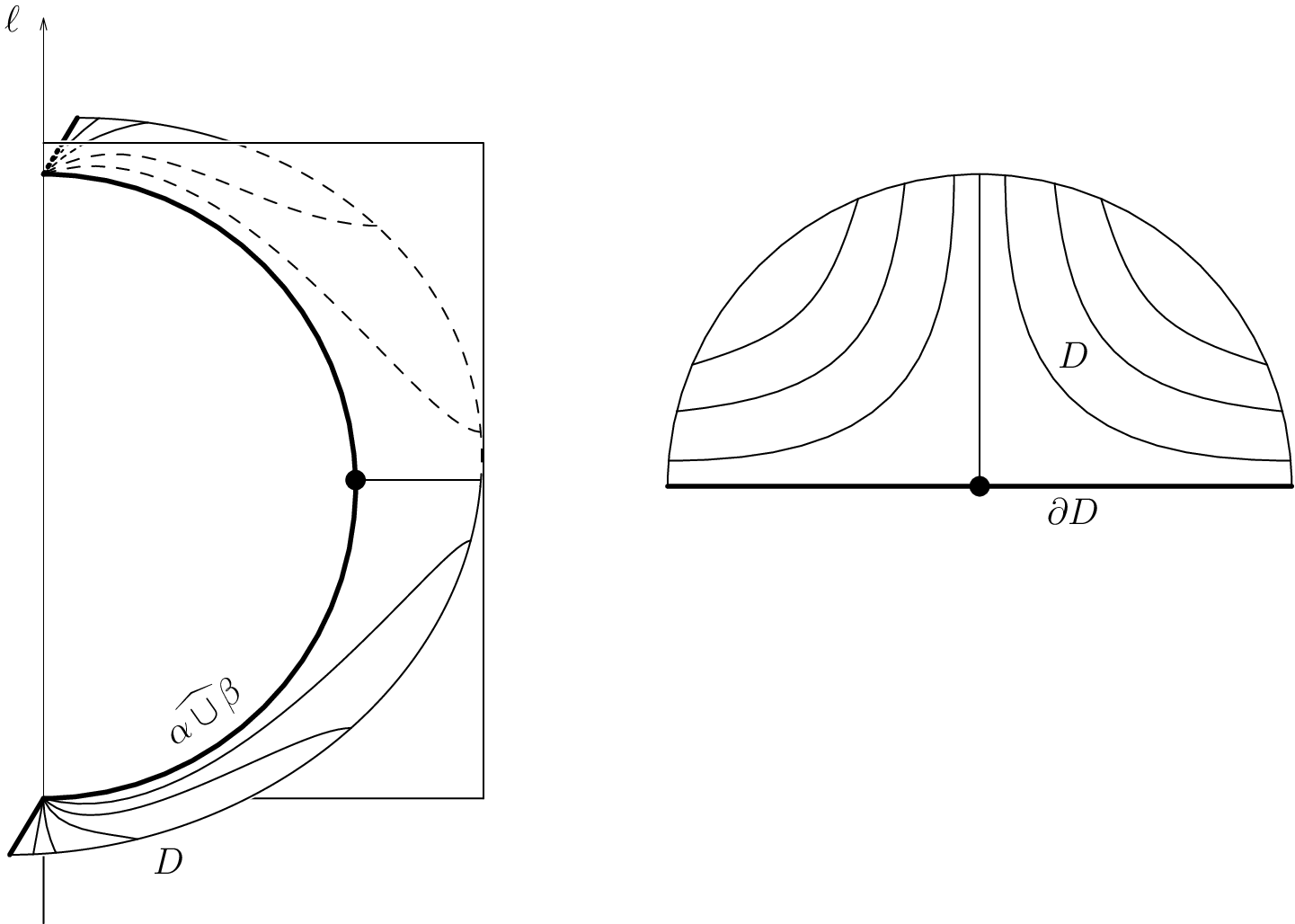}}
\caption{Half-saddle}\label{halfsaddle}
\end{figure}

With every saddle or half-saddle we associate a sign ``$+$'' or ``$-$'' depending
on whether $\theta$ increases or decreases in the direction of the
coorientation of~$D$ at the (half-)saddle.

Intersection points of $D$ with the binding line will be
called \emph{vertices} of~$D$, and we will distinguish \emph{boundary} and \emph{internal} vertices
in accordance with the position of the vertex. To each vertex we also assign ``$+$'' or ``$-$''
indicating whether or not the coorientation of~$D$ at the vertex coincide with the orientation of~$\ell$.

Since the disc boundary twists a half turn between any two boundary vertices,
one can see that all boundary vertices will have the same sign. We choose
the coorientation of~$D$ so as to make them all positive. This ensures fulfilling of condition~(D\ref{d6}).

Moreover, it can be seen from the construction above that condition~(D\ref{d7}) is then
automatically satisfied.

Validity of condition~(D\ref{d8}) is achieved by a small perturbation of the disc~$D$.
It remains to take care of condition~(D\ref{d5}), which means the absence of pole type
singularities and closed regular fibers of the foliation~$\mathcal F$. This is done in a standard way.

Namely, we remove from $D$ the union of all discs bounded by closed
fibers of $\mathcal F$.
As a result holes appear in~$D$ around each of which the foliation looks as shown in Fig.~\ref{hole}.
\begin{figure}[ht]
\center{\includegraphics{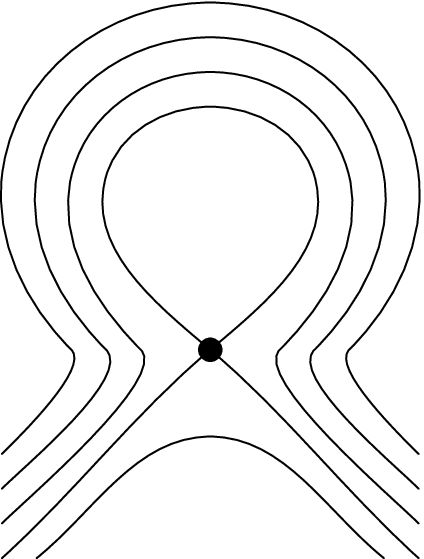}}
\caption{Foliation $\mathcal F$ around a hole}\label{hole}
\end{figure}

The boundary of each hole consists of a separatrix forming a loop. It is contained by whole
in a single page, and bounds a disc in it. By gluing up the hole by this disc and deforming the surface
slightly we get rid of the singularity.
The only type of singularity that the foliation~$\mathcal F$ can have inside~$D$ after
that is a simple saddle, and no separatrix forms a loop.

Every regular leave of the foliation~$\mathcal F$ has the form of an arc joining
two vertices of opposite sign. There are four (respectively, three) separatrices
attached to every saddle (respectively, half-saddle), and the other ends of them
approach vertices. So, the whole disc~$D$ with foliation~$\mathcal F$
can be cut along regular leaves into tiles where the foliation looks as shown in Fig.~\ref{elempieces}.
\begin{figure}[ht]
\center{\includegraphics{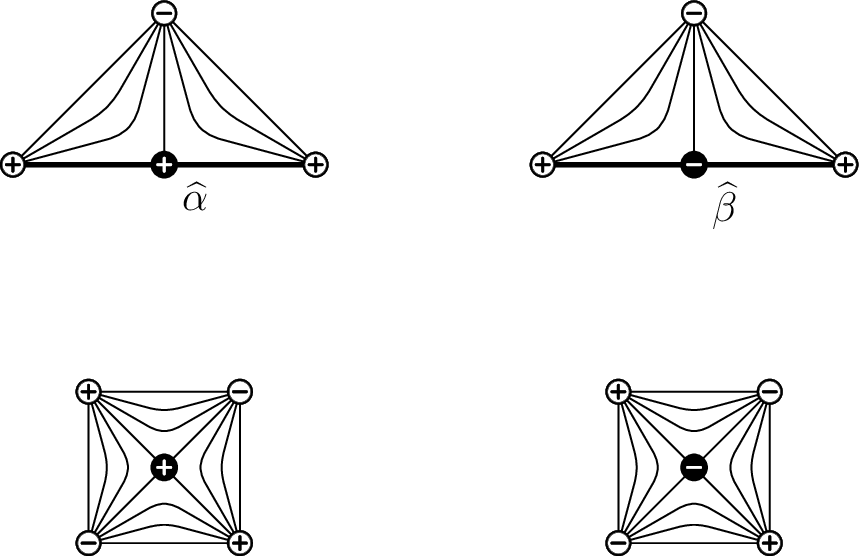}}
\caption{The foliation $\mathcal F$ is composed of such patterns}\label{elempieces}
\end{figure}

The union of all separatrices of~$\mathcal F$ cuts the disc~$D$ into parts that we call \emph{cells} of~$\mathcal F$.
Each cell is filled by regular leaves and has two vertices of opposite signs and two (half-)saddles whose signs may be
arbitrary, at the boundary, see Fig.~\ref{cellpic}.

\begin{figure}[ht]
\center{\includegraphics{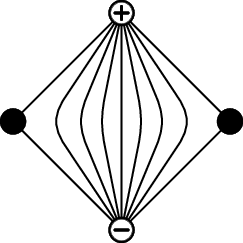}}
\caption{A cell of $\mathcal F$}\label{cellpic}
\end{figure}

\subsection{The induction}\label{induction}
It will be handy for as to call saddles and half-saddles in the sequel just saddles,
and in order to distinguish saddles in the interior of ~$D$ we will call them \emph{internal}.

Thus, we have a tuple of objects $(R,\alpha,\beta,D)$ in which $\alpha$ is a bypass of weight~$b$ and length~$a$
with ends at horizontal edges of a rectangular diagram~$R$, $\beta$ is the bypassed path of length~$b$,
and $D$ is a suitable disc for~$(\widehat R,\widehat\alpha)$. We aim at showing that
one can apply $b$ successive elementary simplifications to~$R$ so as to obtain
a diagram Legendrian equivalent to~$(R\setminus\beta)\cup\alpha$.
We do it by induction in the tuple~$(a+b,c,d)$, where $c$ is the number of internal vertices of~$D$,
and $d$ is the number of internal saddles that do not lie at bridges (see below).
Triples $(a+b,c,d)$ are ordered lexicographically.

By \emph{a bridge} we call an arc in~$D$ connecting two boundary vertices that is composed
of two separatrices, see Fig.~\ref{bridge}.

\begin{figure}[ht]
\center{\includegraphics{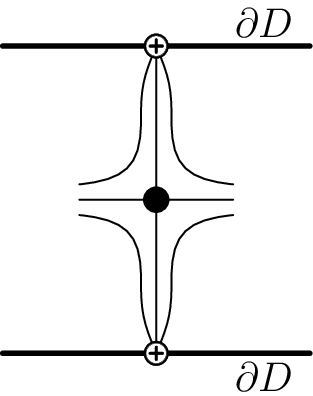}}
\caption{A bridge}\label{bridge}
\end{figure}

The induction base corresponds to $a=b=1$, $c=1$, $d=0$. For the induction step, in all other cases, we
modify the tuple~$(R,\alpha,\beta,D)$ into another tuple~$(R',\alpha',\beta',D')$ with the same properties
(for which the lengths of~$\alpha'$ and $\beta'$, the number of internal vertices of~$D'$,
and the number of internal saddles located outside bridges are denoted by~$a'$, $b'$, $c'$, and $d'$, respectively)
so that one of the following cases occur:
\begin{enumerate}
\def\labelenumi{(E\theenumi)}
\item
$R\mapsto R'$ is a type~II elementary simplification,
the transformation~$(R\setminus\beta)\cup\alpha\hm\mapsto(R'\setminus\beta')\cup\alpha'$
preserves the Legendrian type of the diagram, and we have $a'=a$, $b'=b-1$;
\item
$R'$ is obtained from $R$ by commutations and cyclic permutations,
$(R\setminus\beta)\cup\alpha\hm\mapsto(R'\setminus\beta')\cup\alpha'$ is a type~I elementary simplification,
and we have $a'=a-1$, $b'=b$;
\item
$R'\cup\alpha'$ is obtained from $R\cup\alpha$ by commutations and cyclic permutations, and we have
$a'=a$, $b'=b$ and either $c'=c-2$ or $c'=c$, $d'=d-1$.
\end{enumerate}
Whether one of these modifications can be made is determined solely
from the existence of certain patterns in the foliation~$\mathcal F$, where the signs of vertices and saddles are also
taken into account. Why this is so and how~$R,\beta,\alpha$ are changed is discussed in subsequent
subsections. Here we only describe how the foliation~$\mathcal F$ changes and show that at least
one of the patterns of~$\mathcal F$ that enables the induction step is always present.

By \emph{the star} of a vertex of~$D$ (internal or boundary one) we call the union of all regular
leaves approaching this vertex.
\emph{The valence} of a vertex is the number of separatrices having it as an endpoint.

\emph{Rearranging of saddles}, which may be ordinary, i.e.\ internal, saddles or half-saddles at the boundary~$\partial D$,
is possible once two saddles of the same sign appear at the boundary of the same cell. The change of the foliation~$\mathcal F$
is shown in Fig.~\ref{rearrange1}, \ref{rearrange2}, \ref{rearrange3}.

\begin{figure}[ht]
\center{\includegraphics{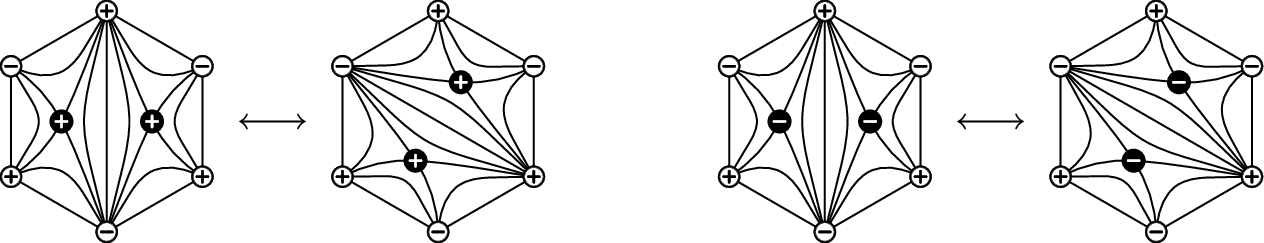}}
\caption{Rearranging of internal saddles}\label{rearrange1}
\end{figure}

\begin{figure}[ht]
\center{\includegraphics{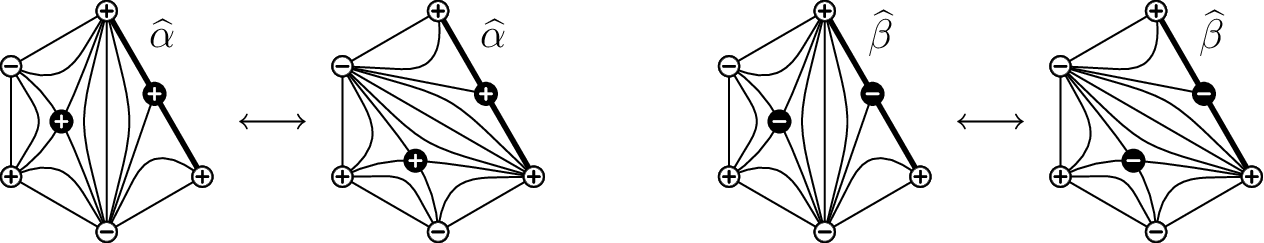}}
\caption{Rearranging an internal saddle and a half-saddle}\label{rearrange2}
\end{figure}

\begin{figure}[ht]
\center{\includegraphics{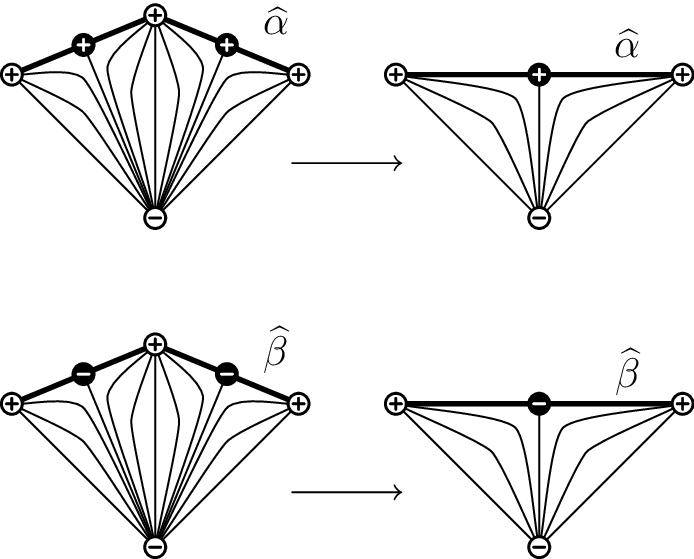}}
\caption{Rearranging two half-saddles}\label{rearrange3}
\end{figure}

If the two saddles being rearranged are negative half-saddles, then case~(E1) occurs,
if they are positive half-saddles, then so does case~(E2).
If at least one of the saddles being rearranged is an internal one,
then the diagram~$R\cup\alpha$ undergoes commutations and cyclic permutations,
and the numbers~$a$, $b$, $c$ do not change.
We will use this in a special situation which will lead to
case~(E3).

\emph{Smoothing out a wrinkle} is possible once an internal $2$-valent vertex appear, and one of the following
two conditions hold:\\
(i)~the $2$-valent vertex is connected by a regular leaf with another internal vertex (Fig.~\ref{wrinkleinside}),
then case~(E3) takes place;\\
(ii)~the star of the $2$-valent vertex contains exactly one half-saddle (Fig.~\ref{wrinkleatboundary}),
then we have case~(E1) for a negative half-saddle and case~(E2) for a positive one.

\begin{figure}[ht]
\center{\includegraphics{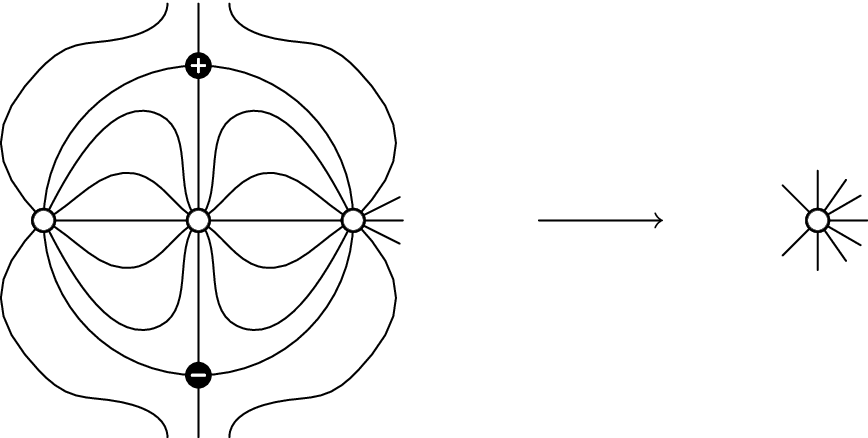}}
\caption{Smoothing out a wrinkle inside~$D$}\label{wrinkleinside}
\end{figure}

\begin{figure}[ht]
\center{\includegraphics{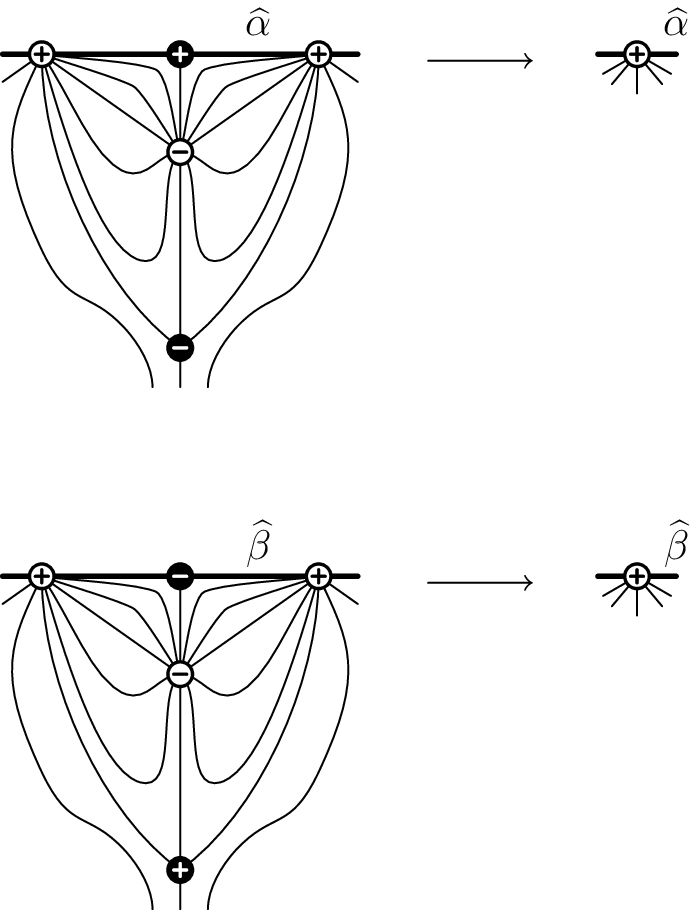}}
\caption{Smoothing out a wrinkle at the boundary of $D$}\label{wrinkleatboundary}
\end{figure}

Now we show that the inequality $a+b>2$ implies that at least one of the variants (E1), (E2), or (E3) of the induction step is applicable.

If there are no bridges in~$D$ we set~$D_0=D$. Otherwise, among the parts into which the bridges cut
the disc~$D$ there are at least two discs whose boundary contains just one bridge.
Such discs will be called \emph{terminal}. Let us mark all endpoints of bridges and those of~$\widehat\alpha$.
The boundary of at least one terminal disc will contain no more than three marked vertices at the boundary.
We take this disc as~$D_0$ and consider the restriction of~$\mathcal F$ to~$D_0$.

Let~$V_k$ be the number of internal vertices of~$D_0$ of valence~$k$, and
$B_k$ the number of $k$-valent vertices at the boundary~$\partial D_0$ (the valence is determined with respect to~$D_0$).
By construction, the vertices and half-saddles at~$\partial D_0$ follow in the alternate order
(the saddle at the bridge that cuts out~$D_0$ from $D$ is regarded as a half-saddle for~$D_0$).
Therefore, there are~$\sum_kB_k$ half-saddles at~$\partial D_0$, and the number of saddles inside~$D_0$ is one
less than that of vertices, i.e. $\sum_kV_k-1$,
by the Euler characteristics argument. Notice also that there are no vertices of valence~$<2$ at the boundary as well as
inside the disc~$D_0$.

Now we count the number of separatrices in two ways, at the vertices and at the saddles.
There are three separatrices coming out from every half-saddle, and four from every internal one.
Therefore,
$$3\sum_kB_k+4\Bigl(\sum_kV_k-1\Bigr)=\sum_kkB_k+\sum_kkV_k.$$
Rearranging the summands gives
\begin{equation}\label{counting}
B_2+2V_2+V_3=4+\sum_{k\geqslant3}(k-3)B_k+\sum_{k\geqslant4}(k-4)V_k\geqslant4.
\end{equation}

If we have $V_2>0$, then there is a two-valent vertex inside~$D_0$. If at least one of the other vertices
in its star is an internal one, we can smooth out a wrinkle inside~$D$.
If both other vertices in the star lie at $\partial D$, then $D_0$
is the whole star. Then if we have~$D_0\ne D$, then a wrinkle at the boundary of~$D$ can be smoothed out.
Finally, if we have~$D_0=D$, then the case~$a=b=c=1$, $d=0$ occurs, which is the induction base.

Now let $V_2=0$, $V_3>0$ hold, i.e. assume that there is a three-valent vertex~$P$ inside~$D_0$.
Its star contains two saddles of the same sign, to which a rearranging of saddles is applicable.
If both saddles lie at~$\partial D$ the rearranging results in reducing one boundary vertex. Otherwise
the numbers~$a$, $b$, and $c$ are preserved, but~$P$ turns into a two-valent vertex.

In the latter case, both vertices in the star of~$P$ distinct from~$P$ may lie at the boundary~$\partial D$. 
This means that the rearranging results in appearing a new bridge, and we have~$d'=d-1$.
Otherwise, we can smooth out a wrinkle inside~$D$. In both cases we have
induction step~(E3).

If $V_2=V_3=0$ hold, then~\eqref{counting} implies $B_2\geqslant4$. By construction there are at most three
vertices at the boundary of~$D_0$ that are endpoints of a bridge or of the path$\widehat\alpha$,
hence at least one of the two-valent boundary vertices is not such.
It is internal for either~$\widehat\alpha$ or $\widehat\beta$ and is two-valent
with respect to~$D$, not only $D_0$.
Rearranging of the saddles in its star results in the reduction of this vertex.

Now we describe in detail the manipulations that allow for the induction step
in all the cases mentioned above.

\subsection{Rearranging of saddles} This trick is applicable whenever there is
a cell~$\Delta$ in~$D$ with two saddles of the same sign at the boundary.
Let~$S_1$ and~$S_2$ be those saddles, and $\page_{\theta_1}$,
$\page_{\theta_2}$ be the pages containing them, with $0<\theta_1<\theta_2<2\pi$ and the disc~$\Delta$
lying in the sector~$\theta_1\leqslant\theta\leqslant\theta_2$ (we can always achieve this
by rotating the whole construction around the binding line, as the result of which the diagram~$R\cup\alpha$
may undergo only a few cyclic permutations), see Fig.~\ref{saddleattraction}.

\begin{figure}[ht]
\center{\includegraphics{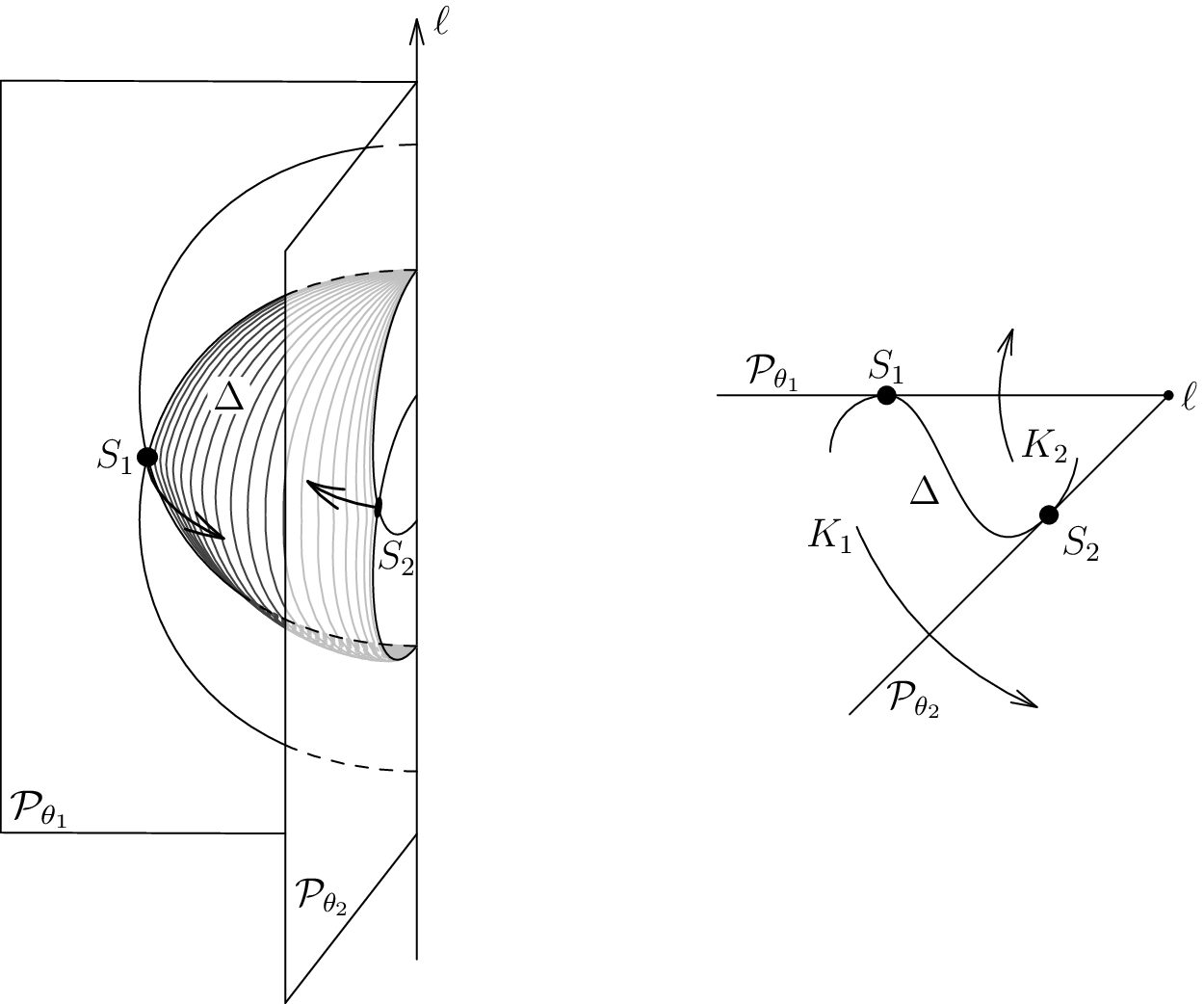}}
\caption{Pulling saddles of the same sign toward each other}\label{saddleattraction}
\end{figure}

Denote this sector by~$K$. The disc~$\Delta$ cuts it into two parts, which we denote
by~$K_1$ and $K_2$. The coincidence of signs of the saddles~$S_1$ and~$S_2$ means
that the separatrices coming out from them and not lying at the boundary of~$\Delta$
lie in~$K$ at different sides of~$\Delta$. By choosing a proper numeration of~$K_1$ and~$K_2$
we may ensure that the separatrices coming out from~$S_i$ lie at the boundary of~$K_i$,
$i=1,2$, which is assumed in the sequel.

If there are arcs of~$\widehat R\cup\widehat\alpha$ or saddles of~$\mathcal F$ inside~$K_1$,
then by rotating the arcs around~$\ell$ in the positive direction and deforming the disc~$D$ one can push all them out of~$K$ into the 
region~$\theta_2<\theta<\theta_2+\varepsilon$ with small enough~$\varepsilon>0$. Similarly,
arcs and saddles that got inside~$K_2$ can be pushed into the sector~$\theta_1-\varepsilon<\theta<\theta_1$
by a negative rotation and a deformation of the disc~$D$. As a result, the corresponding
$\Theta$-diagram~$R\cup\alpha$ may undergo only commutations of vertical edges, whereas
the combinatorial structure of the foliation~$\mathcal F$ stays unchanged.

In this way we remove all possible obstructions to pulling the saddles~$S_1$ and~$S_2$ 
toward each other. If both saddles are internal, then one can collapse the cell~$\Delta$ completely
by deforming the disc~$D$, thus producing a ``monkey saddle'', which can be resolved
into a pair of simple saddles in three different ways. Fig.~\ref{monkey3cases} shows
subsequent sections of the altered part of the disc~$D$ by pages from the sector~$K$,
for all three cases.

\begin{figure}[ht]
\center{\includegraphics{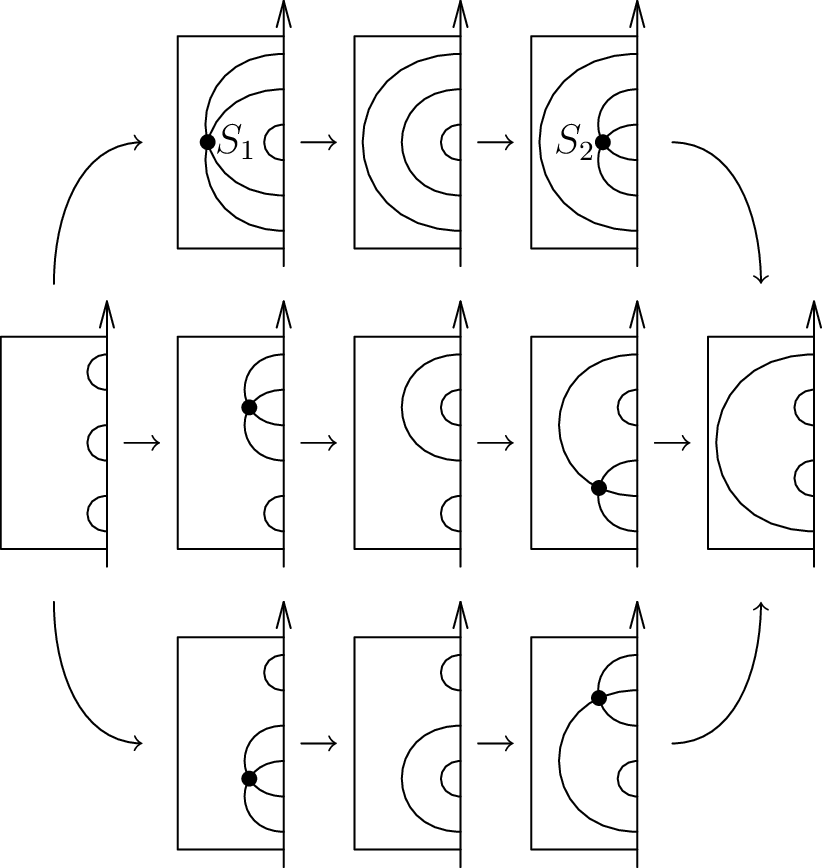}}
\caption{Three ways to resolve a ``monkey saddle''}\label{monkey3cases}
\end{figure}

If exactly one of the saddles~$S_1$, $S_2$ is a half-saddle, $S_2$, say, then only two of the three
resolutions of the ``monkey saddle'' are left, see Fig.~\ref{monkayatboundaryresolutions}.

\begin{figure}[ht]
\center{\includegraphics{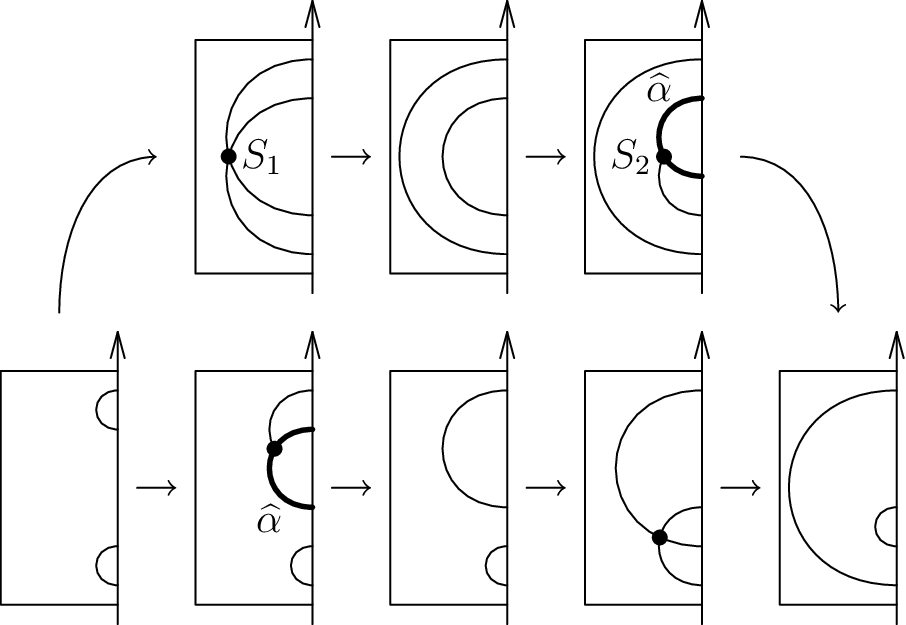}}
\caption{Two resolutions of a ``monkey saddle at the boundary}\label{monkayatboundaryresolutions}
\end{figure}

If both~$S_1$ and $S_2$ are half-saddles, then after the procedure described above
of pushing arcs and saddles off the sector~$K$, only two arcs from~$\widehat{\alpha\cup\beta}$
remain in~$K$, which contain the half-saddles~$S_1$ and $S_2$.
The vertical edges of the $\Theta$-diagram~$R\cup\alpha$ corresponding to these arcs
become neighboring, so, after a few commutations one can apply a destabilization, see Fig.~\ref{bondaryrearranginganddestabilization},
where on of the four possible cases of the position of $\Delta$ in the three-space is shown.
Other cases are obtained from this one by symmetries in a plane orthogonal to~$\ell$
and the bisector plane between pages $\page_{\theta_1}$ and $\page_{\theta_2}$.

\begin{figure}[ht]
\center{\includegraphics{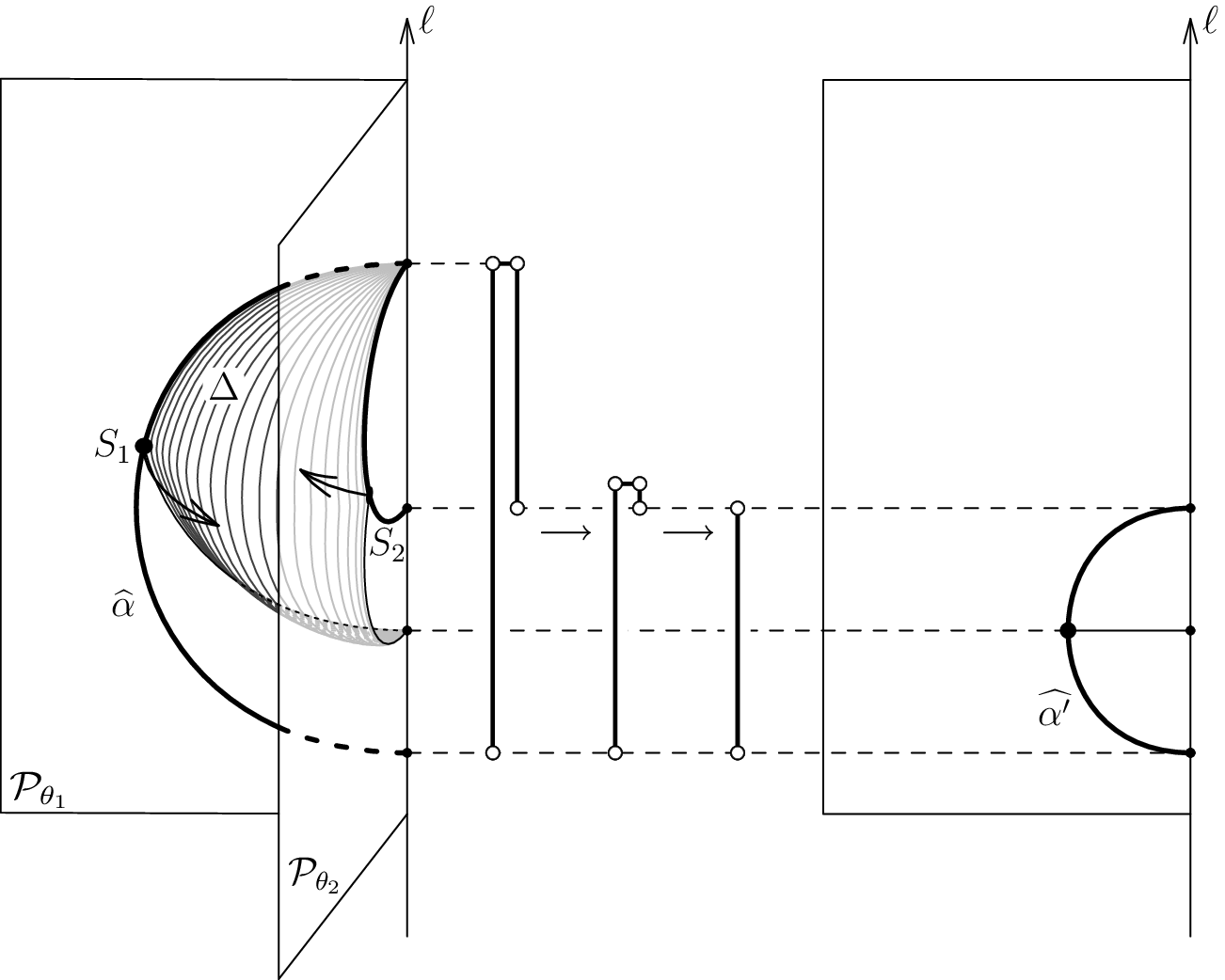}}
\caption{Pulling half-saddles to each other and a destabilization}\label{bondaryrearranginganddestabilization}
\end{figure}

Moreover, since the sings of the half-saddles~$S_1$ and~$S_2$ coincide,
they both lie either in~$\widehat\alpha$ or in~$\widehat\beta$. It can be checked directly that
in the former case, $\alpha$ undergoes a type~I destabilization, and in the latter case, $\beta$
undergoes a type~II destabilization.
Fig.~\ref{saddleattr&destab} shows how the sequence of sections of~$D$ by pages from the sector~$K$ changes.
The original sequence is shown in the upper row, and the sequence after rearranging of saddles in the lower one.
\begin{figure}[ht]
\center{\includegraphics{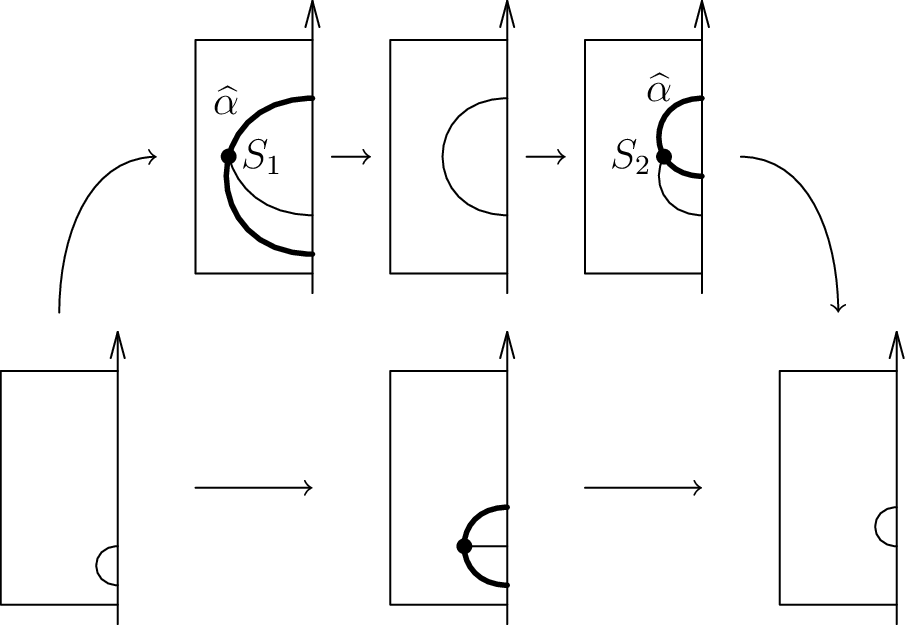}}
\caption{The change of the sections under rearranging of two boundary saddles}\label{saddleattr&destab}
\end{figure}

In each case of rearranging of saddles one can directly check that the saddle signs are preserved, and that the
the foliation~$\mathcal F$ changes as shown in Fig.~\ref{rearrange1}, \ref{rearrange2}, \ref{rearrange3}.
Every type of rearranging we needed to apply to the diagram~$R\cup\alpha$ only
cyclic permutations, commutations and---in the case of rearranging saddles at the boundary---a
destabilization. So, the requirements of the corresponding induction step~(E1), (E2), or (E3)
were satisfied.

\subsection{Smoothing out a wrinkle}
Let~$P_0$ be an internal two-valent vertex of~$D$. For now we do not wonder if the two other vertices
in its star, $P_1$ and~$P_2$, say, are internal or boundary ones. The star of~$P_0$ consists of two cells
that will be denoted by $\Delta_1$, $\Delta_2$ in accordance with the positions of vertices~$P_1$ and~$P_2$.
Let the $z$-coordinates of~$P_0$, $P_1$, and $P_2$ be $z_0$, $z_1$, and $z_2$, respectively.
Without loss of generality we may assume that these points follow on the binding line in this order: $z_1<z_0<z_2$. 
Indeed, if~$z_0$ does not lie between~$z_1$ and~$z_2$, we can add the~$\infty$ point to~$\mathbb R^3$
(and to the binding line), thus obtaining a three-sphere, and then remove an arbitrary point from the interval
between $z_1$ and $z_2$ disjoint from~$D\cup\widehat R$. For the corresponding $\Theta$-diagram~$\widehat R\cup\widehat\alpha$
this operation leads to a cyclic permutation of horizontal edges and, possibly, the ends of the path~$\alpha$.

If we have $z_2<z_0<z_1$, we can simply exchange notation for $P_1$ and~$P_2$.

Also without loss of generality we may assume that the cell~$\Delta_1$ is contained in the half-space~$\theta\in[0,\pi]$, 
and the cell~$\Delta_2$ in the half-space~$\theta\in[\pi,2\pi]$. Indeed, this always can be achieved
by applying a mapping of the form~$(\rho,\theta,z)\mapsto(\rho,f(\theta),z)$, where
$f$ is an appropriate degree~$1$ diffeomorphism of the circle into itself.
Such a mapping is isotopic to identity. The corresponding diagram~$R\cup\alpha$
may undergo only cyclic permutations of vertical edges.

\begin{figure}[ht]
\center{\includegraphics{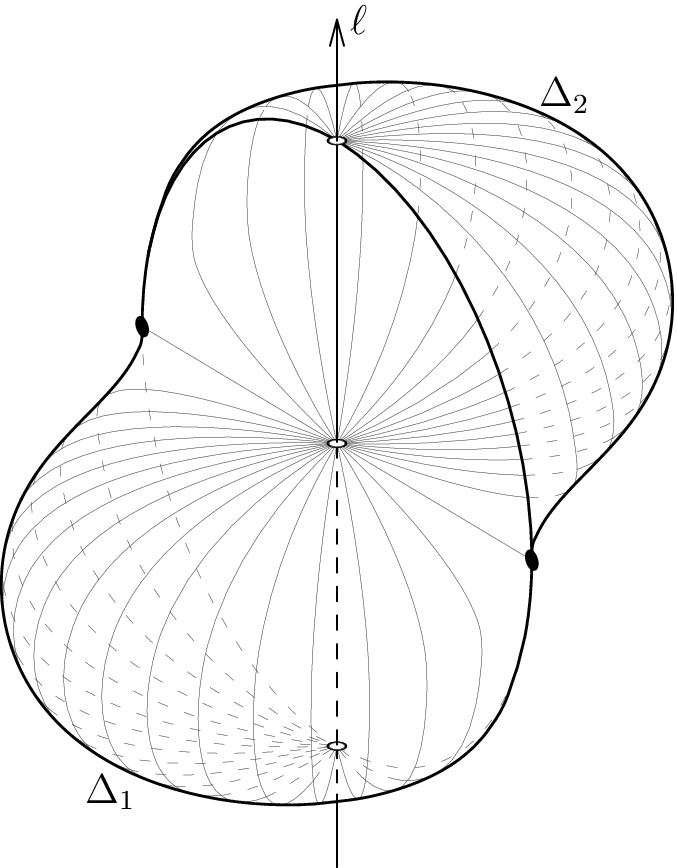}\put(-188,158){$\theta=0$}\put(-35,80){$\theta=\pi$}}
\caption{Position of a wrinkle in the three-space}\label{wrinkleinspace}
\end{figure}

The position of the cells~$\Delta_1$, $\Delta_2$ in the three-space is shown in Fig.~\ref{wrinkleinspace}.
Each of~$\Delta_1$, $\Delta_2$ cuts off a three-dimensional half-ball from the corresponding half-space.
We denote those half-balls by~$K_1$ and~$K_2$, respectively.

By deforming the disc~$D$ and rotating arcs from~$\widehat R\cup\widehat\alpha$ around
the binding line we can push all saddles of~$\mathcal F$ and arcs of $\widehat R\cup\widehat\alpha$ out of~$K_1$ and~$K_2$.
This will result only in commutations and cyclic permutations of vertical edges of the
corresponding diagram~$R\cup\alpha$.
The following rectangles will become clear of vertices of the diagram $R\cup\alpha$:
$$[0,\pi]\times[z_1,z_0]\quad\text{and}\quad[\pi,2\pi]\times[z_0,z_2].$$

Now there are no topological obstruction to moving all vertices of the disc~$D$, of the link~$\widehat R$,
and of the path~$\widehat\alpha$ from the interval~$(z_1,z_0)$ to the interval~$(z_2,z_2+\varepsilon)$, 
and from the interval~$(z_0,z_2)$ to the interval~$(z_1-\varepsilon,z_1)$,
with small enough~$\varepsilon$, see Fig.~\ref{vertexmotionforwrinkleremoval}.
\begin{figure}[ht]
\center{\includegraphics{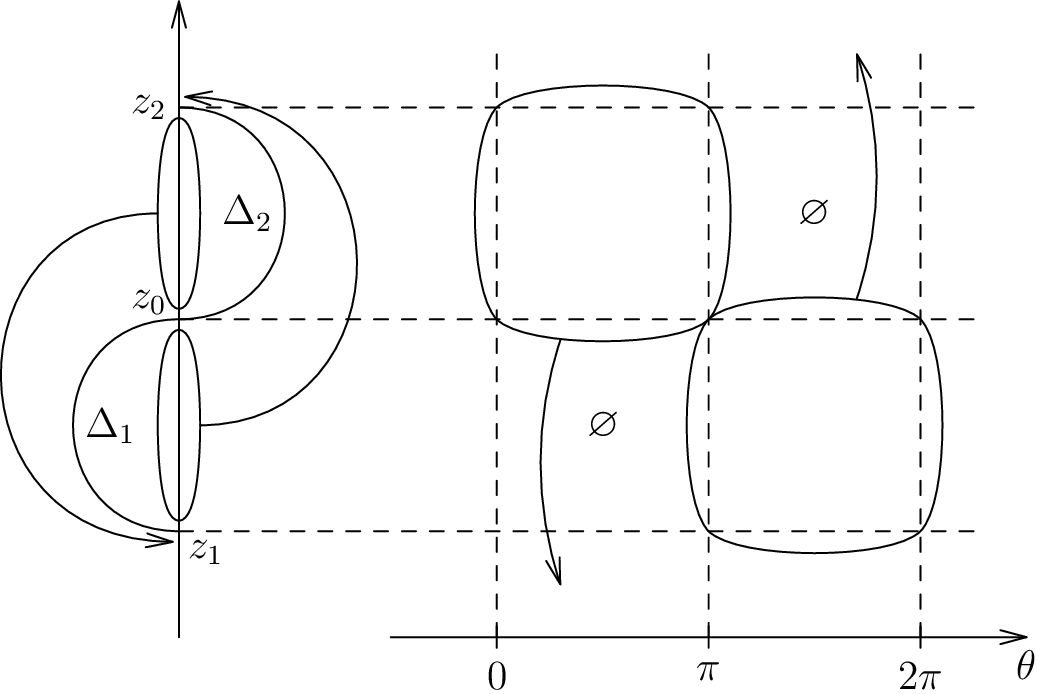}}
\caption{Changing vertex positions during smoothing out a wrinkle}\label{vertexmotionforwrinkleremoval}
\end{figure}
In each of the two portions of shifted vertices their relative order at the binding line is preserved.

Notice that no section of the form~$\page_t\cap(D\cup\widehat R)$ contains more than one arc
coming out from a single vertex of the disc~$D$ or the link~$\widehat R$.
In the present construction this is important to know for vertices~$P_0$, $P_1$, and~$P_2$
as otherwise there could be an obstruction to the vertex exchange described above.
The absence of pairs of arcs with a common end in pages~$\page_t$
is ensured by conditions~(D\ref{d3}) and~(D\ref{d4}) from the definition of a suitable disc.

Now there are no vertices of the disc~$D$ and the link~$\widehat R$ between~$P_1$ and~$P_2$ except
the vertex~$P_0$. Further actions depend on the position of the pattern~$\Delta_1\cup\Delta_2$
with respect to the boundary of~$D$.

If one of the vertices~$P_1$ or~$P_2$ is internal, then two vertices can be reduced
by a deformation of the disc, see Fig.~\ref{removewrinkleinside}, where it is assumed
that $P_1$ is an internal vertex.
This deformation can be chosen so as to realize the change
of the foliation~$\mathcal F$ shown in Fig.~\ref{wrinkleinside}.
\begin{figure}[ht]
\center{\includegraphics{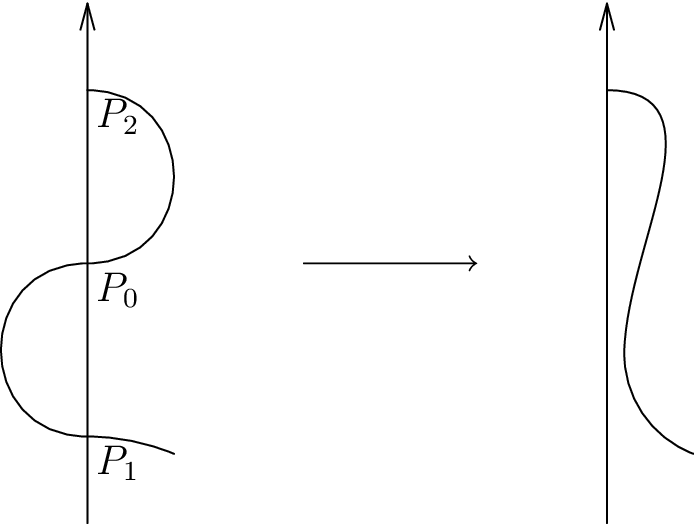}}
\caption{Removing a wrinkle}\label{removewrinkleinside}
\end{figure}

In this procedure the boundary of~$D$ is untouched, so, the disc remains suitable.

Now let both vertices~$P_1$ and $P_2$ lie at the boundary of~$D$.
We did not need in subsection~\ref{induction} to smooth out a wrinkle in the case when the star~$\Delta_1\cup\Delta_2$ of the vertex~$P_0$
was bounded by two bridges. We excluded this situation by choosing~$D_0$ among terminal discs.

Thus, if we have~$P_1,P_2\in\partial D$, then at least one of the two arcs with endpoints~$P_1,P_2$ that bound~$\Delta_1\cup\Delta_2$
is contained in~$\widehat{\alpha\cup\beta}$. First we consider the case when there is just one such arc
and it is contained in~$\widehat\alpha$, which means that it contains a positive half-saddle.
The considered positions of the cells $\Delta_1$ and $\Delta_2$ imply that
the positive half-saddle will be the one in the page~$\page_\pi$, see Fig.~\ref{wrinkleatboundaryinspace}.

\begin{figure}[ht]
\center{\includegraphics{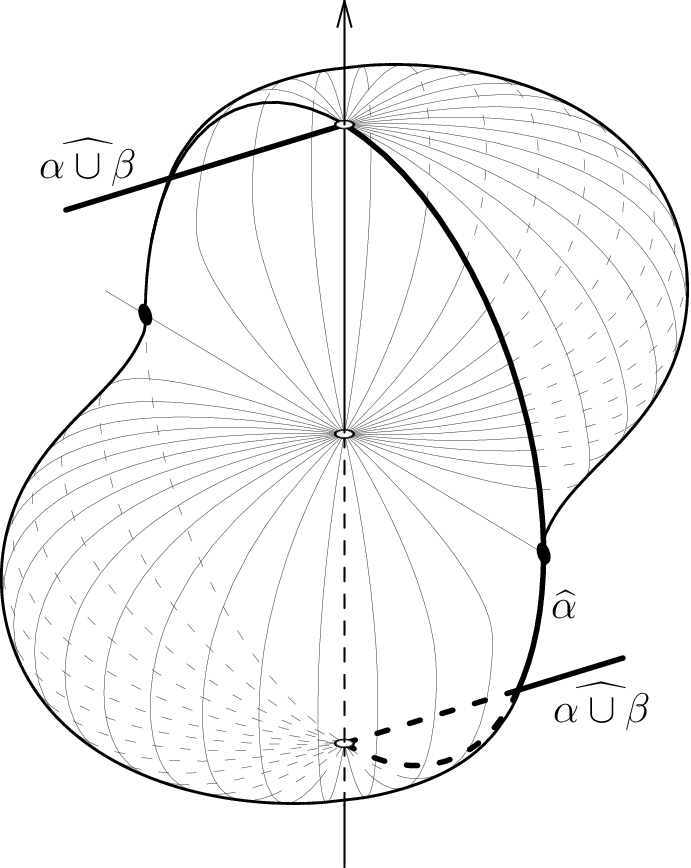}}
\caption{Relative position of a wrinkle and arcs of
$\widehat{\alpha\cup\beta}$}\label{wrinkleatboundaryinspace}
\end{figure}

In this case smoothing out a wrinkle is as follows. First we remove the interior of~$\Delta_1\cup\Delta_2$
with the part of the boundary contained in~$\widehat\alpha$. Then we collapse the segment~$[P_1,P_2]$
of the binding line to a point, and the disc that is cut off the page~$\page_0$ by the separatrices
lying at~$\partial(\Delta_1\cup\Delta_2)$ to a straight line segment, see Fig.~\ref{collapsing}.
\begin{figure}[ht]
\center{\includegraphics{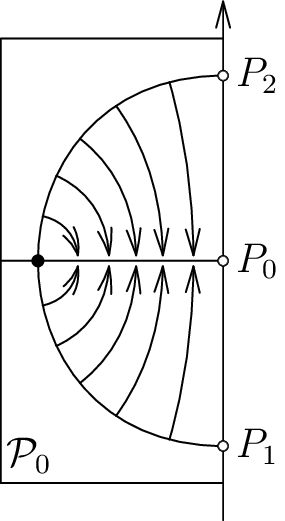}}
\caption{Collapsing a disc in the page~$\mathcal{P}_0$}\label{collapsing}
\end{figure}

In the corresponding rectangular diagram this results in collapsing of the vertical edge~$\pi\times[z_1,z_2]$, which is included
in~$\alpha$, see Fig.~\ref{collapsingedge}. There are no other vertices of~$R\cup\alpha$
in the strip~$\mathbb R\times[z_1,z_2]$, hence this transformation can be decomposed into a few commutations and a
destabilization provided that the path $\alpha$ contains at least two edges. Namely,
if the edge $z=z_1$ belongs to $\alpha$, then the short vertical edge $\pi\times[z_1,z_2]$
must be shifted to the right by using commutations until the edge $z=z_1$ becomes short,
and then a destabilization is applied. If this edge does not belong to $\alpha$,
then the short vertical edge must be shifted similarly to the left.

One can see from the relative position of the 
reduced edges
that the destabilization is of type~I.
One can also see that the corresponding change of the foliation~$\mathcal F$
has the form shown in Fig.~\ref{wrinkleatboundary}.
\begin{figure}[ht]
\center{\includegraphics{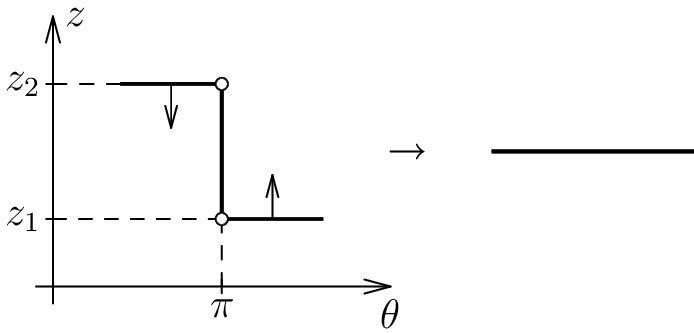}}
\caption{Change of the rectangular diagram~$\alpha\cup\beta$ under smoothing out a wrinkle at the boundary}\label{collapsingedge}
\end{figure}

If $\pi\times[z_1,z_2]$ is the only edge of the path~$\alpha$, a destabilization is impossible.
Suppose this situation do occur. Then the diagram~$\alpha\cup\beta$ still admits a type~I destabilization
that results in a diagram of the unknot of complexity equal to its Thurston--Bennequin number.
Due to relation~\eqref{complexityandtb} this contradicts Theorem~\ref{erltheo}.

Thus, the situation when~$\pi\times[z_1,z_2]$ is the only edge of~$\alpha$, which creates
an obstacle to smoothing out a wrinkle, cannot occur.

If there is exactly on arc of~$\widehat{\alpha\cup\beta}$ at the boundary of the disc~$\Delta_1\cup\Delta_2$
and it is contained in~$\widehat\beta$,
then a similar procedure allows to simplify the path~$\beta$. The length of~$\beta$,
which is at the same time the weight of the bypass~$\alpha$, decreases by one.
From the symmetry argument the path~$\beta$ undergoes a type~II simplification.

This time for smoothing out a wrinkle the path $\beta$ must have more than
one edge, which is again proved by contradiction. Namely, if
$\beta$ consists of a single edge, then we have $\tb(\alpha\cup\beta)=-1$,
and $\alpha\cup\beta$ admits a type~II destabilization, which results
in an unknot diagram with
$\tb=0$, a contradiction with Theorem~\ref{erltheo}.

In all cases of smoothing out a wrinkle, the diagram $R\cup\alpha$
undergoes only cyclic permutations, commutations and---in the case
of a wrinkle at the boundary---a destabilization, hence, the
requirements of the corresponding induction step (E1), (E2), or (E3) are satisfied.

Finally, the boundary of the disc~$\Delta_1\cup\Delta_2$ may coincide with~$\widehat{\alpha\cup\beta}$.
This means that each of the paths~$\widehat\alpha$ and~$\widehat\beta$ consists of a single arc,
and the disc~$\Delta_1\cup\Delta_2$ coincides with~$D$. This situation is discussed below.

\subsection{The induction base}
We keep settings from the previous section. As a result of the manipulations described there
we come to the situation when~$\alpha$ and~$\beta$ each have exactly one edge that are~$\pi\times[z_1,z_2]$ 
and~$0\times[z_1,z_2]$, respectively, and the ends of the path~$R\setminus\beta$ are located
inside the straight line segments~$[0,\pi]\times z_2$ and~$[\pi,2\pi]\times z_1$, see Fig.~\ref{inductionbasewrinkle}.
\begin{figure}[ht]
\center{\includegraphics{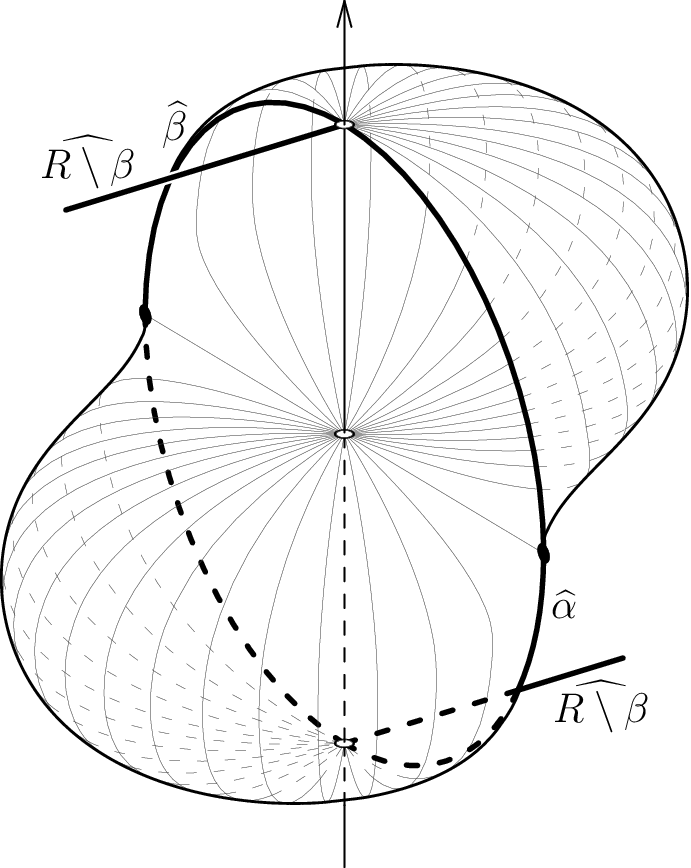}}
\caption{Relative position of the wrinkle and the arcs of~$\widehat R\cup\widehat\alpha$ for a disc
simplified as much as possible}\label{inductionbasewrinkle}
\end{figure}
\begin{figure}[ht]
\center{\includegraphics{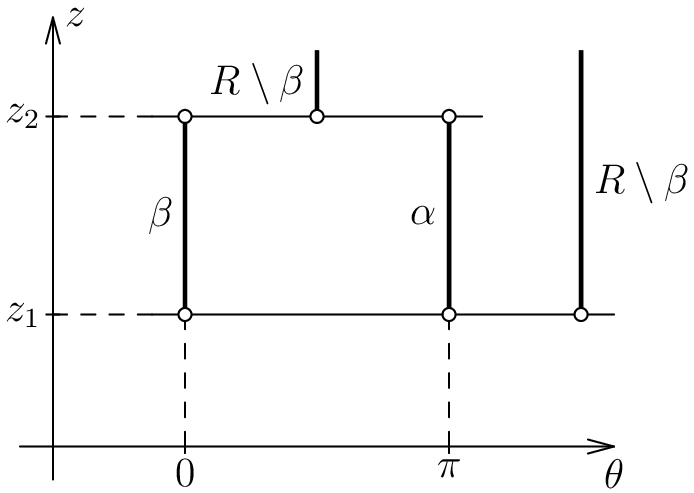}}
\caption{Simplification of $R\cup\alpha$ corresponding to a bypass of weight~$1$ and length~$1$}\label{basesimplification}
\end{figure}

In addition, there are no vertices of~$R\cup\alpha$ inside the strip~$\mathbb R\times[z_1,z_2]$, see Fig.~\ref{basesimplification}.
Clearly, one can apply a type~II simplification to~$R$ so as to get a diagram that
can also be obtained by a type~I simplification from~$(R\setminus\beta)\cup\alpha$,
which means it is Legendrian equivalent to~$(R\setminus\beta)\cup\alpha$.

The Key Lemma is proved.

\begin{proof}[Yet another proof of the monotonic simplification theorem for the unknot]
In the prof of Corollary~\ref{monosimpl} we used Eliashberg--Fraser's classification
theorem, which states much more that we would need if we don't apply
the Key Lemma as is, but repeat its proof with small (simplifying) modifications.

Let $K$ be a rectangular diagram of the unknot. Denote by
$a$ and $b$ the numbers $(-\tb(\overline K))$ and $(-\tb(K))$, respectively
According to Theorem~\ref{erltheo} we have $a,b>0$.
Recall that according to~\eqref{complexityandtb} the total number of vertical edges in $K$
equals~$a+b$. Therefore, the diagram~$K$ can be represented
as the union of two rectangular paths $\alpha$ and $\beta$ having
$a$ and $b$ vertical edges, respectively.

Now we just have to return to the beginning of Section~\ref{properdisc} and repeat
the procedure of construction and simplification of the disc~$D$ ignoring the presence of
the rectangular path $R\setminus\beta$ and everything related to it.
\end{proof}

\section{Applications to braids and transversal links}\label{braidsection}

\subsection{Birman--Menasco classes}
According to a well-known theorem by J.Alexander~\cite{A} any oriented link can be presented in the form
of a closed braid, an example is shown in Fig.~\ref{braidclosure}.
\begin{figure}[ht]
\center{\includegraphics{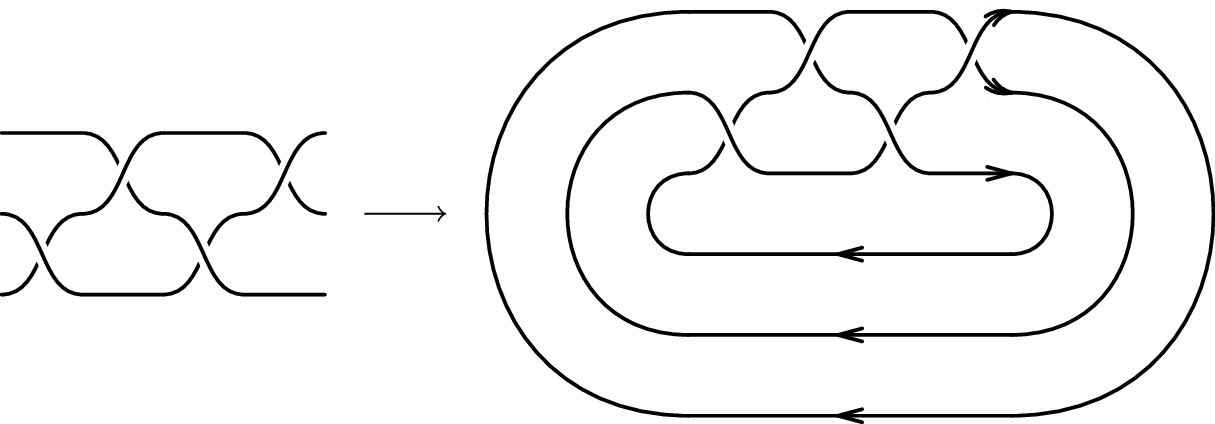}}
\caption{Closure of the braid $(\sigma_2^{-1}\sigma_1)^2$}\label{braidclosure}
\end{figure}
Another well-known theorem, due to Markov~\cite{mark,birman}, claims that the closures of two
braids are equivalent as oriented links if and only if they can be obtained from each other by
transformations that are now called \emph{Markov moves}. These moves include
conjugation in the standard group sense, \emph{stabilizations} and \emph{destabilization}
(see Fig.~\ref{braidstabilization}), which are described in algebraic terms as follows.
\begin{figure}[ht]
\center{\includegraphics{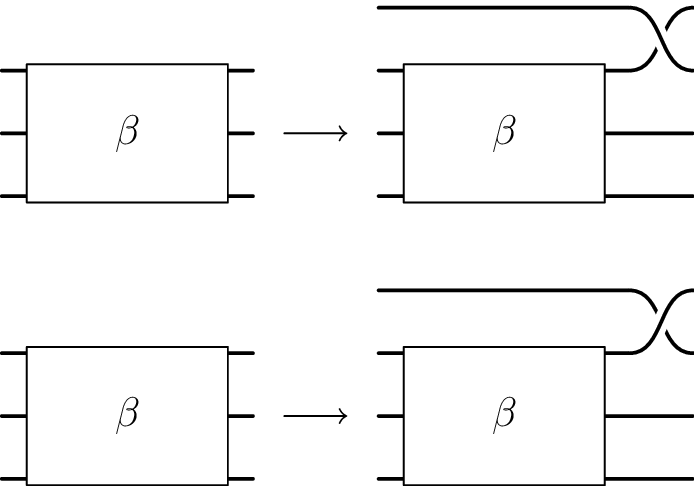}}
\caption{Stabilization applied to a braid $\beta$}\label{braidstabilization}
\end{figure}

We denote the $n$-strand braid group by~$B_n$, and the standard Artin generators
by~$\sigma_1,\ldots,\sigma_{n-1}$. The embedding~$\iota_n:B_n\rightarrow B_{n+1}$ is defined
by adding a free $(n+1)$th strand and is written tautologically in terms of generators:
$\sigma_i\mapsto\sigma_i$, $i=1,\ldots,n-1$.

In this notations a stabilization of a braid $\beta\in B_n$ is defined as the transition
$$\beta\mapsto\iota_n(\beta)\sigma_n^{\pm1}.$$
The stabilization is said to be \emph{positive} or \emph{negative} depending on the
power of the generator~$\sigma_n$ being used. A destabilization (positive or negative, respectively)
is defined as the inverse operation.

A ``monotonic simplification theorem'' is proved in~\cite{BM2} that states that every braid
whose closure is an unknot can be transformed into the trivial braid on one strand by using
moves that include conjugations, destabilizations, and \emph{exchange moves} introduced by J.Birman
and W.Menasco, which are defined algebraically as follows:
$$\beta_1\sigma_n\beta_2\sigma_n^{-1}\mapsto\beta_1\sigma_n^{-1}\beta_2\sigma_n,
\quad\text{where}\quad\beta_1,\beta_2\in\iota_n(B_n),$$
see Fig.~\ref{exchangemove}.
\begin{figure}[ht]
\center{\includegraphics{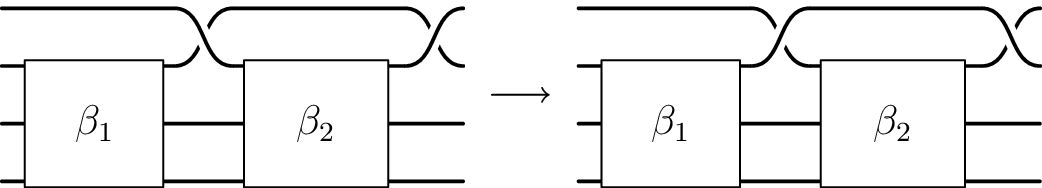}}
\caption{Exchange move}\label{exchangemove}
\end{figure}
The inverse transformation is obviously decomposed into conjugations and another exchange move.
(Strictly speaking, a larger family of moves is used in~\cite{BM2},
but those moves can be readily expressed through the ones defined above. This fact
was noticed by the authors of~\cite{BM2} later.)

\begin{defi}
By \emph{the Birman--Menasco class} of a braid~$\beta\in B_n$ we call the set
of all braids~$\beta'\in B_n$ that can be obtained from~$\beta$ by a sequence of conjugations and
exchange moves. This class will be denoted by~$\bm(\beta)$,
whereas the set of all Birman--Menasco classes in~$B_n$ will be denoted by~$\BM_n$.

We say that a class~$\mathcal B\in\BM_n$ admits \emph{a positive (respectively, negative) 
destabilization}~$\mathcal B\mapsto\mathcal B'$ if there is a positive (respectively, negative)
destabilization~$\beta\mapsto\beta'$ with~$\beta\in\mathcal B$ and~$\beta'\in\mathcal B'$. In this case we
also say that~$\mathcal B$ is obtained from~$\mathcal B'$ by a positive (respectively, negative) \emph{destabilization}.
\end{defi}

\begin{theo}\label{mainbraidtheo}
Let Birman--Menasco classes~$\mathcal B_1$ and~$\mathcal B_2$ define equivalent oriented links.
Then there exists a Birman--Menasco class~$\mathcal B$ that can be obtained from~$\mathcal B_1$
by a sequence of positive stabilizations and destabilizations, and from~$\mathcal B_2$ by that of negative ones.

The same is true with Birman--Menasco classes replaced by braid conjugacy classes.
\end{theo}

Before proceeding with the proof we recall the connection between rectangular diagrams and braids.

Throughout this section we deal only with \emph{oriented} rectangular diagrams of links. This means that
all their edges are oriented in a consistent way, namely so as to have an oriented link diagram. It is convenient
to specify the orientation by coloring vertices in black and white so that vertical edges be directed from
a black vertex to a white one, and horizontal ones vice versa. Rectangular diagrams
with such coloring are also called \emph{grid diagrams} in the literature.

In the case of oriented rectangular diagrams, one naturally distinguish four rather than two types of stabilizations
and destabilizations, which we will call \emph{oriented types}. Namely, each of the types~I and~II
is subdivided into two types: \Iright, \Ileft and \IIright, \IIleft, respectively, where
the arrow indicates the direction of the short horizontal edge that emerge from the stabilization.
All four types are demonstrated in Fig.~\ref{orientedstabilizationtypes} in terms of colored vertices.
\begin{figure}[ht]
\center{\includegraphics{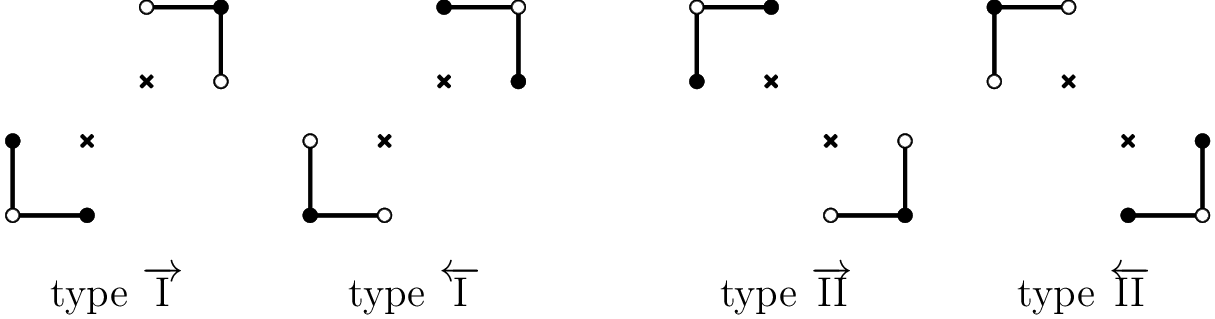}}
\caption{Oriented types of (de)stabilizations}\label{orientedstabilizationtypes}
\end{figure}

With every oriented rectangular diagram~$R$ we associate a braid~$\beta_R$ by the following rule
(see.~\cite{Cro,Dyn,NgTh}). Let all vertices of~$R$ be contained in the domain~$[0,1]\times\mathbb R$.
Replace each horizontal edge~$[x_1,x_2]\times y_0$ directed from right to left by the
pair of straight line segments~$[0,x_1]\times y_0$ and~$[x_2,1]\times y_0$. At every crossing
of vertical and horizontal segments the former are considered overpasses.
Now we slant vertical edges slightly and smooth out corners so as to obtain a braid diagram of~$\beta_R$,
see Fig.~\ref{fromgridtobraid}.
\begin{figure}[ht]
\center{\includegraphics{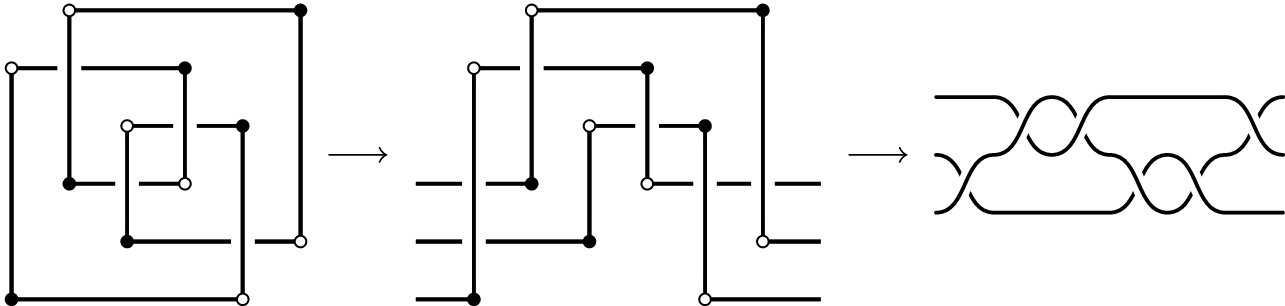}}
\caption{The braid associated with a rectangular diagram}\label{fromgridtobraid}
\end{figure}

\begin{prop}\label{beta(R)}
Any braid can presented in the form~$\beta_R$ for an appropriate rectangular diagram~$R$.

For any rectangular diagram~$R$, the corresponding oriented link is equivalent to the closure of the braid~$\beta_R$.

Let $R$ and $R'$ be oriented rectangular diagrams. Then the Birman--Menasco classes
of the braids~$\beta_R$ and~$\beta_{R'}$ coincide if and only if the diagrams~$R$ and~$R'$ can
be obtained from each other by a sequence of elementary moves not
including stabilizations and destabilizations of types~\Ileft and~\IIleft.
\end{prop}

For a proof see~\cite{NgTh}.

One can see that a \Ileft-stabilization of a diagram~$R$ causes a positive stabilization of the class~$\bm(\beta_R)$,
and a \IIleft-stabilization causes a negative one. Since \Iright- and \IIright-stabilizations preserve
the class~$\bm(\beta_R)$, we see that diagrams obtained from an oriented
rectangular diagram~$R$ by stabilizations of different oriented types cannot be
obtained from each other by cyclic permutations and commutations.

\begin{proof}[Proof of Theorem~\ref{mainbraidtheo}]
Let $R_1$ and $R_2$ be oriented rectangular diagrams such that we have~$\beta_{R_1}\in\mathcal B_1$ 
and~$\beta_{R_2}\in\mathcal B_2$. By construction the closures of the braids~$\beta_1$, $\beta_2$ are equivalent,
hence, by Proposition~\ref{beta(R)}, the diagrams~$R_1$ and~$R_2$ define equivalent oriented links.

By Theorem~\ref{legendriancombine} there exists a diagram~$R$ that is Legendrian equivalent
to~$R_1$ and such that $\overline R$ is Legendrian equivalent to~$\overline{R_2}$. The class~$\bm(\beta_R)$
is then a sought-for one.

Indeed, the diagram~$R$ can be obtained from~$R_1$ by a sequence of cyclic permutations, commutations and
type~I stabilizations and destabilizations. Among these moves only type~\Ileft\ stabilizations and destabilization
affect the Birman--Menasco class of the corresponding braid, causing positive stabilizations and destabilizations of the latter.

In the case of~$R_2$ the argument is the same with changing the type of stabilizations and destabilizations
from~I to~II for rectangular diagrams and from positive to negative for the corresponding braids.

The second statement of the theorem, which concerns braid conjugacy classes, follows from the first one
and an observation of J.Birman and N.Wrinkle~\cite{BW} that any exchange move can
be decomposed into a sequence consisting of conjugations, one stabilization, and one destabilization,
where the stabilization and destabilization are of the same sign \emph{fixed in advance}.
For positive stabilization and destabilization this decomposition looks as follows:
\begin{equation}\label{exchangeviastabilization}
\begin{array}{ccc}
\beta_1\sigma_n\beta_2\sigma_n^{-1}&&\\
\hbox to 0pt{\hss conjugation}\downarrow\\
\sigma_n\beta_1\sigma_n\beta_2\sigma_n^{-2}&&\beta_1\sigma_n^{-1}\beta_2\sigma_n\\
\hbox to 0pt{\hss stabilization}\downarrow&&\uparrow\hbox to0pt{conjugation\hss}\\
\sigma_n\beta_1\sigma_n\beta_2\sigma_n^{-2}\sigma_{n+1}&&\sigma_n^{-2}\beta_2\sigma_n\beta_1\sigma_n\\
\hbox to 0pt{\hss conjugation}\downarrow&&\uparrow\hbox to0pt{destabilization\hss}\\
\sigma_n^{-2}\sigma_{n+1}^{-1}\beta_2\sigma_n^{-2}\sigma_{n+1}\sigma_n\beta_1\sigma_n\sigma_{n+1}\sigma_n^2
&=&\sigma_n^{-2}\beta_2\sigma_n\beta_1\sigma_n\sigma_{n+1}\\
\end{array}\end{equation}
A similar decomposition using negative stabilization and destabilization is obtained by inverting~$\sigma_n$, $\sigma_{n+1}$,
and the order of the moves.
\end{proof}

\subsection{Jones' conjecture}

For every~$n\geqslant2$ we denote by $\exp$ the homomorphism from~$B_n$ to~$\mathbb Z$
defined by~$\exp(\sigma_1)=1$. In other words, $\exp(\beta)$ is the algebraic number of crossings
of a braid diagram representing~$\beta$.

In paper~\cite{Jon} V.Jones addressed the following question: is $\exp(\beta)$
taken for a minimal braid $\beta$ an invariant of the corresponding link? By a minimal
braid we mean one that has minimal possible number of strands among
all braids representing the same oriented link.

A positive answer to this question, which has become known as Jones' conjecture,
was given for a few infinite series of knots, see~\cite{kawa} and references therein.
The following, stronger, statement was formulated as a conjecture in~\cite{kawa} and~\cite{tra} 
and called generalized Jones' conjecture.

\begin{theo}\label{jonesconj}
Let braids $\beta_1\in B_m$ and~$\beta_2\in B_n$ close up to equivalent oriented links,
with~$\beta_1$ having minimal possible number of strands for this link type. Then we have
$$|\exp(\beta_2)-\exp(\beta_1)|\leqslant n-m.$$
In particular, if~$n=m$, then we have~$\exp(\beta_1)=\exp(\beta_2)$.
\end{theo}

\begin{proof}
By Theorem~\ref{mainbraidtheo} there exists a natural~$p$ and a braid~$\beta\in B_p$ such that
the classes~$\bm(\beta)$ and~$\bm(\beta_1)$ can be connected by a sequence of positive stabilizations and destabilizations,
and the classes~$\bm(\beta)$ and~$\bm(\beta_2)$ by that of negative ones.

The difference between the algebraic number of crossings and the number of strands is preserved under
positive (de)stabilizations, hence we have
$$\exp(\beta_1)-m=\exp(\beta)-p.$$
Similarly, the sum of the algebraic number of crossings and the number of strands is preserved
under negative (de)stabilizations, which implies
$$\exp(\beta_2)+n=\exp(\beta)+p.$$
Therefore, we have
$$p=\frac12(m+n+\exp(\beta_2)-\exp(\beta_1)).$$
From~$p\geqslant m$ we get
$$\exp(\beta_2)-\exp(\beta_1)\geqslant m-n.$$

Similarly, by choosing a braid~$\beta$ so that the class~$\bm(\beta)$ be connected by a sequence of negative
stabilizations and destabilizations with~$\bm(\beta_1)$, and by that of positive ones with~$\bm(\beta_2)$,
we get
$$\exp(\beta_2)-\exp(\beta_1)\leqslant n-m.$$\end{proof}

It is observed in~\cite{tra} that the generalized Jones' conjecture that we have just proved is equivalent to the following
statement.

\begin{theo}
Let planar diagrams $D_1$ and~$D_2$ represent equivalent oriented links and have~$m$ and~$n$
Seifert circles, respectively, where $m$ is the smallest possible number of Seifert circles for
diagrams representing the same link type. Then we have
$$|w(D_2)-w(D_1)|\leqslant n-m,$$
where $w(D)$ stays for the writhe of~$D$.
\end{theo}

\begin{rem}
Our method of proving generalized Jones' conjecture is partially similar
to the one suggested by K.Kawamuro in 2008, namely,
in the reduction of the statement to ``the commutativity principle''.
However, the formulation of the principle involved then the conjugacy classes
of braids instead of Birman--Menasco classes, which made it invalid.
\end{rem}

\subsection{Transversal links}
\begin{defi}
By \emph{a transversal link} one call an oriented link such that the restriction of the standard
contact form~$\omega=x\,dy+dz$ takes positive values at all non-zero tangent vectors to the link that agree
with the orientation.

Two transversal links are said to be \emph{transversally equivalent} if they are isotopic within
the class of transversal links.
\end{defi}

With every braid~$\beta\in B_n$ one associates canonically a transversal link defined up to transversal equivalence
as follows. The closure~$L$ of the braid~$\beta$ can be positioned in the space so that it intersects each page~$\page_t$
(see notation in Subsection~\ref{booksection}) transversely $n$ times, and, moreover, the angular coordinate~$\theta$
is locally increasing on~$L$ in the direction defined by the link orientation.
For small enough~$\varepsilon>0$ the $1$-form~$\varphi_\varepsilon^*\omega$ is a small disturbation of~$\rho\,d\theta=x\,dy-y\,dx$, 
where~$\varphi_\varepsilon$ is the diffeomorphism defined by~$(x,y,z)\mapsto(2x,y,\varepsilon z-xy)$,
hence the link~$\varphi_\varepsilon(L)$ is transversal. The transversal link constructed in this way will denoted by~$T_\beta$.

\begin{theo}\label{transversallinksviabraids}
Any transversal link is transversally equivalent to a link of the form~$T_\beta$.

For two braids~$\beta_1$ and~$\beta_2$ the links~$T_{\beta_1}$ and~$T_{\beta_2}$ are transversally equivalent
if and only if the braids~$\beta_1$ and~$\beta_2$ can be obtained from each other by a sequence of conjugations
and positive stabilizations and destabilizations.
\end{theo}

The first part of this theorem is due to D.Bennequin~\cite{Ben}. The second was proved by S.Orevkov jointly with
V.Shevchishin~\cite{OSh} and, independently, by N.Wrinkle~\cite{Wr}.

\emph{The self-linking number} $\sl(L)$ of a transversal link~$L$ is defined as the difference~$\exp(\beta)-n$, 
where~$\beta\in B_n$ is an arbitrary braid for which $L$ is transversally equivalent to~$T_\beta$.
As one sees from Theorem~\ref{transversallinksviabraids} self-linking number is well defined
because it does not change under positive stabilizations.

For a braid~$\beta$ we denote by~$\trans(\beta)$ the set of all braids~$\beta'$ for which
the links~$T_{\beta'}$ and~$T_\beta$ are transversally equivalent. We call~$\trans(\beta)$
\emph{the transversal class} of the braid~$\beta$. As follows from~\eqref{exchangeviastabilization}, every
Birman--Menasco class is a subset of a transversal class.

In the language of transversal links Theorem~\ref{mainbraidtheo} means that, for any two braids~$\beta_1$, $\beta_2$
that close up to equivalent oriented links, there is a braid~$\beta$ such that we have $\trans(\beta)=\trans(\beta_1)$,
$\trans(\beta^{-1})=\trans(\beta_2^{-1})$.

\begin{theo}
Let Birman--Menasco classes~$\mathcal B_1$, $\mathcal B_2$ be contained in the same transversal class,
and let~$\mathcal B_1$ admit a negative destabilization~$\mathcal B_1\mapsto\mathcal B'_1$. Then~$\mathcal B_2$
also admits a negative destabilization $\mathcal B_2\mapsto\mathcal B'_2$ with~$\mathcal B'_1$ and~$\mathcal B'_2$
being contained in the same transversal class.
\end{theo}

\begin{proof}
Let~$R_2$ and $R_1'$ be oriented rectangular diagrams such that
$\beta_{R_2}\in\mathcal B_2$, $\beta_{R_1'}\in\mathcal B_1'$.
Then one can proceed from~$R_2$ to~$R_1'$ by elementary moves including exactly one~\IIleft-destabilization,
and not including~\IIleft-stabilizations.

From Lemmas~\ref{stabilizations2start} and~\ref{stabilizationatanyvertex} we conclude that
one can proceed from~$R_2$ to~$R_1'$ by a sequence of elementary moves in which all \IIright-stabilizations go first,
and all \IIright-destabilizations go last. Since type \IIright\ stabilizations and destabilizations don't affect the corresponding braid,
we may assume without loss of generality that they are not present in the sequence.

By applying Lemmas~\ref{stabilizations2start} and~\ref{stabilizationatanyvertex} again we can move
the single \IIleft-destabilization to the very end of the sequence. The obtained sequence of elementary moves
with the last destabilization removed will contain only type~I stabilizations and destabilizations. Let
$R$ be the diagram this sequence arrive at. By construction it is Legendrian equivalent
to~$R_2$, and $R\mapsto R_1'$ is a type \IIleft\ destabilization.

It follows from Theorem~\ref{singlesimplification} that the diagram $R_2$ admits a type~\IIleft\ simplification $R_2\mapsto R_2'$
with the diagrams~$R_2'$ and~$R_1'$ being Legendrian equivalent. One can see that this implies $\trans(\beta_{R_2'})=\trans(\beta_{R_1'})$, and
that $\bm(R_2)\mapsto\bm(R_2')$ is a negative destabilization.
\end{proof}

Finally, we notice that Theorem~\ref{jonesconj} gives a positive answer to Question~2 of~\cite{NG}:

\begin{coro}
Let a braid $\beta$ have the smallest possible number of strands among all braids closing up
to a given oriented link type. Then the transversal link~$T_\beta$ has the largest possible self-linking number
among all transversal links having the same topological type.
\end{coro}

\begin{proof}
Let $\beta\in B_m$, and let $\beta'\in B_n$ be any other braid closing up to the same oriented
link type. By Theorem~\ref{transversallinksviabraids} we have
$$\sl(T_{\beta'})=\exp(\beta')-n\leqslant\exp(\beta)-m=\sl(T_\beta).$$
\end{proof}

\vskip1cm
\noindent{\sc
V.A.Steklov Mathematical Institute\\
of Russian Academy of Sciences\\
8 Gubkina str.\\
Moscow 119991, Russia}\\
{\tt dynnikov@mech.math.msu.su}\\[3mm]
{\sc Dept.\ of Mechanics \& Mathematics\\
Moscow State University\\
1 Leninskije gory\\
Moscow 119991, Russia}\\
{\tt 0x00002a@gmail.com}

\end{document}